\documentclass[10pt,a4paper,oneside]{amsart}

\usepackage{amsmath} 
\usepackage{amssymb} 
\usepackage{amsthm} 
\usepackage[english]{babel} 
\usepackage{subcaption}
\usepackage[nodayofweek]{datetime}
\usepackage[shortlabels]{enumitem} 
\usepackage[T1]{fontenc} 
\usepackage[a4paper]{geometry} 
\usepackage[utf8]{inputenc} 
\usepackage{ifthen} 
\usepackage{mathabx} 
\usepackage{mathtools} 
\usepackage[dvipsnames]{xcolor} 
\usepackage[pdftex,  colorlinks=true,pagebackref=true]{hyperref} 
    \hypersetup{urlcolor=RoyalBlue, linkcolor=RoyalBlue,  citecolor=black}
\usepackage{setspace} 
\usepackage{tikz-cd} 
\usepackage{xfrac} 
\usepackage{mathrsfs} 
\usepackage[capitalize]{cleveref} 
\usepackage{cprotect}
\usepackage{multicol,multirow}

\setenumerate{label=\textup{(\roman*)}}

\renewcommand*{\backref}[1]{}
\renewcommand*{\backrefalt}[4]
{
    \ifcase #1
        No citation in the text.
    \or
        Cited on Page #2.
    \else
        Cited on Pages #2.
    \fi
}

\newcounter{commentcounter}

\makeatletter
\def\@tocline#1#2#3#4#5#6#7{\relax
  \ifnum #1>\c@tocdepth 
  \else
    \par \addpenalty\@secpenalty\addvspace{#2}%
    \begingroup \hyphenpenalty\@M
    \@ifempty{#4}{%
      \@tempdima\csname r@tocindent\number#1\endcsname\relax
    }{%
      \@tempdima#4\relax
    }%
    \parindent\z@ \leftskip#3\relax \advance\leftskip\@tempdima\relax
    \rightskip\@pnumwidth plus4em \parfillskip-\@pnumwidth
    #5\leavevmode\hskip-\@tempdima
      \ifcase #1
       \or\or \hskip 1em \or \hskip 2em \else \hskip 3em \fi%
      #6\nobreak\relax
    \dotfill\hbox to\@pnumwidth{\@tocpagenum{#7}}\par
    \nobreak
    \endgroup
  \fi}
\makeatother

\newtheorem{thm}{Theorem}[section]
\newtheorem{lemma}[thm]{Lemma}
\newtheorem{corollary}[thm]{Corollary}
\newtheorem{prop}[thm]{Proposition}

\newtheorem{claim}[thm]{Claim}

\AddToHook{env/lemma/begin}{\crefalias{thm}{lemma}}
\AddToHook{env/corollary/begin}{\crefalias{thm}{corollary}}
\AddToHook{env/prop/begin}{\crefalias{thm}{prop}}
\AddToHook{env/conjecture/begin}{\crefalias{thm}{conjecture}}
\AddToHook{env/claim/begin}{\crefalias{thm}{claim}}

\crefname{thm}{Theorem}{Theorems}
\crefname{lemma}{Lemma}{lemmas}
\crefname{prop}{Proposition}{Propositions}
\crefname{conjecture}{Conjecture}{Conjectures}
\crefname{claim}{Claim}{Claims}

\newtheorem{thmx}{Theorem}

\AddToHook{env/corx/begin}{\crefalias{corx}{thmx}}

\crefname{thmx}{Theorem}{Theorems}
\crefname{corx}{Corollary}{Corollaries}

\theoremstyle{definition}
\newtheorem{defn}[thm]{Definition}
\newtheorem{remark}[thm]{Remark}
\newtheorem{remarks}[thm]{Remarks}

\newtheorem{quest}[thm]{Question}
\newtheorem{example}[thm]{Example}
\newtheorem{examplex}[thmx]{Example}

\AddToHook{env/defn/begin}{\crefalias{thm}{defn}}
\AddToHook{env/remark/begin}{\crefalias{thm}{remark}}
\AddToHook{env/remarks/begin}{\crefalias{thm}{remarks}}
\AddToHook{env/quest/begin}{\crefalias{thm}{quest}}
\AddToHook{env/example/begin}{\crefalias{thm}{example}}

\crefname{defn}{Definition}{Definitions}
\crefname{remark}{Remark}{Remarks}
\crefname{remarks}{Remarks}{Remarks}
\crefname{quest}{Question}{Questions}
\crefname{example}{Example}{Examples}

\AddToHook{env/examplex/begin}{\crefalias{thmx}{examplex}}
\crefname{examplex}{Example}{Examples}

    \newtheoremstyle{TheoremNum}
        {8.0pt plus 2.0pt minus 4.0pt}{8.0pt plus 2.0pt minus 4.0pt} 
        {\itshape} 
        {-0.15cm} 
        {\bfseries} 
        {.} 
        { }  
        {\thmname{#1}\thmnote{ \bfseries #3}}
    \theoremstyle{TheoremNum}
    \newtheorem{duplicate}{}

\newcommand*{\claimproofname}{My proof}
\newenvironment{claimproof}[1][\claimproofname]{\begin{proof}[#1]}{\end{proof}}

\DeclareMathOperator{\aster}{\textnormal{\LARGE{\textasteriskcentered}}}

\usepackage{stackengine}

\newcommand\sqsubsetdot{\mathrel{{\ensurestackMath{%
  \stackinset{c}{0.3ex}{c}{-0.0ex}{\cdot}{\sqsubset}}}}}

\newcommand\nsqsubsetdot{\mathrel{{\ensurestackMath{%
  \stackinset{c}{0.3ex}{c}{-0.05ex}{\cdot}{\nsqsubset}}}}}

\DeclareMathOperator{\Aut}{\mathrm{Aut}}
\DeclareMathOperator{\Out}{\mathrm{Out}}
\DeclareMathOperator{\Inn}{\mathrm{Inn}}

\DeclareMathOperator{\im}{\mathrm{Im}}
\DeclareMathOperator{\Fix}{\mathrm{Fix}}

\DeclareMathOperator{\End}{\mathrm{End}}

\DeclareMathOperator{\id}{id}
\DeclareMathOperator{\Stab}{Stab}
\DeclareMathOperator{\Gal}{Gal}
\DeclareMathOperator{\Fitt}{Fitt}
\DeclareMathOperator{\bass}{Bass}
\DeclareMathOperator{\diag}{diag}

\newcommand{\cala}{{\mathcal{A}}}
\newcommand{\calb}{{\mathcal{B}}}
\newcommand{\calc}{{\mathcal{C}}}
\newcommand{\cald}{{\mathcal{D}}}

\newcommand{\calg}{{\mathcal{G}}}
\newcommand{\calh}{{\mathcal{H}}}

\newcommand{\calo}{{\mathcal{O}}}

\newcommand{\cals}{{\mathcal{S}}}
\newcommand{\calt}{{\mathcal{T}}}
\newcommand{\calu}{{\mathcal{U}}}

\newcommand{\calw}{{\mathcal{W}}}

\newcommand{\scra}{{\mathscr{A}}}

\newcommand{\GL}{\mathrm{GL}}

\newcommand{\OO}{\mathrm{O}}

\newcommand{\AGL}{\mathrm{AGL}}

\newcommand{\BS}{\mathrm{BS}}
\newcommand{\LM}{\mathrm{LM}}

\newcommand{\CAT}{\mathrm{CAT}}

\newcommand{\EE}{\mathbb{E}}

\newcommand{\MM}{\mathbf{M}}

\newcommand{\NN}{\mathbb{N}}
\newcommand{\ZZ}{\mathbb{Z}}

\newcommand{\RR}{\mathbb{R}}
\newcommand{\QQ}{\mathbb{Q}}

\usepackage{tikz}
\usetikzlibrary{arrows,quotes,calc,positioning}
\tikzstyle{blackNode}=[fill=black, draw=black, shape=circle]

\usetikzlibrary{decorations.markings}
\tikzset{middlearrow/.style n args={5}{
		decoration={
			markings,
			mark=at position #1 with {\arrow[rotate=#2]{#3},\node[transform shape,#5] {#4};}},postaction={decorate}},
	middlearrow/.default={.5}{0}{>}{}{below}	}

\newcommand{\ind}[2]{#1{\upharpoonright}_{\!#2}}

\title[Outer automorphisms of Leary--Minasyan groups]{Outer automorphisms of HNN extensions\\ of abelian groups}
\author{Naomi Andrew}
\address{Naomi Andrew, Universit\'e Paris-Saclay, CNRS, Laboratoire de math\'ematiques d’Orsay, 91405, Orsay, France}
\email{naomi.andrew@universite-paris-saclay.fr}
\urladdr{https://naomigandrew.wordpress.com/}
\author{Sam Hughes}
\address{Sam Hughes, Max-Planck-Institut für Mathematik, Vivatsgasse 7, 53111 Bonn, Germany}
\email{sam.hughes.maths@gmail.com; hughess@mpim-bonn.mpg.de}
\urladdr{https://samhughesmaths.github.io}
\author{Motiejus Valiunas}
\address{Motiejus Valiunas, Instytut Matematyczny, Uniwersytet Wroc{\l}awski, plac Grunwaldzki 2, 50-384 Wroc{\l}aw, Poland}
\email{motiejus.valiunas@math.uni.wroc.pl}
\urladdr{https://valiunas.wordpress.com/}
\date{\today}

\begin{document}
\begin{abstract}
    Let $G$ be a commensurating HNN extension of a finitely generated free abelian group.  We characterise when $\mathrm{Out}(G)$ is finitely generated and derive structural information about the group.  We give soluble and non-soluble examples of commensurating HNN extensions of finitely generated free abelian groups with outer automorphism group not finitely generated.    
\end{abstract}
\dedicatory{To Ian Leary on the occasion of his sixtieth (sixty-second) birthday}
\maketitle

\section*{Introduction}
Every countable group is the outer automorphism group of a finitely generated group \cite{BumaginWise2005}.  In fact every countable group is the automorphism group of an HNN-extension of a triangle group \cite{Logan2019}.  On the other hand for many well-studied classes of groups the structure of $\Out(G)$ is well understood.  For a free abelian group $\ZZ^n$ it is isomorphic to $\GL_n(\ZZ)$.  Vastly more generally, Baues--Grunewald \cite{BauesGrunewald2006} showed that the outer automorphism group of any virtually polycyclic group is an arithmetic group. The study of outer automorphisms of free and surface groups is a central theme in geometric group theory, often drawing parallels with the study of arithmetic groups; see for example the surveys and books \cite{Bestvina2002, Vogtmann2002, BridsonVogtmann2006, FarbMargalit2011, Margalit2019, AndrewHenselHughesWade2026}.  In each case the outer automorphism group admits a finite index subgroup with finite Eilenberg--Maclane space \cite{BorelSerre1973, CullerVogtmann1986, Harvey1979, Harvey1981}. In many cases the growth rates of elements are either exponential or polynomial and fall into finitely many growth types, see \cite{Nielsen1944,Thurston1988,Levitt2009}; similar results have been proven for outer automorphism groups of hyperbolic groups \cite{Coulon2025} and virtually compact special groups \cite{Fioravanti2026}.

The one-relator groups $\BS(r,s)=\langle a,t\ |\ (a^r)^t=a^s\rangle$ were introduced by Baumslag and Solitar in \cite{BaumslagSolitar1962} to provide simple examples of non-Hopfian groups.  
A complete characterisation of their Hopficity was provided by Collins--Levin \cite{CollinsLevin1983}, namely, $\BS(r,s)$ is Hopfian if and only if either $r=\pm 1$, or $s=\pm1$, or $r$ and $s$ have the same prime divisors.  One can also consider so called \emph{$m$-generalised Baumslag--Solitar groups} (GBS$_m$ groups) which are given by graphs of groups with all edge and vertex groups isomorphic to $\ZZ^m$.  For $m = 1$, these are simply called \emph{generalised Baumslag--Solitar groups} (GBS groups).

Since their inception Baumslag--Solitar groups and GBS groups have received considerable attention from numerous perspectives: quasi-isometric rigidity \cite{FarbMosher1998,FarbMosher1999,Whyte2001,Margolis2021}, BNS invariants \cite{CashenLevitt2016,SgobbiWong2023,AbbottBalasubramanyaRasmussen2023}, measure equivalence rigidity \cite{Kida2014,HoudayerRaum2015}, and lattices in the universal cover of the natural presentation complex \cite{Forester2024,Verma2025}; as well as many other topics \cite{Levitt2015_quosubBS,Levitt2015_genBSfi,Button2025,LopezdeGamizZearraShepherd2025,BanerjeeSchreve2025,CarderiGaboriauLeMaitreStalder2025,Wykowski2025,CohenWykowski2026}.  

The isomorphism problem for GBS$_m$ groups has also received considerable attention \cite{Levitt2007,ClayForester2008,ClayForester2009,Dudkin2017,CasalsRuizKazachkovZakharov2021,Valiunas2022,Wang2025,AscariCasalsKazachkov2025i,AscariCasalsKazachkov2025ii,AscariCasalsKazachkov2025iii,AscariCasalsKazachkov2025iv}, but remains unsolved in general, even in the case $m = 1$.  Returning to Baumslag--Solitar groups, a shadow of the difficulty of the isomorphism problem for GBS groups is already visible in trying to understand $\Out(\BS(r,s))$.  For example, Collins--Levin showed that $\Out(\BS(2,4))$ is not finitely generated \cite{CollinsLevin1983}, whereas $\Out(\BS(2,3))$ is finite \cite{Collins1978,GilbertHowieMetaftsisRaptis2000}.

A special class of GBS$_m$ groups are given by the so-called \emph{Leary--Minasyan groups}, introduced in \cite{LearyMinasyan2021} to give examples of $\CAT(0)$ but not biautomatic groups.  The $\CAT(0)$ examples were studied in a number of different contexts \cite{HuangPrytula2020,Hughes2022,Hughes2025,Button2025,CavallucciSambusetti2025}.  Further examples were given in \cite{Hughes2022,HughesValiunas2024,HughesValiunas2025}.  The non-biautomaticity of Leary--Minasyan groups has been used to give information about the biautomaticity of GBS groups \cite{ShepherdValiunas2025} and the structure of biautomatic groups \cite{Valiunas2023}.  

A Leary--Minasyan group $\LM(A,L)$ is a commensurating HNN extension of $\ZZ^m$, i.e.\ it is an HNN extension of $\ZZ^m$ where the edge group includes with finite index.  Equivalently, a Leary--Minasyan group is defined (although not uniquely, see \cite{Valiunas2022}) by a matrix $A\in\GL_m(\QQ)$ and a finite index subgroup $L\leqslant \ZZ^m$ such that $AL\leqslant \ZZ^m$.  From here we identify $\ZZ^m$ with $\langle x_1,\dots, x_m\rangle$ and define
\begin{equation}\label{eqn:LMAL}
\LM(A,L)=\langle x_1,\dots,x_m, t \mid [x_i,x_j] \text{ for } 1\leq i<j\leq m,\ tx^vt^{-1}=x^{Av}\ \forall v\in L\rangle, 
\end{equation}
where $x^v \coloneqq x_1^{v_1} \cdots x_m^{v_m}$ for $v = (v_1,\ldots,v_m) \in \ZZ^m$.  In the remainder of the paper, we abuse notation by identifying $x^v$ with $v$ and viewing $v$ and $Av$ as elements of $\LM(A,L)$.
Note that $\BS(r,s)\cong\LM(\frac{s}{r},r\ZZ)$.

\subsection*{Results}
We have split our main theorem depending on whether $\LM(A,L)$ is metabelian or not.  By work of the third author \cite[\S2.4]{Valiunas2022}, this can be completely determined from the data $(A,L)$.  Indeed, $\LM(A,L)$ is polycyclic if and only if $\ZZ^m=L=AL$; $\LM(A,L)$ is metabelian but not polycyclic if exactly one of $\ZZ^m=L$ or $\ZZ^m=AL$ holds; otherwise $\LM(A,L)$ is not metabelian.  We first state our result on finite generation of outer automorphism groups of non-metabelian Leary--Minasyan groups. 

\begin{duplicate}[\Cref{thmx:not_metabelian}]
    Let $G=\LM(A,L)$ and suppose $G$ is not metabelian.  $\Out(G)$ is finitely generated if and only if $L$ and $AL$ do not properly contain each other.
\end{duplicate}

It remains a long-standing open problem to decide if every virtually $\CAT(0)$ group is in fact a $\CAT(0)$ group.  In \Cref{thm:CATs}, we show that if a group $\Gamma$ contains a $\CAT(0)$ centreless Leary--Minasyan group as a finite index normal subgroup, then $\Gamma$ is also $\CAT(0)$.

We now describe our results in the case when the group $\LM(A,L)$ is metabelian.  This is equivalent to the group being soluble or amenable \cite[\S2.4]{Valiunas2022}, and is exactly the case where $\LM(A,L)$ splits as an ascending HNN extension of $\ZZ^m$.  In this case we obtain a semi-direct product decomposition $\left(\bigcup_{s\in\ZZ} A^s L\right)\rtimes_A\ZZ$, where the normal subgroup is a not necessarily finitely generated abelian group of rank $m$.

\begin{defn}[Essentially Jordan gap free]
    Let $A \in \GL_m(\QQ)$ and let $\lambda$ be an eigenvalue of $A$. We say an integer $d \geq 1$ is a \emph{$\lambda$-Jordan gap} for $A$ if the Jordan normal form for $A$ has no $\lambda$-Jordan blocks of size $d$, but has at least one $\lambda$-Jordan block of size $d+1$. We say $A$ is \emph{essentially Jordan gap free} if all $\lambda$-Jordan gaps for $A$ occur when $\lambda$ is an algebraic unit.
\end{defn}

\begin{duplicate}[\Cref{thmx:metabelian}]
     Let $G=\LM(A,L)$ and suppose $G$ is metabelian.  $\Out(G)$ is finitely generated if and only if $A$ is essentially Jordan gap free.
\end{duplicate}

Of course in the case of $\BS(1,n)$, a metabelian Baumslag--Solitar group, $A  = \begin{pmatrix} n \end{pmatrix}$ consists of a single Jordan block of size one, and so $\Out(G)$ is always finitely generated, as proved by Collins \cite[Proposition A]{Collins1978}. The lack of finite generation here is a higher rank phenomenon.

\begin{remark}
    The case when $\LM(A,L)$ is polycyclic is actually already covered by a deep theorem of Baues--Grunewald \cite{BauesGrunewald2006}.  In this case $\Out(G)$ is an arithmetic group and is always finitely generated \cite[Theorem~6.5]{BorelHarishChandra1962}.  Even though we do not directly prove finite generation of $\Out(G)$ for polycyclic $G = \LM(A,L)$, instead relying on the result of Baues--Grunewald, we note that in that case all the eigenvalues of $A$ are algebraic units, and so Theorem~\ref{thmx:metabelian} still holds.
\end{remark}

Our \Cref{thmx:metabelian} shows that Baues--Grunewald's work cannot hope to extend from polycyclic to solvable groups. Elia Fioravanti asked us for an example of a finitely presented soluble group $G$ with $\Out(G)$ not finitely generated.  Here we present an explicit example of a type $\mathsf{F}$ (hence finitely presented) metabelian (hence soluble) group $G$ with $\Out(G)$ not finitely generated.  Note that we provide more details for this example near the end of \Cref{sec:metabelian}.

\begin{duplicate}[\Cref{ex:intro}]
Let $G=\langle a,b,t \mid [a,b],\ a^t=a^2b,\ b^t=b^2\rangle.$  Note that this is the Leary--Minasyan group $\LM(A,L)$ where $A = \begin{pmatrix}
    2 & 1 \\
    0 & 2
\end{pmatrix}$ and $L=\ZZ^2$.  The group $G$ is metabelian but not polycyclic and is of type $\mathsf{F}$ since it is an ascending HNN extension of $\ZZ^2$.  Moreover, $\Out(G)$ is not finitely generated.  In fact, $\Out(G)$ is isomorphic to $\ZZ[\frac{1}{2}] \times \ZZ/2\ZZ$.
\end{duplicate}

\begin{quest}
    If $\LM(A,L)$ is metabelian and $\Out(\LM(A,L)$ is finitely generated, what other finiteness properties must it satisfy? (For example, for each $n$, does there exist a metabelian Leary--Minasyan group $\LM(A,L)$ such that $\Out(\LM(A,L))$ is type $\mathsf{F}_n$ but not $\mathsf{F}_{n+1}$?)
\end{quest}

\subsection*{Overview of the methods}

In the case of a non-ascending HNN extension $G = \LM(A,L)$ we use the \emph{deformation space} of $G$-trees, introduced by Forester \cite{Forester2002Deform}, and further studied by Clay \cite{Clay05}.  In this case, the deformation space is unique (Proposition~\ref{prop:D-AutG-inv}), and therefore admits an $\Out(G)$-action.  If $L$ and $AL$ do not properly contain each other (i.e., they are either equal or incomparable), then the deformation space turns out to contain a unique reduced tree (\Cref{prop:rigid_iff_no_proper_containment}); in that case, we can show that $\Out(G)$ is finitely generated and describe it explicitly, up to taking a subgroup of index at most two.

In the case when one of $L$ and $AL$ properly contains the other ($AL<L$ or $L<AL$), we consider a lattice of lattices to form a deformation retract of the deformation space.  More precisely, following the work of Clay on the deformation spaces of Baumslag--Solitar groups \cite{Clay09}, we describe a simplicial structure on the deformation space.  In our case, the space consists of copies of a complex with vertices $G$-trees of type $\calt(\Lambda)$, where $\Lambda$ is a lattice in $\QQ^m$ containing $L$, that are connected by joining them to $G$-trees of type $\cals$; see Figure~\ref{fig:gog}. We then show the following:

\begin{duplicate}[\Cref{thmx:deformation_retract}]
    If $G = \LM(A,L)$ with either $L < AL < \ZZ^m$ or $AL < L < \ZZ^m$, then the deformation space of reduced $G$-trees $\Out(G)$-equivariantly deformation retracts to a simplicial tree.
\end{duplicate}

Such a description, by considering the vertex stabilisers under the resulting $\Out(G)$-action on a tree, allows us to express $\Out(G)$ as the fundamental group of an (infinite) graph of groups (Theorem~\ref{thm:OutG-semidir-nonfg}), from which we can deduce a very explicit structure to prove that it is not finitely generated.  

In the case that the Leary--Minasyan group is an ascending HNN extension we study the problem via commutative algebra and relate finite generation to the structure of the Jordan blocks of the defining endomorphism.  The proof is technical and involves a careful analysis of the $\ZZ$-module structure on a finite rank abelian group $K=\bigcup_{s\in\ZZ}A^sL$.  A key step is showing that $\Out(G)$, where $G = \LM(A,L)$, is finitely generated if and only if $C=C_{\GL(K)}(A)$ is finitely generated (\Cref{prop:metab-Out}).  More precisely we show that $\Aut(G)\cong C\ltimes K$ and $\Inn(G)\cong \langle A\rangle \ltimes (I-A)K$.  So $\Out(G)\cong (C/\langle A\rangle)\ltimes F$, where $F$ is a finitely generated abelian group.   The remainder of the work is then to describe the structure of $C$.  We give a more detailed overview of the proof strategy and a worked example in \Cref{sec:metabelian}.

\begin{remark}
Whilst the authors were preparing this paper Der\'e and Pengitore \cite{DerePengitore2025} (building on \cite{ArapuraNori1999} and generalising \cite{BauesGrunewald2006}) announced a number of structural results about $\Aut(G)$ when $G$ is a virtually torsion-free soluble group of finite abelian ranks whose fitting subgroup is $S$-arithmetic.  Most relevant to our work is \cite[Theorem~1.6]{DerePengitore2025} which gives sufficient conditions for $\Aut(G)$ of such a soluble group $G$ to be an $S$-arithmetic group.  It seems that our results on finite generation here cannot be derived from their work for two reasons: firstly, they allow $S$-arithmetic lattices in non-reductive algebraic groups and these lattices are not necessarily finitely generated (e.g. $\ZZ[1/2]$); secondly, the fitting subgroup of a metabelian Leary--Minasyan group may not be $S$-arithmetic. For example, $G=\LM(A,L)$ with $A=\diag(2,3)$ and $L=\ZZ^2$ has $\Fitt(G)\cong \ZZ[1/2]\times \ZZ[1/3]$, which is not $S$-arithmetic for any $S$.  Note here that $G\cong\langle a,b,t\mid [a,b],\ a^t=a^2,\ b^t=b^3\rangle$.
\end{remark}

\subsection*{Structure of the paper}
In \Cref{sec:prelims} we recount the necessary background on Bass--Serre theory, irreducible actions on trees, and deformation spaces of trees.

In \Cref{sec:non_metab_intro} we introduce a deformation space of trees for a Leary--Minasyan group $G$.  In the case where the Bass-Serre tree for the standard HNN splitting of a Leary--Minasyan group is rigid we provide detailed information about finiteness properties of $\Aut(G)$ and $\Out(G)$, see \Cref{thm:nonMet_VF_Foo_fin}.  We then prove that if $G$ is a group containing a $\CAT(0)$ Leary--Minasyan group with trivial centre as a finite index subgroup, then $G$ is $\CAT(0)$, see \Cref{thm:CATs}.  The remainder of the section studies the case of standard splitting for a Leary--Minasyan group $G$ being non-rigid.  More precisely, we analyse the stabilisers of $\Out(G)$ acting on certain $G$-trees.

In \Cref{sec:deform} we prove \Cref{thmx:not_metabelian}. To do this we analyse a lattice of lattices to construct a deformation retract of the deformation space of trees onto a tree. That is, we characterise the reduced trees in (a subspace of) the deformation space of trees in terms of their vertex groups, identified with certain lattices in $\QQ^n$, and use this characterisation to construct an explicit deformation retract. These arguments generalise those of \cite{Clay09} in the Baumslag--Solitar group case.  We then analyse the tree obtained from the deformation retraction in and give a detailed structural description of $\Out(G)$ in the non-rigid case, see \Cref{thm:OutG-semidir-nonfg}.  

In \Cref{sec:metabelian} we prove \Cref{thmx:metabelian} establishing the characterisation in the metabelian case.  We also provide the details for \Cref{ex:intro} and summarise our results in the case of Baumslag--Solitar groups, see \Cref{cor.BS}. We emphasise that with the exception of the final corollary, the results in \Cref{sec:metabelian} are completely independent of those in the earlier sections.

\subsection*{AI declaration} The authors did not use AI in the preparation of this article.

\tableofcontents

\section*{Acknowledgements}
All three authors would like to thank Ian Leary for the constant support, enthusiasm, kind natured humour, and mentorship he has shown us over the years.  We are all indebted to him.

The authors thank Adrien Abgrall for useful discussions about deformation spaces, Mark Pengitore for answering their questions about \cite{DerePengitore2025} and Elia Fioravanti for encouragement.  

{\small NA acknowledges support from the European Research Council (ERC) under Grant 101040507 Artin-Out-ME-OA. NA was also supported by the Royal Society of Great Britain.  NA and SH were supported by ERC Grant no. 850930 Fibring. SH was supported by a Humboldt Research Fellowship at Universit\"at Bonn and by the Deutsche Forschungsgemeinschaft (EXC-2047/1 - 390685813).    MV was partially supported by the National Science Centre (Poland) grants No.\ 2022/47/D/ST1/00779 and 2025/58/E/ST1/00143. The authors are grateful for the hospitality and support they received during visits to Bonn (NA and MV), Orsay (SH), Oxford (MV), Wroc\l aw (NA and SH), and at the Isaac Newton Institute for Mathematical Sciences during the programme Operators, Graphs, Groups (EPSRC grant EP/Z000580/1s).}

\section{Preliminaries}\label{sec:prelims}

Let $G$ be a group.  We denote a subgroup of $G$ by $H\leqslant G$ and a proper subgroup of $G$ by $H<G$.  We use the same convention for subsets and partial orders.

\subsection{Bass--Serre theory}
In this section we establish our notation and definitions for the Bass--Serre theory we will use.  For a broad introduction the reader is referred to \cite{Serre2002}.

\begin{defn}[Graph]
A \emph{graph} $X$ consists of a tuple of sets $(V(X), E(X))$, where $V(X)$ is called the \emph{vertex set}, and $E(X)$ the \emph{edge set}, together with a pair of functions $\iota\colon E(X) \to V(X)$ and $\tau\colon E(X) \to V(X)$, and a fixed-point-free involution $\bar{\ }\  \colon E(X) \to E(X)$, such that for every edge $e \in E(X)$ we have that $\iota(e) = \tau(\bar{e})$.
\end{defn}

\begin{defn}[Graph of groups]
A \emph{graph of groups} $\calg$ is a triple $(X, \calg_{\bullet}, \iota_{\bullet})$ where $X$ is a graph, $\calg_{\bullet}$ is an assignment of a group $\calg_v$ to every vertex $v \in V(X)$ and a group $\calg_e$ to every edge $ e\in E(X)$ satisfying $\calg_e = \calg_{\bar{e}}$, and $\iota_{\bullet}$ determines monomorphisms $\iota_e \colon \calg_e \rightarrowtail \calg_{\iota(e)}$ for all edges $e \in E(X)$. We also define the monomorphism $\tau_e \colon \calg_e \rightarrowtail \calg_{\tau(e)}$, to be $\tau_e = \iota_{\bar{e}}$ for each $e \in E(X)$. Throughout the paper we will assume that the underlying graph $X$ is finite.
\end{defn}

\begin{defn}[Bass group and fundamental group of a graph of groups]
Let $\mathcal G$ be a graph of groups and denote by $F_{E(X)}$ the free group with basis $E(X)$. The \emph{Bass group}, $\bass(\calg)$ is defined to be
\[ \bass(\calg)  =  \left(\left(\bigast_{v \in V(X)} \calg_v \right) \aster F_{E(X)}  \right) / \langle\!\langle e\bar{e} = 1,\, \bar{e}\iota_e(g)e = \tau_e(g) \rangle\!\rangle.\]
For a vertex $v \in V(X)$, the \emph{fundamental group of $\calg$ based at v}, denoted by $\pi_1(\calg,v)$, is the subgroup of $\bass(\calg)$ given by the elements of the form
\[
a_0e_1a_1e_2\cdots e_n a_n
\]
where $e_i \in E(X)$, $a_i \in \calg_{\tau(e_i)} = \calg_{\iota(e_{i+1})}$  (when the indices occur) and the edge path $e_1\cdots e_n$ is a closed loop based at $v$ in the graph $X$.   We will omit the basepoint and simply write $\pi_1(\calg)$ when there is no risk of confusion.
\end{defn}

\begin{defn}[Bass--Serre tree]
    Let $X_\calg$ be a graph and fix $v_0\in V(X_\calg)$.  The universal cover of a graph of groups $\calg=(X_\calg,\calg_\bullet,\iota_\bullet)$ is the graph $T=T_\calg$ with vertices
\[V(T)=\coprod_{v\in V(X_\calg)} \pi_1(\calg,v)/\calg_v,\]
 edges
\[E(T)=\coprod_{e\in E(X_\calg)}\pi_1(\calg,\iota(e))/\iota_e(\calg_e),\]
and the adjacency map $\iota$ given by inclusions of cosets.  We call $T$ the \emph{Bass--Serre tree} of $\calg$, it is indeed a tree by \cite{Serre2002}.  The graph $T$ comes equipped with an action of $\pi_1(\calg,v_0)$ with cell stabilisers conjugates of the groups in $\calg_\bullet$.
\end{defn}

\subsection{Irreducible Actions on Trees and Outer Automorphisms}

It is convenient to view a group action on a tree as a homomorphism $G \to \Aut(T)$. Rather than viewing trees purely combinatorially, we allow metric simplicial trees: every combinatorial edge is assigned a positive length, and the whole tree is viewed as a geodesic metric space by summing the lengths along an embedded path joining two points. From this perspective, vertices of valence $2$ with the same stabiliser as the incident edges are ``invisible'', and so we are free to assume there are none. A $G$-tree $T$ is called \emph{irreducible} (following Culler and Morgan \cite{CullerMorgan1987}) if $G$ does not stabilise a point, line or end of $T$; an equivalent characterisation of this property is that $G$ contains a non-abelian free subgroup acting freely on $T$. A useful consequence of this property \cite{CullerMorgan1987, AlperinBass1987} is that the translation length function $\|{-}\|_T\colon G \to \ZZ$, recording the minimal displacement of each element of $G$ (it is a standard fact that in trees the infimum is realised), uniquely characterises the action (up to unique equivariant isometry, and provided there is not a proper invariant subtree).

The same group may have many actions on trees. One way to produce an a priori different action is to precompose with an automorphism of the group: $G \xrightarrow{\varphi} G \rightarrow \Aut(T)$. This gives an action of $\Aut(G)$ on the set of $G$-trees. Note that the action has swapped sides: $\Aut(G)$ has a \emph{right} action on the set of \emph{left} $G$-trees. Inner automorphisms do not affect the equivariant isometry class of a $G$-tree, so up to this equivalence relation we can view the action as factoring through the outer automorphism group $\Out(G)$.

\subsection{Deformation spaces and trees}

\subsubsection{Deformation spaces of trees}
Following Forester \cite{Forester2002Deform} we say that two $G$-trees, always viewed up to equivariant isometry, lie in the same \emph{deformation space} if they have the same elliptic subgroups. Irreducibility is an invariant of the deformation space. If the set of elliptic subgroups is preserved by automorphisms then the deformation space is $\Out(G)$-invariant, and we can use this action to study $\Out(G)$ (as for instance with Culler--Vogtmann Outer Space and the outer automorphisms of a free group).

Deformation spaces are contractible \cite[Theorem~6.7]{Clay05} (see also \cite[Theorem~6.1]{GuirardelLevitt2007_defom}) and have the structure of a simplicial complex with some missing faces. Often it is convenient to study the \emph{spine} of a deformation space: this is an equivariant (with respect to the $\Out(G)$ action) deformation retract of the whole space which is a true simplicial complex. It is the geometric realisation of a poset whose elements are the equivariant homothety classes of $G$-trees in the deformation space, and where $T \prec T'$ if $T$ can be obtained from $T'$ by collapsing certain orbits of edges.

\subsubsection{Moves on trees}
A $G$-tree $T$ is \emph{reduced} if whenever there is an edge group equal to an adjacent vertex group ($G_e=G_{\tau(e)}$) its endpoints are in the same orbit. If a $G$-tree is not reduced, then by collapsing orbits of edges violating this condition we obtain a reduced tree in the same deformation space: reduced trees correspond to minimal elements in the poset described above.  Clay and Forester \cite{ClayForester2009} provide a set of moves (defined below), called \emph{slides, inductions} and \emph{$\scra^{\pm1}$-moves}, and prove that any two reduced trees in the same deformation space can be transformed into one another using a combination of these moves. A tree is said to have an \emph{ascending loop} (the terminology comes from considering the quotient graph of groups) if there is an edge whose endpoints are in the same orbit and whose stabiliser is equal to the stabiliser of one of its endpoints.   The property of having an ascending loop in the graph of groups is not an invariant of the deformation space: they can be created or destroyed by the $\scra^{\pm1}$-moves. However, when no tree in the deformation space contains an ascending loop, slide moves suffice to connect any two reduced trees. Note that although we follow Clay and Forester's definitions we have situated them in the tree rather than the graph of groups. 

\begin{defn}[Moves]\label{def:slide_induction_A_moves}
Suppose $T$ is a reduced and irreducible $G$-tree.  In all these definitions, the phrase ``expand $v$ into an edge $e$ stabilised by $H$'' assumes that $H \leq G_v$ and is shorthand for introducing a set of edges indexed by $G_v/H$ with a common initial vertex also denoted $v$. Letting $e$ be the representative corresponding to the coset $H$, we have $G_e=G_{\tau(e)}=H$, $G_{\iota(e)} = G_v$, and for every edge $f$ with $\iota(f)=v$ a $G_v$-equivariant (re)assignment of $\iota(f)$ to some $g\tau(e)$ or to $\iota(e)$. Of course, if $G_f \nleq H^g$ for any $g \in G_v$ then it must be assigned to $\iota(e)$. To preserve the property of being a $G$-tree, every expansion (and their inverses, the collapses) are assumed to be $G$-equivariant: the behaviour described is duplicated across all relevant $G$-orbits.
    
    \begin{enumerate}[(i)]
        \item \textbf{Slide move.} Suppose $e$ and $f$ are edges of $T$ with $\iota(e)=\iota(f)=v$, and that $G_e \leq G_f$. Then there is a \emph{slide move} given by equivariantly sliding the attaching point of $e$ over $f$: first expand $v$ into an edge $e'$, stabilised by $G_f$, with $e$, $f$, and the $G_e$-orbit of $f$ attached to $\tau(e')$, and the remaining edges at $v$ attached to $\iota(e')$. Then collapse $e$.
        
        \item \textbf{Induction moves.} Suppose that $\iota(e)$ and $\tau(e)$ are in the same $G$-orbit, and that $G_{\tau(e)} = G_e < G_{\iota(e)}$. \begin{enumerate}[(a)]
            \item Suppose that $H$ satisfies $G_{\tau(e)} < H \leq G_{\iota(e)}$. Then there is an \emph{increasing induction move} associated to $H$: first expand $\iota(e)$ into $e'$, stabilised by $H$, setting $\iota(e)=\tau(e')$, and leaving all other orbits of edges attached at $\iota(e')$. Then collapse $e$.
            \item Suppose that $H < G_{\tau(e)}$ and that, for every other edge $f$ with $\iota(f)=\tau(e)=v$ there is a representative of its $G_v$-orbit with $G_f \leq H$. Then there is a \emph{decreasing induction move} associated to $H$: first expand $\tau(e)$ into $e'$ stabilised by $H$, setting $\tau(e)=\iota(e')$ and attaching all other orbits of edges to an appropriate $g\tau(e')$. Then collapse $e$.
        \end{enumerate}
        
        \item \textbf{$\scra$-move.}  Suppose that $\iota(e)$ and $\tau(e)$ are in the same $G$-orbit, that $e'$ is an edge in the same $G$-orbit as $e$ with $\iota(e')=\tau(e) = v$ and $G_{e'} < G_e < G_v$. Then there is an \emph{$\scra$-move} consisting of first expanding $v$ into an edge $f$ stabilised by $G_e$, attaching $\tau(e)$ and $\iota(e')$ to $\tau(f)$ and the remaining edges to $\iota(f)$. Then perform an increasing induction associated to any $H$ with $G_{\tau(e)} < H \leq G_{\iota(e)}$.
        
        \item \textbf{$\scra^{-1}$-move.} Suppose that $\iota(e)$ and $\tau(e)$ are in the same $G$-orbit, that $G_{\tau(e)} = G_e < G_{\iota(e)}$, that $e'$ is an edge in the same $G$-orbit as $e$ with $\iota(e')=\tau(e) = v$, and there is exactly one more $G_{\tau(e)}$-orbit of edges incident at $\tau(e)$, one of whose representatives $f$ has stabiliser $G_f$ satisfying $G_f < G_{e'}$. Then there is an \emph{$\scra^{-1}$ move}: first perform the decreasing induction associated to $G_{f}$, resulting in a tree that is not reduced: $f$---which was not a loop by assumption---is now attached at a vertex with the same stabiliser. Collapse $f$: provided $G_{f}$ was a proper subgroup of its other vertex group, this produces a reduced tree without the ascending loop.
    \end{enumerate}
\end{defn}

We have illustrated some of these moves in the case of the Leary--Minasyan groups, see Figures~\ref{fig:unique} and \ref{fig:expand}.

\begin{remark} We record a number of remarks.
    \begin{enumerate}
        \item In a locally finite $G$ tree with multiple orbits of edges, the valence at the operative vertex (orbit) increases after an increasing induction move and decreases after a decreasing one. It follows that increasing inductions are possible whenever the graph of groups contains an ascending loop, but decreasing ones may not be. 

        \item The result of a slide or a decreasing induction depends on the choice of an edge and not just its orbit.

        \item An induction with $H=G_{\tau(e)}$ could in principle be allowed (as it is in \cite{ClayForester2009}, for instance)---however, it would be ineffective (essentially because the vertex created after the expansion is invisible).  We therefore, without loss of generality, do not allow these types of inductions.  Note also that an increasing induction with $H=G_{\iota(e)}$ is allowed, but agrees with sliding all other edges at $\iota(e)$ across $\overline{e}$, and a decreasing induction where $H=G_f$ with $f$ in the same orbit as $e$ and $\iota(f)=\tau(e)$ is the same as sliding across $e$.

        \item Decreasing inductions (as in the definition of an $\scra^{-1}$-move) and also slides do not necessarily preserve the property of being reduced.  
    \end{enumerate}
\end{remark}

\subsubsection{Actions on trees}
Forester \cite{Forester2002Deform} and Levitt \cite{Levitt2005_rigid} characterised the \emph{rigid} actions on trees: these have the property that they are the only reduced tree in their deformation space. In particular, whenever the deformation space is $\Out(G)$-invariant, they are preserved by every outer automorphism.

\begin{thm}~\cite[Theorem 1]{Levitt2005_rigid}
\label{thm:rigidity}
    Let $T$ be a reduced cocompact $G$-tree which is not an ascending HNN extension. It is rigid if and only if, given two oriented edges $e, f$ with the same origin $v$ such that $G_e \leq G_f$, one of the following holds: \begin{enumerate}
        \item \label{rigidity-1} $e$ and $f$ are in the same $G_v$ orbit.
        \item \label{rigidity-2} $e$ and $\overline{f}$ are in the same $G$-orbit, and $G_e = G_f$.
        \item \label{rigidity-3} there is an edge $f'$ with origin $v$, in the same $G$-orbit as $\overline{f}$ such that $G_f = G_{f'} = G_v$. Furthermore, there are only three $G_v$-orbits of edges with origin $v$ (those of $e, f, f'$).
    \end{enumerate}
\end{thm}

Given an $\Out(G)$-invariant deformation space $\cald$, one can (in principle, and in our case in practice) understand $\Out(G)$ through studying this action and its stabilisers. In particular, we will need to study the $\Aut(G)$ (or $\Out(G)$) stabiliser of a $G$-tree $T$. This study is made easier by the observation (derived from the uniqueness of the equivariant isometries of \cite{CullerMorgan1987} and \cite{AlperinBass1987}) that $\Stab_{\Aut(G)}(T)$ itself acts on $T$, in a way which is \emph{compatible} with the $G$-action via \[\varphi \cdot (g \cdot x) = \varphi(g) \cdot (\varphi \cdot x).\] Note that a consequence of this compatibility is that if $\varphi$ fixes a vertex $v$, it preserves its stabiliser too.  Namely, we have $v = \varphi \cdot (G_v \cdot v) = \varphi(G_v) \cdot (\varphi \cdot v) = \varphi(G_v) \cdot v$ for one inclusion, and the same argument with $\varphi^{-1}$ for the other.

Within the literature one finds a small number of similar calculations of automorphisms of graphs of groups. We mention two general versions of the theory: Levitt \cite{Levitt2005_autHyp} covers many cases but is poorly suited to our situation since a running assumption is that the edge groups have trivial outer automorphism group, or at least that the subgroup that can be induced given the other constraints of the situation is trivial. In our case this is usually false, since the edge groups are finite index and hence we would expect the automorphisms of the vertex group to induce many automorphisms of the edge groups. Bass and Jiang's \cite{BassJiang1996} covers our case, but brings many technicalities that can be avoided since our edge and vertex groups are all abelian. We indicate the broad strokes of the common strategy here:
\begin{enumerate}
    \item After passing to a finite index subgroup, we can assume that $G$-orbits are preserved by the $\Stab_{\Aut(G)}(T)$-action.
    \item Now there are quotient maps to the (outer) automorphism group of each vertex group. (Outer automorphisms are needed here since there is a choice of conjugating element involved in defining the quotient maps.) The image is usually much smaller than the whole group, since it needs to preserve a ``peripheral'' structure coming from the incident edge groups.
    \item The common kernel of these maps consists of \emph{(bi)twists}: on the level of the Bass group these are automorphisms replacing an edge $e$ by a path $geh$, where $g$ and $h$ are elements drawn from the initial and terminal vertex groups of the edge. (In the general theory, not every vertex group element will give rise to an automorphism.)
\end{enumerate}

\section{The non-metabelian case: introducing the deformation space and first results}
\label{sec:non_metab_intro}

Let $G=\LM(A,L)$ be a Leary--Minasyan group and suppose that $G$ is not metabelian, or equivalently that $G$ is not an ascending HNN extension. 
This is sufficient (and indeed necessary) to give that the action of $G$ on the Bass--Serre tree $S$ arising from its presentation as an HNN extension is irreducible: for instance, since there exists an element $g \in \ZZ^m \smallsetminus (L \cup AL)$ and for such an element $t$ and $tg$ are two independent loxodromic elements. 
A graph of groups associated to this action is shown in \Cref{fig:gog_for_LMAL}. We fix once and for all an isomorphism between $\ZZ^m$ and $H = \langle x_1, \dots, x_m \rangle$, our preferred vertex stabiliser. 

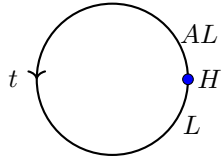
\begin{figure}[ht]
\centering
\begin{tikzpicture}
\draw [thick, middlearrow={.5}{0}{>}{$t~$}{left, rotate=90}] (0,0) arc [start angle =0, end angle =360, x radius =1cm, y radius=1cm] (0,0);
\draw [fill=blue] (0,0) circle (2pt) node [right] {$H$};
\node at (0.15,0.6) {$AL$} ;
\node at (0.05,-0.6) {$L$} ;
\end{tikzpicture}
\caption{The graph of groups splitting for $\LM(A,L)$ as in \eqref{eqn:LMAL}.}
\label{fig:gog_for_LMAL}
\end{figure}

We consider the deformation space $\cald$ of $G$-trees containing $S$, the Bass--Serre tree of the defining presentation of $G$ as a HNN extension \eqref{eqn:LMAL}.

\begin{table}[ht]
    \centering
    \begin{tabular}{|c|c|}
       \hline
        $G=\LM(A,L)$ & $A\in \GL_m(\QQ)$ and $L \neq \ZZ^m \neq AL$ \\
        $S$ & The Bass--Serre tree of \Cref{fig:gog_for_LMAL} \\
        $\cald$ & The deformation space of $G$-trees containing $S$ \\
        $\calc^+$ & $\{M \in \GL_m(\ZZ) : MAM^{-1}=A, M(L)=L \}$ \\
        $\calc^-$ & $\{M \in \GL_m(\ZZ) : MAM^{-1}=A^{-1}, M(L)=AL \}$ \\
        $\calc = \calc^+ \sqcup \calc^-$ & The subgroup $\calc^+ \times \{+1\} \cup \calc^- \times \{-1\} \leq \GL_m(\ZZ) \times \{\pm1\}$ \\
        $\gamma\colon \calc \to \Aut(G)$ & The automorphisms $\gamma_{M,\varepsilon}|_H \equiv M$ and $\gamma_{M,\varepsilon}(t) = t^\varepsilon$ for $(M,\varepsilon) \in \calc$ \\
        $\beta\colon H \to \Aut(G)$ & The automorphisms $\beta_h|_H \equiv \id_H$ and $\beta_h(t) = ht$ for $h \in H$ \\
        $\delta\colon H \rtimes \calc \to \Aut(G)$ & The common extension of $\beta$ and $\gamma$ \\
        \hline
        $T_n$ & The Bass--Serre tree of \Cref{fig:BAD_trees}, for $n \in \ZZ_{>0}$ \\
        $L_n$ & The subgroup $t^{-n}Lt^n$, $n \in \ZZ_{\geq 0}$ \\
        $L_\infty$ & The subgroup $\bigcup_{n=0}^\infty L_n$ \\
        $\beta_n\colon L_n \to \Aut(G)$ & The automorphisms $\beta_{n,\ell}|_H \equiv \id_H$ and $\beta_{n,\ell}(t) = \ell t$ for $\ell \in L_n$ \\
        $\delta_n\colon L_n \rtimes \calc \to \Aut(G)$ & The common extension of $\beta_n$ and $\gamma$ \\
        \hline
    \end{tabular}
    \caption{Table of notation for \Cref{sec:non_metab_intro}.  The second part of the table is only applicable when $AL < L$.}
    \label{tab:notation_non_metab}
\end{table}

\begin{prop} \label{prop:D-AutG-inv}
    The deformation space $\cald$ is $\Aut(G)$-invariant.
\end{prop}
\begin{proof}
    By \cite[Lemma 3.1(1)]{Valiunas2022}, $H=\langle x_1 , \dots, x_m \rangle$ represents the unique conjugacy class of maximal elliptic subgroups in $S$. We need to show it is $\Aut(G)$-invariant. Since it is a finitely generated group, it is enough to check that $\varphi(g)$ is elliptic for every $g \in H$ and $\varphi \in \Aut(G)$. (A finitely generated subgroup is elliptic if and only all of its elements are.) But \cite[Proposition 2.7]{Valiunas2022} provides an algebraic characterisation of these elements, and so this holds.
\end{proof}

We have the following characterisation of Forester--Levitt rigidity among these groups and trees. 

\begin{prop} \label{prop:rigid_iff_no_proper_containment}
    Suppose $G$ is a Leary--Minasyan group that is not an ascending HNN extension. Then the Bass--Serre tree $S$ corresponding to its presentation as an HNN extension is rigid if and only if the associated subgroups $L$ and $AL$ are either incomparable or equal.
\end{prop}

\begin{proof}
    Observe that $S$ is reduced since there is only one orbit of vertices and so each edge necessarily has its endpoints in the same orbit. Also note that since $G_v$ is abelian, if two incident edges $e,e'$ are in the same $G_v$-orbit then $G_e=G_{e'}$. In the light of \Cref{thm:rigidity} we need to show that our hypotheses are equivalent to every pair of edges $\{e,f\}$ originating at a common vertex satisfying either condition \ref{rigidity-1} or \ref{rigidity-2} (the third condition is irrelevant since there are only two $G_v$-orbits of edges at each vertex). It is no loss of generality to suppose that the originating vertex $v$ is stabilised by $H$ and so the incident edge stabilisers are all equal to $L$ or $AL$ (depending on the $G_v$-orbit of the edge).  There are two cases depending on whether $L$ and $AL$ are incomparable or equal.

    For the first case suppose that $L$ and $AL$ are incomparable.  Let $e$ and $f$ be vertices with the same origin so that $G_e \leq G_f$. But since the two $G_v$-orbits of edges have incomparable stabilisers, we must have that $e$ and $f$ are in the same $G$-orbit (and in fact $G_e=G_f$). Conversely, if every pair of edges with $G_e \leq G_f$ are in the same $G_v$-orbit, then $L$ and $AL$ must be incomparable.  This completes the first case.

    For the second case suppose that $L=AL$.  As in the first case we have a pair of edges $e,f$ with the same origin and $G_e \leq G_f$. By the observations above, this means that $G_e=G_f$. If they are not in the same $G_v$-orbit, then (since there is only one $G$-orbit of edges) there is an element $g$ (it is contained either in $G_vtG_v$ or in $G_vt^{-1}G_v$) so that $ge=\overline{f}$.  Thus, $\{e,f\}$ satisfy item \ref{rigidity-2} of \Cref{thm:rigidity}.   Conversely, suppose $e$ and $f$ are adjacent edges satisfying item \ref{rigidity-2}. Since they are in distinct $G_v$-orbits, we can assume (without losing generality) that $G_e=L$ and $G_f=AL$. But then we have that $L=G_e=G_f=AL$.  This completes the second case and the proof of the proposition.
\end{proof}

\subsection{The stabiliser of a one-edge tree}
In this section we calculate the $\Out(G)$-stabiliser of the action on $S$, as in \Cref{fig:gog_for_LMAL}. This will complete the proof in the rigid case, where this tree is $\Out(G)$-invariant, and will also be required in our analysis of the non-rigid case.

\begin{lemma}
\label{find_reps_of_stab_S}
    If $\left|\det(A)\right|\neq 1$ then, after composing with an inner automorphism, any element of $\Stab_{\Aut(G)}(S)$ preserves $H$ and sends $t$ to $ht$ for some element $h$ of $H$. If $\left|\det(A)\right|=1$ then these automorphisms form a subgroup of index at most 2 in $\Stab_{\Aut(G)}(S)$, and any element outside this subgroup (after composing with inner automorphisms) preserves $H$ and sends $t$ to $t^{-1}h$ for some $h \in H$. 
    Conversely, any automorphism preserving $H$ and sending $t$ to $ht$ or $t^{-1}h$ (with $h \in H$) is contained in $\Stab_{\Aut(G)}(S)$.
\end{lemma}
\begin{proof}
Let $\varphi \in \Stab_{\Aut(G)}(S)$. The action of $\Stab_{\Aut(G)}(S)$ on $S$ has one orbit of edges and of vertices (though some care is needed; the new action might have edge inversions). Let $v$ be the (unique) vertex stabilised by $H$, and note that after composing with an inner automorphism, $v$ is also stabilised by $\varphi$. This implies that $\varphi$ preserves $H$. 

Now consider the edges adjacent to $v$, of which there are two $H$-orbits. They are either preserved or swapped by $\varphi$, and the second case is only possible if there is an automorphism $\varphi'$ of $H$ exchanging $L$ and $AL$: indeed, since $\Stab(\varphi(e))=\varphi(\Stab(e))$ and $H$ is abelian, edges adjacent to $v$ and in the same $H$-orbit will have the same stabiliser. In particular, $L$ and $AL$ must have the same index in $H$, and so $\left|\det(A)\right|=1$.

Suppose the $H$-orbits of edges adjacent to $v$ are swapped by $\varphi$.  Consider $tv$, stabilised by $tHt^{-1}$. After composing with another inner automorphism (by an element of $H$), we can assume that $\varphi(H)=H$ and $\varphi(tv)=t^{-1}v$, the latter vertex stabilised by $t^{-1}Ht$. In particular, we see $\varphi(tHt^{-1})=t^{-1}Ht$, and expanding and rearranging we find that $t\varphi(t)$ normalises $H$. But $H$ is self-normalising (it is not equal to any of its incident edge stabilisers) and so $t\varphi(t)$ is contained in $H$, and in particular $\varphi(t)=t^{-1}h$ for some element $h$ of $H$.

Now consider the subgroup (of index at most 2) where the $H$-orbits of edges adjacent to $v$ are not interchanged by $\varphi$. This time, apply another inner automorphism (by an element of $H$) so that $\varphi(H)=H$ and $\varphi(t^{-1}Ht) = t^{-1}Ht$, equivalent to fixing $t^{-1}v$ as well as $v$. This implies that $\varphi(t)t^{-1}$ normalises $H$ and so is an element of $H$. In particular $\varphi(t) = ht$ for some element $h$ of $H$.

We now show the converse. Since $\varphi(H)=H$, its conjugacy class is preserved; since $H$ is self-normalising the conjugates of $H$ are in one-to-one correspondence with the vertices of $S$. So to claim that $\varphi$ stabilises $S$, it is enough to show that two conjugates stabilise adjacent vertices if and only if their images under $\varphi$ do.  We have that inner automorphisms stabilise every tree.  Thus it is enough to show that $H$ and some distinct $H^g$ stabilise adjacent vertices if and only if their images under $\varphi$ do.   We have that $H$ and $H^g$ stabilise adjacent vertices if and only if $g \in HtH$ or $g \in Ht^{-1}H$. If $\varphi(t)=t^{-1}h$ we have that $\varphi(g) \in Ht^{-1}hH=Ht^{-1}H$ or $\varphi(g) \in Hh^{-1}tH=HtH$. Alternatively, if $\varphi(t)=ht$, then $\varphi(g) \in HhtH=HtH$ or $\varphi(g) \in Ht^{-1}h^{-1}H = Ht^{-1}H$, so the adjacency condition is preserved by all of these automorphisms.
\end{proof}

Let $\calc^+$ be the subgroup of $\GL_m(\ZZ)$ commuting with $A$ and fixing $L$ setwise; that is,
\[
\calc^+ \coloneqq \{M \in \GL_m(\ZZ) : MAM^{-1}=A, M(L)=L \}.
\]
Note that this implies that elements of $\calc^+$ fix $AL$ setwise, and that $\calc^+$ is a finite index subgroup of the centraliser of $A$ in $\GL_m(\ZZ)$. 
Also, let
\[
\calc^- \coloneqq \{ M \in \GL_m(\ZZ): MAM^{-1}=A^{-1}, M(L)=AL\}.
\]
Sometimes $\calc^- = \varnothing$, which must in particular happen if we have ${\left|\det A\right| \neq 1}$. Otherwise, it follows that for any $M \in \calc^-$, we have $M(AL)=MA(L)=A^{-1}M(L)=A^{-1}(AL)=L$: the subgroups $L$ and $AL$ are exchanged by $M$.

Note that if $\calc^-$ is not disjoint from $\calc^+$, then $A = A^{-1}$ and $L = AL$, and consequently $\calc^- = \calc^+$.  Note also that the product of two elements of $\calc^-$ is contained in $\calc^+$, and the product of an element of $\calc^+$ and an element of $\calc^-$ (in either order) is contained in $\calc^-$.  Thus $\calc^-$ is either empty, or equal to $\calc^+$, or disjoint from $\calc^+$ and makes $\calc^+ \cup \calc^-$ into a subgroup of $\GL_m(\ZZ)$. 

Now let
\[
\calc = \calc^+ \times \{+1\} \cup \calc^- \times \{-1\} \subseteq \GL_m(\ZZ) \times \{\pm1\},
\]
and note that $\calc$ is a subgroup of $\GL_m(\ZZ) \times \{\pm1\}$ by the above discussion.  There is a homomorphism $\gamma\colon \calc \rightarrow \Aut(G)$, where the image $\gamma_{M,\varepsilon}$ of $(M,\varepsilon)$ for a matrix $M$ and $\varepsilon \in \{\pm 1\}$ is given by $\gamma_{M,\varepsilon}(h) = M(h)$ for elements $h \in H$ (through the fixed isomorphism $H \cong \ZZ^m$), and $\gamma_{M,\varepsilon}(t)=t^\varepsilon$. This is an endomorphism of $G$ since the relations are all satisfied, and an automorphism since $\gamma_{M^{-1},\varepsilon}$ is an inverse.

In the following arguments, we will slightly abuse the notation by identifying $\calc^+$ and $\calc^-$ with the corresponding subsets $\calc^+ \times \{+1\}$ and $\calc^- \times \{-1\}$ of $\calc$.

\begin{prop}
\label{prop:centraliser_to_aut_nonmet}
    The following conclusions hold:
    \begin{enumerate}
        \item The map $\gamma\colon \calc \rightarrow \Aut(G)$ is injective.
        \item $\im(\gamma)\cap \Inn(G)=\{1\}$.
        \item $\im(\gamma)=\{\varphi \in \Aut(G): \varphi(H)=H,  \varphi(t)=t^{\pm 1}\}$. 
        \item $\gamma(\calc^+)=\{\varphi \in \Aut(G): \varphi(H)=H,  \varphi(t)=t\}$.
    \end{enumerate} 
\end{prop}

\begin{proof}
    First, for any $M \in \calc^-$, the map $\gamma_{M,-1}$ is neither trivial nor inner (exponent sum on the stable letter is preserved by inner automorphisms), so for the first two claims it suffices to consider the image of $\calc^+$. Since $\gamma_{M,1}$ acts like the matrix $M \in \calc^+$ on the subgroup~$H$, it is trivial if and only if $M$ was the identity matrix. This is also the only element of $\calc^+$ mapped into $\Inn(G)$: suppose for some $M \in \calc^+$ and $g \in G$ that $\gamma_{M,1}(x) = gxg^{-1}$ for all $x \in G$. The subgroup $H$ is preserved by $\gamma_{M,1}$, and is its own normaliser, and so we have $g \in H$. But since $H$ is abelian, this implies that $\gamma_{M,1}$ is trivial on $H$, and hence that $M=\id$.

    Note that by definition the images of $\calc$ and $\calc^+$ under $\gamma$ are contained in the claimed subgroups. We show the other inclusions, and to see this note that the fact that $\varphi(H) = H$ implies that there exists $M \in \GL_m(\ZZ)$ satisfying $\varphi(h) = M h$ for all $h \in H$.

    For the subgroup satisfying $\varphi(t)=t$, note that $t^{-1}Ht$ is also preserved and so too is the intersection $H \cap t^{-1}Ht = L$. Hence, $M(L)=L$. 
    Finally we check that $M$ commutes with $A$. For all $h \in L$, we have
\[
MA(h) = M(tht^{-1}) = \varphi(tht^{-1}) = t\varphi(h)t^{-1} = A\varphi(h) = AM(h),
\] implying that $MA = AM$. So $\varphi$ agrees with $\gamma_{M,1}$ for some $M \in \calc^+$, giving the reverse inclusion.

For the general statement, suppose $\varphi(t)=t^{-1}$. Then $\varphi(t^{-1}Ht)=tHt^{-1}$, and so for the intersection $\varphi(L)=\varphi(H \cap t^{-1}Ht) = H \cap tHt^{-1}= AL$. In particular, we have $M(L)=AL$. For all $h \in L$, we have \[
MA(h)=M(tht^{-1})=\varphi(tht^{-1}) = t^{-1}\varphi(h)t = A^{-1}(\varphi(h))= A^{-1}M(h),
\] implying that $MAM^{-1}=A^{-1}$. So $\varphi$ agrees with $\gamma_{M,-1}$ for some $M \in \calc^-$.\qedhere

\end{proof}

Note that this implies that the image of $\gamma$ stabilises $S$. We now study those automorphisms in the stabiliser of $S$ which fix $H$ pointwise. To enable this, consider a map $\beta\colon H \to \Aut(G)$ where the image $\beta_h$ fixes $H$ pointwise and with $\beta_h(t)=ht$. Note that all the relations are satisfied and that $\beta_{h^{-1}}$ is an inverse map.

\begin{prop}
\label{prop:H_to_aut_nonmet}
    The following conclusions hold:
    \begin{enumerate}
        \item The map $\beta\colon H \to \Aut(G)$ is injective.
        \item $\im(\beta)=\{\varphi \in \Aut(G) : \varphi(x) = x \text{ for } x \in H, \varphi(t)=ht \text{ for some } h \in H \}$.
        \item $\beta_h$ is inner if and only if $h$ lies in $(I-A)L$.  In particular, the natural projection $\overline{\beta}\colon H \to \Out(G)$ has image isomorphic to $H/(I-A)L$.
    \end{enumerate}
     
\end{prop}

\begin{proof}
    Injectivity follows immediately: if $\beta_h = \id$ then $h = \beta_h(t)t^{-1}=1$. For surjectivity, note that by definition any $h \in H$ gives rise to an automorphism $\beta_h$.
    
    To understand when these automorphisms are inner, suppose $\beta_h$ agrees with conjugation by some element $g$ of $G$; in particular we will have that $ht=gtg^{-1}$ and so $h = gtg^{-1}t^{-1}$.

    So by HNN normal forms, we must have that $g^{-1}$ (and so also $g$) is an element of $L$. 
    Rewriting $h = g (tg^{-1}t^{-1})$ via our standard abuse of notation (and switching to additive notation), this gives that $h \in (I-A)L$, as claimed. Conversely, every element $g$ of $L$ gives rise to a $\beta_h$ by setting $h = gtg^{-1}t^{-1}$. The final claim follows immediately.
\end{proof}

Note that the homomorphisms $\beta \colon H \to \Aut(G)$ and $\gamma \colon \calc \to \Aut(G)$ extend to a homomorphism $\delta \colon H \rtimes \calc \to \Aut(G)$, where the action of $\calc$ on $H$ is the natural one, by considering the matrix in the first coordinate of $\calc$.  Denote by $\overline{\calb}$ the image of $H$ under $\overline{\beta}: H \to \Out(G)$.

\begin{prop}\label{prop:stab_S}
    The stabiliser of $S$ under the action of $\Out(G)$ is isomorphic to $\overline \calb \rtimes \calc$, with the isomorphism induced by $\delta \colon H \rtimes \calc \to \Aut(G)$. 
    In particular, if $\left|\det(A)\right| \neq 1$, then this stabiliser is isomorphic to $\overline \calb \rtimes \calc^+$.
\end{prop}

\begin{proof}
    First, using \Cref{find_reps_of_stab_S}, we see that any element of $\Stab_{\Out(G)}(S)$ has a representative $\varphi$ (perhaps more than one) satisfying $\varphi(H)=H$ and $\varphi(t)=(ht)^{\pm 1}$, with $h\in H$. So there is a quotient map to $\GL_m(\ZZ)$ given by considering the restriction $\varphi|_H$ (note that this is independent of the choices made in the construction of the representative $\varphi$ since $H$ is self-normalising). 
    There is also a map to $\Out(\ZZ)$ by considering the quotient $\chi\colon G \to \ZZ$ given by sending $h \mapsto 0$ for $h \in H$ and $t \mapsto 1$. The kernel of $\chi$ is generated by the elliptic elements of $G$, and as in \Cref{prop:D-AutG-inv} there is an algebraic characterisation of these elements, and so the kernel is characteristic and there is an induced map $\Out(G) \to \Out(\ZZ)$. 
    Assemble these maps into a product $\Stab_{\Out(G)}(S) \to \GL_m(\ZZ) \times \Out(\ZZ)$. 
    
    The kernel of this map consists of elements of $\Out(G)$ represented by some $\varphi$ fixing $H$ pointwise and sending $t$ to $ht$. By \Cref{prop:H_to_aut_nonmet}, these automorphisms are described by the image of $\beta: H \to \Aut(G)$, and so the natural quotient to $\Out(G)$ is exactly equal to $\overline{\calb}$.

    Suppose $\varphi$ is non-trivial and in the image of $\gamma$. Then by the characterisation of this image (\Cref{prop:centraliser_to_aut_nonmet}), $\varphi$ has non-trivial image in the projection $\Stab_{\Out(G)}(S) \to \GL_m(\ZZ) \times \Out(\ZZ)$. On the other hand, for any $\varphi \in \Stab_{\Out(G)}(S)$, the previous paragraph (together with the description in \Cref{find_reps_of_stab_S}) implies that after composing with an element in the kernel of the projection map, $\varphi$ is contained in $\{\varphi \in \Aut(G): \varphi(H)=H,  \varphi(t)=t^{\pm 1}\}$. In particular, the image of $\Stab_{\Out(G)}(S)$ agrees with the image of the subgroup $\gamma(\calc)$, and so there is a splitting (given by~$\gamma)$.

    The final claim follows since in this case $\calc = \calc^+$.
\end{proof}

\begin{lemma}\label{lem:Inn_G_HNN}
    $\Inn(G)$ is an HNN extension of a finitely generated abelian group.
\end{lemma}
\begin{proof}
    We have that $G$ acts on the tree $S$ without a global fixed point or fixed end.  
    In particular, the kernel of the action of $G$ on $S$ contains $Z(G)$.  
    Thus, $\Inn(G)\cong G/Z(G)$ acts on $S$ and the graph $S/\Inn(G)$ is isomorphic to the graph $S/G$.  
    For a vertex or edge $\sigma$ of $S$, its stabiliser is $\Stab_{\Inn(G)}(\sigma) =G_\sigma/(Z(G)\cap G_\sigma)$ which is a quotient of the finitely generated free abelian group $G_\sigma$.  
    Whence, the lemma. 
\end{proof}

\subsection{Stronger finiteness properties}
For a group $G$ a \emph{model for a $\mathsf{K}(G,1)$ space} is an aspherical CW complex $X$ with $\pi_1(X)\cong G$.  A group $G$ is \emph{type $\mathsf{F}_\infty$} if there exists a model $X$ for a $\mathsf{K}(G,1)$ with at most finitely many cells in each dimension.  A group $G$ is \emph{type $\mathsf{F}$} if there exists a model for a $\mathsf{K}(G,1)$ with finitely many cells.  A group $G$ is \emph{type $\mathsf{VF}$} if it is virtually type $\mathsf{F}$.

We are now in a position to prove a strengthening of the rigid case of \Cref{thmx:not_metabelian}.

\begin{thm}\label{thm:nonMet_VF_Foo_fin}
    Suppose $G=\LM(A,L)$ is a Leary--Minasyan group with $L$ and $AL$ either equal or incomparable. The following conclusions hold:
    \begin{enumerate}
        \item \label{it:Inn-Finfty} $\Inn(G)$ is type $\mathsf{F}_\infty$;
        \item \label{it:Aut-Finfty} $\Aut(G)$ is type $\mathsf{F}_\infty$;
        \item \label{it:Out-VF} $\Out(G)$ is type $\mathsf{VF}$;
        \item \label{it:Out-finite} $\Out(G)$ is finite if and only if $A$ has finite centraliser in $\GL_m(\ZZ)$ and does not have $1$ as an eigenvalue.
    \end{enumerate}
\end{thm}

\begin{proof}
    Note that \ref{it:Aut-Finfty} easily follows from \ref{it:Inn-Finfty} and \ref{it:Out-VF} along with the exact sequence $1\to\Inn(G)\to\Aut(G)\to \Out(G)\to 1$.

    We first prove \ref{it:Inn-Finfty}.  The construction we now describe is completely standard and can be found in for example \cite[Chapter 1.B]{Hatcher2002}.  We include a brief sketch for the reader.   
    By \cref{lem:Inn_G_HNN}, we see that $\Inn(G)$ is an HNN extension of a finitely generated abelian group $A$ with associated subgroups $B,B'$.  
    The abelian group $A$ admits a $\mathsf{K}(A,1)$ space $Y$ with finitely many cells in each dimension (in fact, it is the direct product of a $k$-torus and finitely many spaces of the form $S^\infty/(\ZZ_{d_i})$).  
    A $\mathsf{K}(B,1)$ space $Z$ for the subgroups $B\cong B'$ admit a similar description.  
    Now, (by cellular approximation if necessary) fix cellular maps  $i\colon Z\to X$ and $i'\colon Z\to X'$ which induce the inclusions of $B$ and $B'$ on $\pi_1$.  
    An aspherical CW complex $X$ with $\pi_1(X)\cong \Inn(G)$ is given by the space $(Y \coprod Z\times[0,1])/\sim$ where $\sim$ is defined by $(z,0)\sim i(z)$ and $(z,1)\sim i'(z)$, see \cite[Theorem 1B.11]{Hatcher2002}.
    If $Y$ has $c_{Y,n}$ many $n$-cells and $Z$ has $c_{Z,n}$ many $n$-cells, then $X$ admits a CW structure with $c_{Y,n}+c_{Z,n-1}$ many $n$-cells.  
    Hence, $\Inn(G)$ is type $\mathsf{F}_\infty$ as required.

    We now prove \ref{it:Out-VF}.  By \Cref{prop:rigid_iff_no_proper_containment}, $S$ is a rigid tree for $G$, and so $\Out(G) \cong \Stab_{\Out(G)}(S)$, which is described in \Cref{prop:stab_S}. It has a subgroup of index at most 2 fitting into a split short exact sequence \[ 1 \to \overline{\calb} \to \Out^+(G) \to \calc^+ \to 1.\]
    Here $\calc^+$ is a finite index subgroup of a centraliser in $\GL_m(\ZZ)$, hence it is an arithmetic group by \cite[Theorem~6.5]{BorelHarishChandra1962} (see also \cite[Algorithm B]{GrunewaldSegal1980} and \cite{EickHofmannOBrien2019}) and so is type $\mathsf{VF}$ by \cite{BorelSerre1973}. On the other hand, $\overline{\calb}$ is a finitely generated abelian group, so has a finite index torsion free subgroup which is type $\mathsf{F}$. Since the short exact sequence is split, $\overline \calb$ is finitely generated,
    and both $\overline\calb$ and $\calc^+$ are residually finite, there is a finite index subgroup of $\Out^+(G)$ avoiding the torsion subgroup of $\overline{\calb}$, and hence $\Out(G)$ has a finite index subgroup which is type $\mathsf F$.

    Finally, we prove \ref{it:Out-finite}.   Notice that $\Out(G)$ is finite if and only if both $\overline{\calb}$ and $\calc^+$ are. In the first instance this is equivalent to $(I-A)L$ being finite index in $H$, which happens if and only if $I-A$ is injective (as a map from $L \cong \ZZ^m$ to $H \cong \ZZ^m$), or equivalently, $1$ is not an eigenvalue of~$A$. In the second instance, this is equivalent to $C_{\GL_m(\ZZ)}(A)$ being finite.
\end{proof}

\begin{remarks}
With $G$ as in \Cref{thm:nonMet_VF_Foo_fin}, we record two remarks:
\begin{enumerate}
    \item It is plausible that one cannot conclude $\Aut(G)$ is type $\mathsf{VF}$.  Indeed, $\Aut(G)$ is an extension of $\Out(G)$, which is virtually an extension of an abelian group $\overline{\calb}$ by a linear arithmetic group $\calc^+$, by the group $G/Z(G)$.  Now, in general $G/Z(G)$ contains torsion and it is already a difficult problem to decide when an extension of an arithmetic group by a finite group is residually finite.  In fact, often such extensions are not residually finite \cite{Deligne1978,PrasadRapinchuk1996}.  See \cite{Stover2025} for a survey.
    \item There are algorithms to produce a generating set for the centraliser of an element of $\GL_m(\ZZ)$ \cite{Grunewald1980,EickHofmannOBrien2019} and to decide if a finitely generated subgroup of $\GL_m(\ZZ)$ is finite \cite[\S4.2]{DetinkoFlanneryOBrien2013}.  In particular, it is decidable whether $\Out(G)$ is finite.
\end{enumerate}
\end{remarks}

\subsection{An aside on non-positive curvature}
It is a long standing open question whether every virtually $\CAT(0)$ group is in fact $\CAT(0)$.  We verify here that any group which contains a centreless
$\CAT(0)$ Leary--Minasyan group as a finite index normal subgroup is in fact $\CAT(0)$.

\begin{thm}\label{thm:CATs}
    Let $G=\LM(A,L)$ be a Leary--Minasyan group with trivial centre and with $A$ 
    conjugate in $\GL_m(\RR)$ to a matrix in $\OO(m)$.  If $Q$ is a finite group, then any group $\Gamma$ fitting into an extension \[1\to G\to \Gamma\to Q\to 1\] is $\CAT(0)$.
\end{thm}

\begin{proof}
    Since $A$ is conjugate to an orthogonal matrix, the lattices $L$ and $AL$ have the same index in $\ZZ^m$.  Note also that since $G$ is centreless, we can identify it with a subgroup of $\Aut(G)$. After replacing $\Gamma$ with $\Gamma/N$ if necessary, where $N$ is the kernel of the map $\Gamma \to \Aut(G)$ (note that $N$ is finite and $N \cap G = \{1\}$), we can further identify $\Gamma$ with a subgroup of $\Aut(G)$.

    Suppose first that $L$ and $AL$ are equal to $\ZZ^m$.  Then we have $G = \ZZ^m \rtimes_A \ZZ$ for a matrix $A \in \GL_m(\ZZ)$.  But then, since $A$ is conjugate to an orthogonal matrix, it follows that $A$ has finite order (since $\GL_m(\ZZ)$ is discrete in $\GL_m(\RR)$ and any $\GL_m(\RR)$-conjugate of $\OO(m)$ is compact), contradicting the fact that $G$ has trivial centre.
    
    Suppose now that $L$ and $AL$ have index $\geq 2$ in $\ZZ^m$.  We first note that we are in the rigid case: indeed, since the lattices $L$ and $AL$ have the same index in $\ZZ^m$, they are are either incomparable or equal, and are not equal to $\ZZ^m$ by the assumption.  Hence, we may apply \Cref{prop:rigid_iff_no_proper_containment} to conclude that the tree $S$ as in Figure~\ref{fig:gog_for_LMAL} is rigid.  In particular, by Proposition~\ref{prop:D-AutG-inv}, the $G$-action on $S$ 
    extends to an $\Aut(G)$-action, and so to a $\Gamma$-action.

    Now note that since $A$ is conjugate to a matrix in $\OO(m)$, there exists an $A$-invariant inner product $\langle {-},{-} \rangle$ on $\RR^m$.  Let $\zeta$ be the composite $\Gamma \to \Out(G) \to \calc$, where the second map is the one appearing (implicitly) in Proposition~\ref{prop:stab_S}, and note that $\zeta(G)$ is trivial, implying that $\zeta(\Gamma)$ is finite.  Define an inner product $\langle {-},{-} \rangle'$ on $\RR^m$ by setting
    \[
    \langle x,y \rangle' \coloneqq \sum_{(M,\varepsilon) \in \zeta(\Gamma)} \langle Mx,My \rangle,
    \]
    and note that this inner product is $\zeta(\Gamma)$-invariant.  Moreover, for any $(M,\varepsilon) \in \calc$ we have $MAM^{-1}=A^{\varepsilon}$, implying that
    \[
    \langle Ax,Ay \rangle' = \sum_{(M,\varepsilon) \in \zeta(\Gamma)} \langle MAx,MAy \rangle = \sum_{(M,\varepsilon) \in \zeta(\Gamma)} \langle A^{\varepsilon}Mx,A^{\varepsilon}My \rangle = \langle x,y \rangle'
    \]
    for all $x,y \in \RR^m$, and so the inner product $\langle {-},{-} \rangle'$ is also $A$-invariant.
    
    We set $\EE^m$ to be the $m$-dimensional Euclidean space defined by the inner product $\langle {-},{-} \rangle'$.  Since $\langle {-},{-} \rangle'$ is $A$-invariant, $G$ acts on $\EE^m$ by isometries, with $H$ and $t$ acting as translations and multiplication by $A$, respectively.  We furthermore have an action of $\cala \coloneqq \delta(H \rtimes \zeta(\Gamma))$ on $\EE^m$ by isometries (with $\delta$ as defined before Proposition~\ref{prop:stab_S}), defined by letting each $\beta_h$ act trivially and each $\gamma_{M,\varepsilon}$ act as multiplication by $M$.  By construction and Proposition~\ref{prop:stab_S}, we then have $\Gamma \leqslant \cala$.  Since $\cala$ acts on $\EE^m$ by isometries, so does $\Gamma$.

    Now consider the diagonal action of $\Gamma$ on $\EE^m \times S$.  The action of the finite-index subgroup $G$ of $\Gamma$ is geometric (see \cite[Theorem~7.2]{LearyMinasyan2021}), implying that so is the action of $\Gamma$, and so $\Gamma$ is $\CAT(0)$.
\end{proof}


\subsection{The non-rigid case: new reduced trees and their stabilisers}

In the non-rigid case (which will turn out to correspond to $\Out(G)$ not being finitely generated) we need to understand the deformation space.  We do this
in the sequel by constructing a further deformation retract of its spine, this culminates in \Cref{thmx:deformation_retract}. Before this we content ourselves by exhibiting an infinite family of reduced trees of the form in \Cref{fig:BAD_trees}. Here (and for the rest of the section) we adopt the convention that $AL <L$: this is no loss of generality, since if $L <AL$ then swapping $t$ and $t^{-1}$ in the presentation will also swap the inclusion. Notice that taking \Cref{fig:BAD_trees} with $n=0$ gives a tree which is not reduced.  Indeed, the separating edge can be collapsed, giving back \Cref{fig:gog_for_LMAL}.

The remainder of this section is dedicated to computing the $\Out(G)$ stabiliser of the tree in \Cref{fig:BAD_trees}.  The attentive reader may have observed that to meaningfully understand $\Out(G)$ we would need to understand the $\Out(G)$ stabiliser of every reduced tree in $\cald$.  However, it turns that it is enough to consider only the trees corresponding to the graphs of groups in \Cref{fig:gog_for_LMAL,fig:BAD_trees}, this is the content of \Cref{sec:deform}. A diligent reader could verify that $T_n$ can be transformed into $T_{n+1}$ by sliding the separating edge around the loop, and that $T_1$ can be transformed into the graph of groups in \Cref{fig:gog_for_LMAL} by an $\cala^{-1}$ move: these exercises may be viewed as warm-ups to the proof in the next section.

\begin{figure}[ht]
\centering
\begin{tikzpicture}
\draw [thick, middlearrow={.5}{0}{>}{$t~$}{left, rotate=90}] (0,0) arc [start angle =0, end angle =360, x radius =1.3cm, y radius=1.3cm] (0,0);
\draw [thick] (0,0) -- (2,0) node [midway,below] {$L$};
\draw [fill=blue] (0,0) circle (2pt) node [left] {$t^{-n}Lt^n$};
\draw [fill=yellow] (2,0) circle (2pt) node [right] {$H$};
\node at (0.3,-1) {$t^{-n}Lt^n$} ;
\node at (0.7,1) {$t^{1-n}Lt^{n-1}$} ;
\end{tikzpicture}
\caption{A graph of groups corresponding to some non rigid tree $T_n$ in $\cald$ for $n\in\ZZ_{>0}$.}
    \label{fig:BAD_trees}
\end{figure}
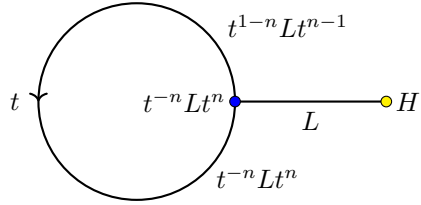

Let $L_0 = L$ and for any $n > 0$, denote by $L_n$ the subgroup $t^{-n}Lt^n$, the vertex group attached to the valence 3 vertex of the graph of groups in \Cref{fig:BAD_trees}. Note that $L_n < L_{n+1}$ for $n \geq 0$, and denote $\bigcup_{n=0}^\infty L_n$ by $L_\infty$. 

\begin{lemma}
\label{find_reps_of_stab_T}
    After composing with an inner automorphism, any element of $\Stab_{\Aut(G)}(T_n)$ preserves $H$ and $L$ and sends $t$ to $\ell t$ for some element $\ell$ of $L_n$. Conversely, any automorphism preserving $H$ and $L$ and sending $t$ to $\ell t$ (with $\ell \in L_n$) is contained in $\Stab_{\Aut(G)}(T_n)$.
\end{lemma}

\begin{proof}
    Since the local pictures at vertices stabilised by a conjugate of $H$ or of $L_n$ are different (e.g.\ only the latter is adjacent to vertices in the same orbit), this partition is preserved by any $\varphi \in \Stab_{\Aut(G)}(T_n)$. Since $H$ properly contains its adjacent edge group, it is self-normalising and there is a unique vertex $v$ stabilised by $H$; composing with an inner automorphism we can assume $\varphi$ fixes $v$ and so preserves $H$.

    There is a unique orbit of edges incident at $v$, and so, composing by a further inner automorphism (by an element of $H$) we can assume that the vertex $u$ stabilised by $L_n$ is also fixed, and so $L_n$ is preserved as well as the edge stabiliser $L$. There are two further $L_n$-orbits of incident edges at $u$, exactly one of which consists of a single edge (joining $u$ to $t^{-1}u$, the latter vertex stabilised by $t^{-1}L_nt=t^{-1}t^{-n}Lt^nt=L_{n+1}$.  The vertex $t^{-1}u$ is therefore also stabilised by $\varphi$, and in particular we have that $t^{-1}u=\varphi(t^{-1}u)=\varphi(t^{-1}) \cdot \varphi(u)$, which implies that $\varphi(t)$ is an element of $\Stab(u) t = L_n t$.

    We prove the converse by building an automorphism of $T_n$ for every element of this form. First note that any automorphism of the given form also preserves every $L_k = t^{-k}Lt^k$, since this is sent to $(\ell t)^{-k} L (\ell t)^k$, and expanding this to $t^{-1} \ell^{-1} \cdots t^{-1}\ell^{-1} L \ell t \cdots \ell t$ we see that since $L < L_k$ (and this is an abelian group) the innermost conjugation $\ell^{-1} L \ell$ reduces to $L$, allowing us to consider $(\ell t)^{-(n-1)} t^{-1} L t (\ell t)^{(n-1)}$, where the middle term is (by definition) equal to $L_1$. Repeating the argument by induction we see that $L_n$ is preserved.
    
    Continue to let $v$ be the unique vertex stabilised by $H$ and $u$ the adjacent vertex stabilised by $L_n$. The vertices of $T_n$ are given by translates $gv$ and $gu$ of these two vertices, and $\varphi$ induces a map on the vertex set sending $gv$ to $\varphi(g) v$ and $gu$ to $\varphi(g) u$. These maps are well defined since if $g_1v=g_2v$ then $g_1^{-1}g_2 \in G_v = H$, and $H$ is preserved by $\varphi$, and similarly for $u$ and $L_n$.

    We now need to show that this map is a graph automorphism, i.e.\ that it preserves adjacency. There are two orbits of edges in $T_n$: one represented by the edge joining $u$ and $v$, and the other by the edge joining $u$ and $tu$ (this last vertex stabilised by $tL_nt^{-1}=L_{n-1}$). Now, $g_1v$ and $g_2u$ are adjacent if and only if $g_1^{-1}g_2$ is contained in the edge stabiliser $L$; again, this condition is preserved by $\varphi$. Finally $g_1u$ and $g_2tu$ are adjacent if and only if $g_1^{-1}g_2$ is contained in the edge stabiliser $L_{n-1}$, which is once again preserved by $\varphi$.
\end{proof}

Note that \Cref{prop:centraliser_to_aut_nonmet} made no claims about the tree under consideration, so we will be able to use its conclusions as is. We now consider the automorphisms that fix $H$ pointwise.

Observe that there is a map $\beta_n \colon L_n \to \Aut(G)$ where $\beta_{n,\ell} \coloneqq \beta_n(\ell)$ fixes $H$ pointwise and sends $t$ to $\ell t$. (Note that $\beta_{n,\ell}$ is not an automorphism, or even an endomorphism, without the assumption that $AL$ is contained in $L$.) These maps are compatible with the inclusion $L_n \rightarrowtail L_{n+1}$ (viewed as subgroups of $G$) in the sense that $\beta_{n,\ell} = \beta_{n+1,\ell}$.

\begin{prop}
\label{prop:Ln_to_aut_nonmet}
    Suppose $G=\LM(A,L)$ is neither metabelian nor rigid. Then the map $\beta_n\colon L_n \to \Aut(G)$ is injective. Its image agrees with $\{\varphi \in \Aut(G) \mid \varphi(x) = x \text{ for } x \in H, \varphi(t)=\ell t \text{ for some } \ell \in L_n \}$ and $\beta_{n,\ell}$ is inner if and only if $\ell$ lies in $(I-A)L$. In particular, the natural projection $\overline{\beta}_n\colon L_n \to \Out(G)$ has image isomorphic to $L_n/(I-A)L$.
\end{prop}

\begin{proof}
    For injectivity, note that if $\beta_{n,\ell}$ is the identity then $\ell = \beta_{n,\ell}(t)t^{-1} = 1$. For the claim about the image, note that any $\ell \in L_n$ gives rise to an automorphism $\beta_{n,\ell}$.

    Now suppose that $\beta_{n,\ell}$ is an inner automorphism: that is, it agrees with conjugation by some element $g$ of $G$. In particular, we have that $\ell t = gtg^{-1}$, so $\ell = gtg^{-1}t^{-1}$. Write this as $t^{-n}\ell't^n = gtg^{-1}t^{-1}$, giving $\ell' = t^n g t g^{-1} t^{-1-n}$ as an element of $L < H$. Applying Britton's lemma, the only reduction possible is at $t g^{-1} t^{-1}$, and so $g$ is contained in $L$. (It follows that $tgt^{-1}$ is contained in $AL < L$, and so we learn nothing new in subsequent applications of Britton's lemma.)

    So $\ell = g(tg^{-1}t^{-1})$, with $g$ an element of $L$. Via our standard abuse of notation (and switching to additive notation), this implies that $\ell \in (I-A)L$. Conversely, any element $g \in L$ gives rise to an inner automorphism $\beta_{n,\ell}$ by setting $\ell = gtg^{-1}t^{-1}$. The final claim follows immediately.
\end{proof}

Note that, since $AL < L$, we have $\left| \det A \right| > 1$, implying that $\calc^- = \varnothing$ and hence $\calc = \calc^+$.  As before, the homomorphisms $\beta_n\colon L_n \to \Aut(G)$ and $\gamma \colon \calc^+ \to \Aut(G)$ extend to a homomorphism $\delta_n\colon L_n \rtimes \calc^+ \to \Aut(G)$. Denote by $\overline{\calb}_n$ the image of $L_n$ under $\overline{\beta}_n$.

\begin{prop}\label{prop:stab_Tn}
    The stabiliser of $T_n$ under the action of $\Out(G)$ is isomorphic to $\overline \calb_n \rtimes \calc^+$, with the isomorphism induced by $\delta_n \colon L_n \rtimes \calc^+ \to \Aut(G)$. 
\end{prop}

\begin{proof}
    Using \Cref{find_reps_of_stab_T} we see that any element of $\Stab_{\Out(G)}(T_n)$ has a representative $\varphi$ that preserves $H$ and sends $t$ to $\ell t$, for some $\ell \in L_n$. So there is a quotient map to $\GL_m(\ZZ)$ given by considering the restriction $\varphi|_H$ (as before, this is independent of the choices made in the construction of $\varphi$ since $H$ is self-normalising).

    By \Cref{prop:centraliser_to_aut_nonmet}, the image agrees with $\calc^+$, and there is a section given by $\gamma$. We now consider the kernel, which consists of elements fixing $H$ pointwise and sending $t$ to $\ell t$ for some $\ell \in L_n$. By \Cref{prop:Ln_to_aut_nonmet}, this subgroup is exactly $\overline{\calb}_n$. 
\end{proof}

\begin{remarks} \label{rem:inters} We record two remarks:
    \begin{enumerate}
        \item There are natural inclusions $\Stab_{\Out(G)}(T_n)\rightarrowtail \Stab_{\Out(G)}(T_{n+1})$ for any $n\geq 1$. 
        \item We have that $\Stab_{\Out(G)}(S)\cap \Stab_{\Out(G)}(T_n)$ is isomorphic to $L/(I-A)L \rtimes \calc^+$ for any $n\geq 1$.
    \end{enumerate}
\end{remarks}

\section{The deformation space in the non-rigid case}\label{sec:deform}
The aim of this section is to prove \Cref{thmx:deformation_retract}. We do this in stages: first we pass to a subcomplex $W$ of the spine of $\cald$, then we ``barycentrically un-subdivide'' this complex to give $\widehat{W}$, introduce a new (one-dimensional) complex spanned by certain vertices of $\widehat{W}$, and finally construct a sequence of elementary collapses enabling us to equivariantly deformation retract $\widehat{W}$ to this subcomplex.

\begin{table}[ht]
    \centering
    \begin{tabular}{|c|c|}
       \hline
        $G=\LM(A,L)$ & $A\in \GL_m(\QQ)$ and $AL < L < \ZZ^m$ \\
        $\Phi\colon G \to \AGL_m(\QQ)$ & The map sending $H \xrightarrow{\cong} \ZZ^m$ and $t \mapsto A$ \\
        $L_n$ & The subgroup $t^{-n}Lt^n$, $n \in \ZZ_{\geq 0}$ \\
        $L_\infty$ & The subgroup $\bigcup_{n=0}^\infty L_n$ \\
        Type $\cals$ & The $G$-trees with one edge orbit, see Figure~\ref{sfig:gog-small} \\
        Type $\calt(\Lambda_1,\ldots,\Lambda_k)$ & The $G$-trees with $k+1 \geq 2$ edge orbits, see Figure~\ref{sfig:gog-standard} \\
        $\cald$ & The deformation space of $G$-trees containing a tree of type $\cals$ \\
        $W$ & The subcomplex of $\cald$ spanned by trees with all edges surviving \\
        $\widehat{W}$ & A complex with $W$ as a barycentric subdivision (Proposition~\ref{prop:What}) \\
        $\widehat{X}$ & The subcomplex of $\widehat{W}$: trees of type $\cals$ or $\calt(\Phi(L_n))$ for $n > 0$ \\
        $\prec$ & The partial order: $T_1 \prec T_2$ if $T_2$ collapses to $T_1$ \\
        $\sqsubsetdot$ & The relation: $T_1 \sqsubsetdot T_2$ if an incr.\ induction or $\scra$-move maps $T_1$ to $T_2$ \\
        $\sqsubset$ & The transitive closure of $\sqsubsetdot$ \\
        \hline
    \end{tabular}
    \caption{Table of notation for \Cref{sec:deform}.}
    \label{tab:notation_deform}
\end{table}

\subsection{Surviving edges and the subcomplex \texorpdfstring{$W$}{W}}\label{sec:survivor}

Throughout \Cref{sec:survivor} we assume that $G$ is a finitely generated group and $T$ is a minimal irreducible $G$-tree with finitely many orbits of edges.  Note that irreducible here means that $T$ contains two axes for elements in $G$ that intersect at most in a finite segment.

Following \cite{Clay09} (in turn following \cite{GuirardelLevitt2007_defom}), suppose $T'$ is a reduced tree in the same deformation space as $T$.   We say that an edge $e$ in $T$ is \emph{surviving in $T'$} if $T$ collapses to $T'$ and $e$ is not collapsed.  We say $e$ is \emph{surviving} if it is surviving in $T'$ for some reduced tree $T'$ in the same deformation space as $T$.  

In a deformation space $\cald$, let $W$ be the subcomplex of the spine spanned by the set of trees all of whose edges are surviving.   The usefulness of this subcomplex comes from the following theorem of Clay:

\begin{thm}[Clay {\cite[Section 1.2]{Clay09}}]
\label{thm:dr-to-W}
    There is an equivariant deformation retraction from the spine of $\cald$ to $W$. In particular, $W$ is contractible and if $\cald$ is $\Out(G)$-invariant then so is $W$.
\end{thm}

As Clay observes, this $W$ recovers the ``reduced spine'' of Culler--Vogtmann outer space. Next we provide a lemma to show that a tree in $W$ is uniquely determined by the reduced trees it collapses to. Observe that in a reduced tree all edges are (by definition) surviving, and so all reduced trees are contained in $W$.

Two trees are \emph{compatible} if there is a third tree admitting collapse maps to both of them. Recall that, given two $G$-trees $T_1$ and $T_2$, we write $T_1 \preceq T_2$ if $T_2$ collapses to $T_1$.

\begin{lemma}
    \label{lem:coll-same-red}
    Given $G$-trees $T$ in $W$ and $T'$ in the spine of $\cald$, we have $T \preceq T'$ if and only if any reduced tree $S$ that satisfies $S \preceq T$ also satisfies $S \preceq T'$.

    Conversely, suppose $S_1,\dots,S_k$ are pairwise compatible reduced trees. Then there is a tree $T$ in $W$ such that $S_i \preceq T$ for each $i$.
\end{lemma}

\begin{proof}
    By transitivity of $\preceq$, if $T \preceq T'$ then every $S$ (reduced or not) satisfying $S \preceq T$ also satisfies $S \preceq T'$. We now show the converse: suppose that $T$ and $T'$ are elements of $W$, and let $S_1, \dots, S_k$ be the reduced trees that $T$ collapses to (note that this is a finite set since $T$ has finitely many edge orbits), and suppose that for every $i$ we have that $S_i \preceq T'$.

    The trees $S_i$ are pairwise compatible (since $T$ collapses to all of them) and so by \cite[Proposition A.26]{GuirardelLevitt2017_JSJbook} there is a tree $S_1 \vee \dots \vee S_k$ satisfying: \begin{enumerate}
        \item  (Part 1 \emph{loc.~cit.}) For each $i$, $S_i \preceq S_1 \vee \dots \vee S_k$, and $S_1 \vee \dots \vee S_k$ is minimal for this property: if $T$ is some other tree with $S_i \preceq T$ for all $S_i$ then $S_1 \vee \dots \vee S_k \preceq T$. In particular, this is true for $T$ and $T'$.
        \item The tree $S_1 \vee \dots \vee S_k$ is also in the deformation space $\cald$. (This follows from the second sentence of Part 3 of the reference, since all $S_i$ are in $W$---and so in $\cald$---and so dominate each other, meaning $S_1 \vee \dots \vee S_k$ is also in $\cald$.)
    \end{enumerate}

    We now show that if $T \neq S_1 \vee \dots \vee S_k$, then $T$ is not in $W$. Suppose that $S_1 \vee \dots \vee S_k \prec T$, so that there is an edge $e$ whose orbit is collapsed in the map from $T$ to $S_1 \vee \dots \vee S_k$. But then $e$ cannot be a surviving edge, for there are subsequent collapses to all of the reduced trees $S_1, \dots ,S_k$, and these are all of the reduced trees that can be obtained by collapsing $T$, implying that $T$ is not in $W$. Taking the contrapositive, we see that since $T$ is in $W$, $T = S_1 \vee \dots \vee S_k$ and so $T \preceq T'$. 

    For the final statement, consider the minimality of $T = S_1 \vee \dots \vee S_k$. This implies that if $T'$ is obtained by collapsing one of its edges $e$, there is some $S_i$ so that $S_i \npreceq T'$. But then $e$ must be surviving in $S_i$. So all edges of the tree $T$ are surviving, thus it is contained in $W$.
\end{proof}

A priori it is not true that the surjective function which sends a set of pairwise compatible reduced trees to the minimal common refinement in the spine of $\cald$ is a bijection.  A key step in our subsequent work is to prove a comparable statement for Leary--Minasyan groups (see Lemma~\ref{lem:Psi-surjective}).

Clay \cite[Proposition~1.13]{Clay09} (generalising \cite[Corollary 7.5(1)]{GuirardelLevitt2007_defom}) characterises the surviving edges in a tree $T$ as those which are contained in a \emph{shelter}, see also \cite[Figure 1]{Clay09}.

\begin{defn}~\cite[Definition 1.12]{Clay09}
    Let $T$ be a $G$-tree. A segment $\gamma$ of $T$ is a \emph{shelter} if its projection to $G \backslash T$ is one of the following: \begin{enumerate}[(S1)]
        \item a segment where edge inclusions at the endpoints are not surjections, but are at every interior vertex;
        \item a cycle where every edge inclusion is a surjection, except perhaps at one vertex;
        \item a cycle where there is a coherent choice of orientation so that $\iota(G_e)=G_{\iota(e)}$ for every edge $e$ in that orientation.
    \end{enumerate}
\end{defn}

\begin{prop}~\cite[Proposition 1.13]{Clay09}
    If $T$ is in $W$, then every edge of $T$ is contained in a shelter.
\end{prop}

\subsection{Understanding \texorpdfstring{$W$}{W} for the non-metabelian Leary--Minasyan groups}
In this subsection, we let $G$ be a non-metabelian Leary--Minasyan group with $AL < L$ and work to characterise the $G$-trees corresponding to vertices of $W$.  We first observe some properties of the fixed points of $H$ and $L$ in any tree of $W$.

\begin{lemma} \label{lem:find_H_and_L}Suppose $T$ is a $G$-tree where every edge is surviving. Then there is a unique vertex of $T$ stabilised by $H$, and there is an $H$-orbit of edges adjacent to it stabilised by $L$.    
\end{lemma}
\begin{proof}
    Recall that $H$ represents the unique conjugacy class of maximal elliptic subgroups in any tree in $\cald$. Consider the fixed subtree of $H$ in any such tree: the stabiliser of every edge and vertex must be exactly $H$ by maximality. If there is no vertex $v \in \Fix(H)$ and non-trivial $g \in G$ with $gv$ distinct but contained in $\Fix(H)$, then it follows that collapsing $\Fix(H)$ to a point does not change the deformation space, and so in every reduced tree $\Fix(H)$ must be a single vertex. If $T$ is any tree in $\cald$ then no surviving edge can be stabilised by $H$ (because no edges are stabilised by $H$ in the reduced trees) and so in a tree where every edge is surviving $\Fix(H)$ is again just a single vertex. 

    Suppose there is some $v$ and some $g$ so that $v$ and $gv$ are both contained in $\Fix(H)$. 
    Both $v$ and $gv$ must be stabilised by $H$ (by maximality), and so $g$ normalises $H$. But $H$ is self-normalising, since it properly contains the two incident edge groups in $S$, the tree of Figure~\ref{fig:gog_for_LMAL} which corresponds to the original presentation of $G$ as a Leary--Minasyan group. So $g$ must be an element of $H$, and so $gv=v$, and we are in the situation of the previous paragraph.

    For the second claim, we will compare the edge stabilisers in $T$ with those of $S$, the tree of Figure~\ref{fig:gog_for_LMAL}. Since $H$ is abelian, its conjugation action on its subgroups is trivial and so the edge stabilisers are constant across an $H$-orbit of adjacent edges. Take $v$ to be the vertex stabilised by $H$ and $t$ to be the original stable letter for the presentation of $G$: in particular we have $L=H \cap t^{-1}Ht$. Consider the arc $[v,t^{-1}v] \subset T$, which must be stabilised by $L$, and let $e=[v,u]$ be the first edge in this arc. Now we have that $L \leq G_e$, and we aim to show the other inclusion.

    Working in $S$, let $v'$ be the vertex stabilised by $H$, and consider the fixed subtree associated to $G_u$. Suppose $v'$ is not contained in this subtree so there is a nearest point $u'$ to $v'$, and we consider the arc $[v',u'] \subset S$, which is stabilised by $G_e = H \cap G_u$. The first edge in this arc is stabilised by either $L$ or $AL$, giving either $G_e \leq L$ or $G_e \leq AL$. The second case is absurd, since $AL < L \leq G_e$, and so we have $G_e=L$.

    Finally suppose $v'$ is contained in the fixed subtree associated to $G_u$: this implies that $G_u < H$ and $G_e = G_u$. Since every edge is surviving, we can consider instead a reduced tree $T$ where the first edge $e$ in $[v,t^{-1}v]$ survives. In such a tree the inclusion of $G_e$ into $G_u$, the other endpoint, must be proper: otherwise we would have that $H$ is conjugate to a subgroup (and hence an overgroup) of itself, again contradicting the maximal ellipticity of $H$. So the argument of the previous paragraph applies to the reduced tree, and again we have $G_e=L$.
\end{proof}

We now fix an affine representation $\Phi$ of $G$, as follows. Recalling that $\AGL_m(\QQ) \cong \QQ^m \rtimes \GL_m(\QQ)$, define $\Phi: G \to \AGL_m(\QQ)$ on generators as $\Phi(H) = \ZZ^m < \QQ^m$ (that is, according to our fixed isomorphism $H \to \ZZ^m$, and taking the standard embedding of $\ZZ^m$ into $\QQ^m$) and $\Phi(t)=A$. This map cannot be injective (for instance because this would imply residual finiteness of $G$), but since it is injective on $H$ it must also be injective on every subgroup of every conjugate of $H$, i.e. on each elliptic subgroup of $G$. In particular it is injective on each $L_i$ and so their union $L_\infty$. Each subgroup $L_i$ is sent to a lattice $\Phi(L_i)$ in $\QQ^m$, but their union $\bigcup_{i=0}^\infty \Phi(L_i)$ = $\Phi(L_\infty)$ cannot be a lattice (because any lattice $\Lambda < \Phi(L_\infty)$ is finitely generated and so must be contained in some $\Phi(L_i)$). 
Also notice that there is a further quotient to $\GL_m(\QQ)$, where the image of $G$ is the subgroup $\langle A \rangle$. Since $\QQ^m$ (viewed as the normal subgroup in the semidirect product structure) is abelian, the conjugation action of any $\Phi(g)$ on any lattice (or subgroup) $\Lambda \leq \QQ^m$ agrees with the application of some power of $A$, 
with the power determined by the exponent sum of $t$ in $g$. Finally note that since $\Phi(L_\infty)$ is the union $\bigcup_{i=0}^\infty A^{-i}(\Phi(L))$, it is invariant under all powers of $A$. In particular, if $\Lambda < \Phi(L_\infty)$, then also $A^{-1}\Lambda < \Phi(L_\infty)$.

Note that in what follows $\Lambda_i$ (as well as $\Phi(L_j)$) will always be lattices in $\AGL_m(\QQ)$, whereas $K_i$ (as well as $L_j$) will always be subgroups of $G$.

\begin{defn} 
We define some types of trees we will make frequent use of:
\begin{itemize}
    \item A $G$-tree has \emph{Type $\cals$} if it has a fundamental domain  consisting of a single edge, with endpoints stabilised by $H$ and $\widehat{t}H\widehat{t}^{-1}$, where $\widehat{t} \in G$ conjugates a finite-index subgroup of $H$ to a proper subgroup of itself.
    \item Suppose $\Phi(L) \leq \Lambda_1 < \dots <\Lambda_k \leq A^{-1}\Lambda_1$ are lattices in $\Phi(L_\infty)$. A $G$-tree has \emph{Type $\calt(\Lambda_1, \dots, \Lambda_k)$} if it has a fundamental domain consisting of a segment whose vertices are stabilised (in order) by $H, K_1, \dots, K_k, \widehat{t}^{-1}K_1\widehat{t}$, where $K_1 < \dots < K_k \leq \widehat{t}^{-1}K_1\widehat{t}$ and $\Phi(K_i)= \Lambda_i$ for all $i$, and where the vertices stabilised by $H,K_1,\dots,K_k$ are in different $G$-orbits. 
    
    \end{itemize}
In both cases, we call $\widehat{t}$ a \emph{stable letter} of the corresponding $G$-tree.
\end{defn}

Quotient graphs of trees of these types (with the given fundamental domains) are shown in \Cref{fig:gog};  as implied by the notation, the trees $S$ of \Cref{fig:gog_for_LMAL} and $T_n$ of \Cref{fig:BAD_trees} are of types $\cals$ and $\calt(\Phi(L_n))$ respectively. Note that for a chain of lattices as in the definition of a tree of the second type, inclusion of all the lattices in $\Phi(L_\infty)$ is automatic as soon as we have $\Lambda_1 < \Phi(L_\infty)$ (by $A^{-1}$-invariance). 

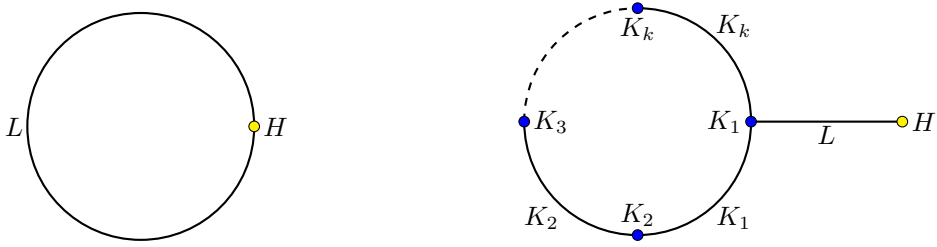
\begin{figure}[ht]
\begin{subfigure}[t]{0.4\textwidth}
\centering
\begin{tikzpicture}
\draw [thick] (-1.5,0) circle (1.5);
\draw [fill=yellow] (0,0) circle (2pt) node [right] {$H$} (-3,0) node [left,xshift=2pt] {$L$};
\end{tikzpicture}
\caption{The graph of groups corresponding to the $G$-tree of type $\cals$; here the edge group homomorphisms are given by the subgroup inclusion and the map $\tau\colon L \to H$ satisfying $\Phi(\tau(h)) = A\Phi(h)$ for $h \in L$.}
\label{sfig:gog-small}
\end{subfigure}\hfill
\begin{subfigure}[t]{0.55\textwidth}
\centering
\begin{tikzpicture}
\draw [thick] (-1.5,1.5) arc (90:0:1.5) node [pos=0.5,xshift=6pt,yshift=6pt] {$K_k$};
\draw [thick] (0,0) arc (0:-180:1.5) node [pos=0.25,xshift=6pt,yshift=-6pt] {$K_1$} node [pos=0.75,xshift=-6pt,yshift=-6pt] {$K_2$};
\draw [thick,dashed] (-3,0) arc (180:90:1.5);
\draw [thick] (0,0) -- (2,0) node [midway,below,yshift=2pt] {$L$};
\draw [fill=blue] (0,0) circle (2pt) node [left] {$K_1$} (-1.5,-1.5) circle (2pt) node [above] {$K_2$} (-3,0) circle (2pt) node [right] {$K_3$} (-1.5,1.5) circle (2pt) node [below] {$K_k$};
\draw [fill=yellow] (2,0) circle (2pt) node [right] {$H$};
\end{tikzpicture}
\caption{The graph of groups corresponding to the $G$-tree of type $\calt(\Lambda_1,\dots,\Lambda_k)$, where $\Lambda_i = \Phi(K_i)$; here $L \leq K_1 < \dots < K_k$, and the edge group homomorphisms are all subgroup inclusions apart from $\tau\colon K_k \to K_1$, which satisfies $\Phi(\tau(h)) = A\Phi(h)$ for $h \in K_k$.}
\label{sfig:gog-standard}
\end{subfigure}
\caption{Some graphs of groups for $G$; in both cases, the description of the map $\tau$ follows from Lemma~\ref{lem:action-stable-letter}.}
\label{fig:gog}
\end{figure}

\begin{lemma} \label{lem:action-stable-letter}
    Let $T$ be a $G$-tree of type $\cals$ or $\calt(\Lambda_1, \dots, \Lambda_k)$ and $\widehat{t}$ its stable letter. Then $\Phi(\widehat{t}k\widehat{t}^{-1}) = A \Phi(k)$ for all $k \in K$, where $K$ is any elliptic subgroup of $G$.
\end{lemma}

\begin{proof}
    Consider $\Theta\colon G \to \pi_1(\Gamma) \cong \ZZ$, where $\Gamma$ is (the underlying graph of) $G \backslash T$.  Then $\Theta(\widehat{t})$ is a generator of $\ZZ$ by construction.  On the other hand, $\Theta$ is just the quotient map by the normal closure of the elliptic elements of $G$ acting on $T$.  Thus, passing to $\AGL_m(\QQ)$ this corresponds to taking the quotient by the translation subgroup.  That is, for 
    $$\Xi\colon \Phi(G) \to \Phi(G) / (\Phi(G) \cap \QQ^m) = \langle A \rangle$$ 
    we have $\Xi \circ \Phi = \Psi \circ \Theta$ for an isomorphism $\Psi \colon \pi_1(\Gamma) \to \langle A \rangle$. So it follows that the conjugation action of $\Phi(\widehat{t})$ on $\QQ^m$ must be either $A$ or $A^{-1}$.  But $\left| \det A \right| > 1$, hence the latter case is impossible since $\widehat{t}$ conjugates an elliptic subgroup $K' < G$ to a proper subgroup of itself and $\Phi$ is injective on $K'$, where $K'$ is a finite-index subgroup of $H$ if $T$ has type $\cals$ or a subgroup mapped by $\Phi$ isomorphically to $\Lambda_1$ if $T$ has type $\calt(\Lambda_1,\ldots,\Lambda_k)$.
\end{proof}

\begin{lemma}
Let $T$ be a tree either of type $\cals$ or of type $\calt(\Lambda_1, \dots, \Lambda_k)$.
\label{lem:same-type-is-iso}\begin{enumerate}
    \item Let $v$ be the unique vertex of $T$ stabilised by $H$. If $T$ is of type $\calt(\Lambda_1, \dots, \Lambda_k)$, then every edge adjacent to $v$ is stabilised by $L$; if $T$ has type $\cals$ then one $H$-orbit of edges is stabilised (pointwise) by $L$ and the other by $AL$. 
    \item Suppose $T'$ is a tree of the same type as $T$. Then there is an automorphism $\psi$ of $G$ so that $T \cdot \psi$ is equivariantly isometric to $T'$.
\end{enumerate}
\end{lemma}

\begin{proof}

    If $T$ was of type $\calt(\Lambda_1, \dots, \Lambda_k)$, then there is a single $H$-orbit of edges adjacent to $v$, and by Lemma~\ref{lem:find_H_and_L} they are pointwise stabilised by $L$, proving the first claim in this case.
    
    If $T$ was of type $\cals$, then there are two $H$-orbits of edges adjacent to $v$ and Lemma~\ref{lem:find_H_and_L} shows that one of these orbits is pointwise stabilised by $L$. Since $\Phi$ is injective on $H$ and $\Phi(\widehat{t})$ acts as $A$ on $H$ by Lemma~\ref{lem:action-stable-letter}, it follows that the other $H$-orbit of edges is pointwise stabilised by either $AL$ or $A^{-1}L$.  However, the latter case is impossible, since this would result in an isomorphism $\LM(A,L) = G \cong \LM(A,A^{-1}L)$, which is impossible by \cite[Theorem~1.1]{Valiunas2022} (for instance, since $L$ and $A^{-1}L$ have different indices in $H$ whenever $A^{-1}L \leq H$).  This proves the first claim.  

    We now consider the second claim; in fact we will show that $G \backslash T$ and $G \backslash T'$ are isomorphic as graphs of groups, which will imply following \cite[Section 2]{Bass1993} (precisely combining Definition 2.12, \S2.11, and Proposition 2.4) that there is an isomorphism between their fundamental groups and an equivariant (with respect to this isomorphism) isometry between the trees. Since both trees were $G$-trees to begin with, composing with isomorphisms coming from the fundamental theorem of Bass--Serre theory gives us the desired automorphism (and similarly for the equivariant isometry).

    Suppose that $T$ and $T'$ are of type $\cals$. Then in each tree there is a vertex stabilised by $H$, and it follows from the first claim that the two $H$-orbits of incident edges are stabilised by $L$ and $AL$ respectively, with the two edge group homomorphisms $L \to H$ given by the inclusion and the matrix $A$. This means that there is a graph of groups isomorphism between their quotient graphs where all edge and vertex maps are the identity. 

    Now suppose that $T$ and $T'$ are of type $\calt(\Lambda_1, \dots, \Lambda_k)$. For the leaf vertex in the graphs of groups and the incident edges, the first claim implies that the vertex group is $H$ and the edge group is $L$, so we can simply take the corresponding vertex and edge maps to be the identity. For the vertices and edges around the loop (including the inclusion of the leaf edge group), note that $\Phi$ is injective on elliptic subgroups, and in particular on $\widehat{t}^{-1}K_1\widehat{t}$ and $({\widehat{t}}')^{-1}K_1'{\widehat{t}}'$, where $\widehat{t}$ and ${\widehat{t}}'$ are the stable letters of $T$ and $T'$, respectively, and where $K_1$ and $K_1'$ are the corresponding stabilisers of a vertex adjacent to the one stabilised by $H$.  Using the map $\left(\Phi|_{({\widehat{t}}')^{-1}K_1'{\widehat{t}}'}\right)^{-1} \circ \Phi|_{\widehat{t}^{-1}K_1\widehat{t}}$ gives us an 
isomorphism between the vertex groups respecting all of the edge inclusions (where the last edge inclusion is respected by Lemma~\ref{lem:action-stable-letter}). So here too we have an isomorphism of the quotient graphs of groups. 
\end{proof}

We emphasise that the previous lemma does not show that every tree of a given type is equivariantly isometric, and indeed it is possible to construct trees where the translation length of $t$ is arbitrarily long: notice that Propositions~\ref{prop:stab_S} and~\ref{prop:stab_Tn} imply that there is an automorphism $\psi$ of $G$ fixing $H$ and sending $t$ to $gt^{-n}\ell t^{n+1}$, where $g \in H \smallsetminus L$ and $\ell \in L \smallsetminus AL$. These choices of $g$ and $\ell$ ensure that the image of $t$ is cyclically reduced as written, and it follows that the translation length of $t$ in $S \cdot \psi$ is $2n+1$. 

\begin{lemma} \label{lem:unique}
Let $T$ be a $G$-tree of type $\calt(\Lambda)$ for some sublattice $\Lambda < \Phi(L_\infty)$ with $\Phi(L) \leq \Lambda$.
\begin{enumerate}[label=\textup{(\roman*)}]
\item \label{it:ind-unique} Given a lattice $\Lambda'$ with $\Lambda < \Lambda' \leq A^{-1}\Lambda$, there is a unique $G$-tree $\ind{T}{\Lambda'}$ of type $\calt(\Lambda')$ obtained from $T$ using an increasing induction.
\item \label{it:slide-unique} There is a unique $G$-tree obtained by sliding the non-loop edge of $G \backslash T$ clockwise around the loop edge, and it coincides with the $G$-tree $\ind{T}{A^{-1}\Lambda}$ from part~\ref{it:ind-unique}.
\item \label{it:slide=ind} Suppose that $\Lambda \geq \Phi(L_1)$. Then sliding the non-loop edge of $G \backslash T$ anti-clockwise around the loop edge is equivalent to a decreasing induction move.
\item \label{it:ind-exp} Suppose that $\Lambda \leq \Phi(L_1)$, and let $S'$ be a $G$-tree (of type $\cals$) obtained from $T$ using an $\scra^{-1}$-move. Then there is a unique $G$-tree $T'$ that can be obtained both from $T$ using a decreasing induction and from $S'$ using an expansion move.
\item \label{it:ind-ind} Given lattices $\Lambda_1,\Lambda_2$ with $\Lambda < \Lambda_1 < \Lambda_2 \leq A^{-1}\Lambda$, we have $\ind{T}{\Lambda_2} = \ind{(\ind{T}{\Lambda_1})}{\Lambda_2}$.
\end{enumerate}
\end{lemma}

\begin{proof}
Note that the tree $T$ can be described as follows. It has two orbits of vertices, \emph{yellow} and \emph{blue}, where for our choice of fundamental domain, their representatives are stabilised by subgroups $H$ and $K$, with $\Phi(H) = \ZZ^m$ and $\Phi(K)=\Lambda$, respectively (see Figure~\ref{sfig:gog-standard}). Any yellow vertex has $[H:L]$ edges incident to it, all corresponding to the non-loop edge of $G \backslash T$. Any blue vertex is adjacent to $[K:L]$ \emph{rightward} edges (corresponding to the non-loop edge of $G \backslash T$), one \emph{downward} edge and $[\Lambda:A\Lambda] = \left|\det A\right|$ \emph{upward} edges (corresponding to the loop in $G \backslash T$). See Figure~\ref{fig:unique}, left. We have a similar description of a $G$-tree of type $\cals$: it has one orbit of vertices, and each vertex has $[H:L]$ \emph{downward} and $[H:\widehat{t}L\widehat{t}^{-1}] = \left| \det A \right| [H:L]$ \emph{upward} incident edges, where $\widehat{t}$ is a stable letter of the $G$-tree.

\begin{figure}[ht]
\centering
\begin{tikzpicture}
\begin{scope}[yscale=-1]
\draw [thick] (2,0) -- (0,0) -- (0,2) -- (0.5,2.1) (0,2) -- (0.5,1.9) (-1,0) -- (0,2) -- (0.1,1.5) (0.2,1.5) -- (0,2) -- (0,2.5) (0,0) -- (0.5,0.15) (1.85,0.45) -- (2,0) -- (2.4,0.3) (1.85,-0.45) -- (2,0) -- (2.4,-0.3) (-0.5,0.1) -- (-1,0) -- (-0.5,-0.1) (-1.15,-0.5) -- (-1,0) -- (-1.05,-0.5) (-0.95,-0.5) -- (-1,0) -- (-0.85,-0.5);
\foreach \x in {-1,0,1,2} {
	\draw [thick] (0,0) -- (\x,-2) -- (\x+0.5,-1.9) (\x,-2) -- (\x+0.5,-2.1) (\x-0.15,-2.5) -- (\x,-2) -- (\x-0.05,-2.5) (\x+0.05,-2.5) -- (\x,-2) -- (\x+0.15,-2.5);
	\draw [fill=blue] (\x,-2) circle (2pt);
}
\draw [fill=yellow] (2,0) circle (2pt);
\draw [fill=blue] (0,0) circle (2pt) (0,2) circle (2pt) (-1,0) circle (2pt);
\end{scope}

\begin{scope}[xshift=5cm,yshift=3cm,yscale=-1]
\draw [thick] (2,0) -- (0,0) -- (-0.25,1) -- (0,2) -- (0.5,2.1) (0,2) -- (0.5,1.9) (-1,0) -- (-0.25,1) (0.2,1.5) -- (0,2) -- (0,2.5) (0,0) -- (0.5,0.15) (1.85,0.45) -- (2,0) -- (2.4,0.3) (1.85,-0.45) -- (2,0) -- (2.4,-0.3) (-0.5,0.1) -- (-1,0) -- (-0.5,-0.1) (-1.1,-0.5) -- (-1,0) -- (-0.9,-0.5);
\draw [green!70!black,->] (-0.25,1) -- (-0.375,0.5);
\draw [green!70!black,->] (-0.25,-1) -- (-0.375,-1.5);
\draw [green!70!black,->] (0.75,-1) -- (1.05,-1.5);
\foreach \x/\y in {-1/-0.25,0/-0.25,1/0.75,2/0.75} {
	\draw [thick] (0,0) -- (\y,-1) -- (\x,-2) -- (\x+0.5,-1.9) (\x,-2) -- (\x+0.5,-2.1) (\x-0.1,-2.5) -- (\x,-2) -- (\x+0.1,-2.5);
	\draw [fill=blue] (\x,-2) circle (2pt);
	\draw [fill=green!70!black] (\y,-1) circle (2pt);
}
\draw [fill=yellow] (2,0) circle (2pt);
\draw [fill=blue] (0,0) circle (2pt) (0,2) circle (2pt) (-1,0) circle (2pt);
\draw [fill=green!70!black] (-0.25,1) circle (2pt);
\end{scope}

\begin{scope}[xshift=10cm,yshift=3cm,yscale=-1]
\draw [thick] (2,0) -- (-0.5,0) -- (0,2) -- (0.5,2.15) (0.5,2.05) -- (0,2) -- (0.5,1.95) (0,1.5) -- (0,2) -- (0.5,1.85) (0,2) -- (0.1,1.5) (0.2,1.5) -- (0,2) -- (0,2.5) (0,0.2) -- (-0.5,0) -- (0,0.1) (-0.5,0) -- (0,-0.1) (-0.75,-0.5) -- (-0.5,0) -- (-0.65,-0.5) (1.85,0.45) -- (2,0) -- (2.4,0.3) (1.85,-0.45) -- (2,0) -- (2.4,-0.3);
\foreach \x in {-0.5,1.5} {
	\draw [thick] (-0.5,0) -- (\x,-2) -- (\x+0.5,-1.85) (\x+0.5,-1.95) -- (\x,-2) -- (\x+0.5,-2.05) (\x,-2) -- (\x+0.5,-2.15) (\x-0.15,-2.5) -- (\x,-2) -- (\x-0.05,-2.5) (\x+0.05,-2.5) -- (\x,-2) -- (\x+0.15,-2.5);
	\draw [fill=blue] (\x,-2) circle (2pt);
}
\draw [fill=yellow] (2,0) circle (2pt);
\draw [fill=blue] (-0.5,0) circle (2pt) (0,2) circle (2pt);
\end{scope}

\begin{scope}[xshift=5cm,yshift=-3.5cm,yscale=-1]
\draw [thick] (2,0) to[out=180,in=-14] (0,0.6) -- (0.4,0.7) (0,0) -- (0,2) (-1,0) -- (0,2) -- (0.1,1.5) (0.2,1.5) -- (0,2) -- (0,2.5) (1.85,0.45) -- (2,0) -- (2.4,0.3) (1.85,-0.45) -- (2,0) -- (2.4,-0.3) (-0.3,0.5) -- (-0.7,0.6) -- (-0.3,0.7) (-1.15,-0.5) -- (-1,0) -- (-1.05,-0.5) (-0.95,-0.5) -- (-1,0) -- (-0.85,-0.5);
\draw [green!70!black,->] (-0.13,0.4) -- (-0.13,0.9);
\draw [green!70!black,->] (-0.95,0.4) -- (-0.7,0.9);
\draw [green!70!black,->] (-0.13,-1.6) -- (-0.13,-1.1);
\draw [green!70!black,->] (-0.95,-1.6) -- (-0.7,-1.1);
\draw [green!70!black,->] (0.65,-1.6) -- (0.4,-1.1);
\draw [green!70!black,->] (1.45,-1.63) -- (0.95,-1.13);
\foreach \x in {-1,0,1,2} {
	\draw [thick] (0,0) -- (\x,-2) (\x*0.7+0.4,-1.3) -- (\x*0.7,-1.4) -- (\x*0.7+0.4,-1.5) (\x-0.15,-2.5) -- (\x,-2) -- (\x-0.05,-2.5) (\x+0.05,-2.5) -- (\x,-2) -- (\x+0.15,-2.5);
	\draw [fill=blue] (\x,-2) circle (2pt);
	\draw [fill=green!70!black] (\x*0.7,-1.4) circle (2pt);
}
\draw [fill=yellow] (2,0) circle (2pt);
\draw [fill=blue] (0,0) circle (2pt) (0,2) circle (2pt) (-1,0) circle (2pt);
\draw [fill=green!70!black] (0,0.6) circle (2pt) (-0.7,0.6) circle (2pt);
\end{scope}

\begin{scope}[xshift=10cm,yshift=-3.5cm,yscale=-1]
\draw [thick] (2,0) to[out=180,in=-21.8] (0.5,1.8) -- (0,2) -- (0,0) -- (0.5,-0.2) (0,2) -- (0.5,2.2) (-1,0) -- (0,2) -- (0.1,1.5) (0.2,1.5) -- (0,2) -- (0,2.5) (0,0) -- (0.5,0.2) (1.85,0.45) -- (2,0) -- (2.4,0.3) (1.85,-0.45) -- (2,0) -- (2.4,-0.3) (-0.5,0.2) -- (-1,0) -- (-0.5,-0.2) (-1.15,-0.5) -- (-1,0) -- (-1.05,-0.5) (-0.95,-0.5) -- (-1,0) -- (-0.85,-0.5);
\draw [thick,densely dotted] (-0.5,0.2) arc (21.8:-21.8:0.5385) (0.5,0.2) arc (21.8:-21.8:0.5385) (0.5,2.2) arc (21.8:-21.8:0.5);
\foreach \x in {-1,0,1,2} {
	\draw [thick] (0,0) -- (\x,-2) (\x+0.5,-1.8) -- (\x,-2) -- (\x+0.5,-2.2) (\x-0.15,-2.5) -- (\x,-2) -- (\x-0.05,-2.5) (\x+0.05,-2.5) -- (\x,-2) -- (\x+0.15,-2.5);
	\draw [thick,densely dotted] (\x+0.5,-1.8) arc (21.8:-21.8:0.5385);
	\draw [fill=blue] (\x,-2) circle (2pt);
}
\draw [fill=yellow] (2,0) circle (2pt);
\draw [fill=blue] (0,0) circle (2pt) (0,2) circle (2pt) (-1,0) circle (2pt);
\end{scope}

\draw [->] (1,3) to[bend left=20] node [midway,above,xshift=-5mm] {expansion} (3,4);
\draw [->] (1,-2.5) to[bend right=20] node [midway,below,xshift=-5mm] {expansion} (3,-3.5);
\draw [->] (6.5,2) to node [midway,above] {collapse} (8.5,2);
\draw [->] (6.5,-4.5) to node [midway,above] {collapse} (8.5,-4.5);
\end{tikzpicture}
\caption{\emph{Left}: a $G$-tree $T$ of type $\calt(\Lambda)$; in this case $[H:L] = 5$, $[K:L] = 2$ and $\left| \det A \right| = 4$. \emph{Above}: an increasing induction $T \to T'$, where $T'$ is of type $\calt(\Lambda')$ and $[\Lambda':\Lambda] = 2$. \emph{Below}: a slide of the non-loop edge of $G \backslash T$ clockwise around the loop in $G \backslash T$; at the end of this operation, each blue vertex has eight rightward edges.}
\label{fig:unique}
\end{figure}
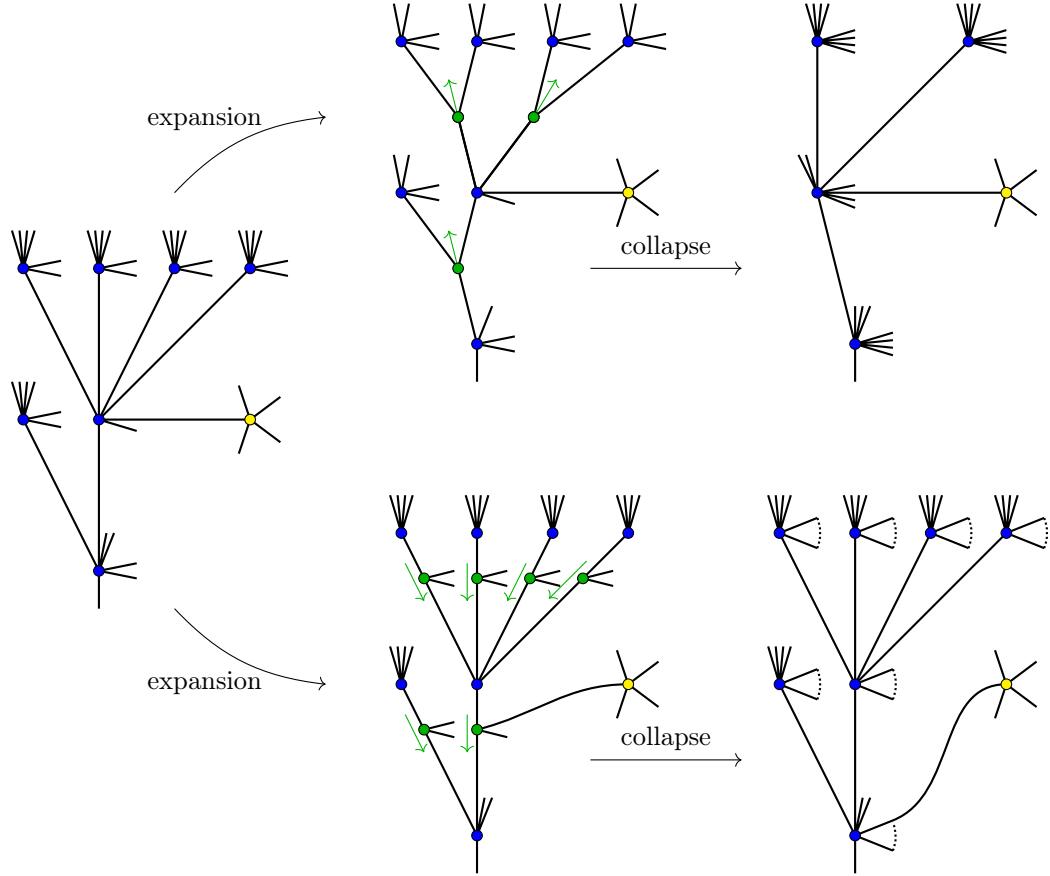

Note that, since $T$ has type $\calt(\Lambda)$, it has as a fundamental domain a segment with vertices $v,w,\widehat{t}^{-1}w$ stabilised by $H,K,\widehat{t}^{-1}K\widehat{t}$, respectively, where $\Phi(K) = \Lambda$ and $K \leq \widehat{t}^{-1}K\widehat{t}$.  At the vertex $w$, the rightward edges are then of the form $[w,gv]$ for $g \in K$, the upward edges of the form $[w,g\widehat{t}w]$ for $g \in K$, and the downward edge is $[w,\widehat{t}^{-1}w]$.

Let $\Lambda'$ be a lattice with $\Lambda < \Lambda' \leq A^{-1}\Lambda$. Since the map $\Phi|_{K}\colon K \to \Lambda$ is bijective, so is $\Phi|_{\widehat{t}^{-1}K\widehat{t}}\colon \widehat{t}^{-1}K\widehat{t} \to A^{-1}\Lambda$ (where the codomain is given by Lemma~\ref{lem:action-stable-letter}), hence there exists a unique subgroup $K' \leq \widehat{t}^{-1}K\widehat{t}$ such that $\Phi(K') = \Lambda'$. Then an increasing induction move resulting in a $G$-tree of type $\calt(\Lambda')$ is possible. However, any such move corresponds to the following operation: for each vertex $w' = g\widehat{t}^{-1}w \in T$ (stabilised by $g (\widehat{t}^{-1}K\widehat{t}) g^{-1}$) in the orbit of $w$, we fold each $gK'g^{-1}$-orbit of upward edges incident to $w'$ onto a single edge (see Figure~\ref{fig:unique}, above). Such an operation is uniquely defined by the choice of $K'$, and therefore by the choice of $\Lambda'$, proving \ref{it:ind-unique}.

On the other hand, the slide move described in \ref{it:slide-unique} corresponds to sliding all the rightward edges incident to $w'$ along the downward edge incident to $w'$, for each blue vertex $w' \in T$ (see Figure~\ref{fig:unique}, below). Such a move is unique since any vertex has only one downward edge incident to it.  We may check that the increasing induction move corresponds to identifying the edges $[g\widehat{t}^{-1}w,gw]$ and $[g\widehat{t}^{-1}w,ghw]$ (and thus the vertices $gw$ and $ghw$) for $h \in \widehat{t}^{-1}K\widehat{t}$, and the clockwise slide move corresponds to detaching the edge $[gv,gw]$ from the vertex $gw$ and reattaching it to the vertex $g\widehat{t}^{-1}w$. Writing $T'$ for the tree resulting from the clockwise slide, and using the notation $- \mapsto [-]$ for the quotient map $T \to \ind{T}{A^{-1}L'}$, consider the map $\varphi\colon V(T') \to V(\ind{T}{A^{-1}\Lambda})$ defined by $\varphi(gv) = [gv]$ and $\varphi(gw) = [g\widehat{t}w]$. 
It is then easy to check that the map $\varphi$ is a well-defined $G$-equivariant bijection and preserves adjacency, i.e.\ vertices $u,u'$ in $T'$ are adjacent if and only if $\varphi(u),\varphi(u')$  are adjacent in $\ind{T}{A^{-1}\Lambda}$. This shows that $T'$ and $\ind{T}{A^{-1}\Lambda}$ are $G$-equivariantly isomorphic, proving \ref{it:slide-unique}.

Suppose now that a $G$-tree $T'$ is obtained from $T$ using an anti-clockwise slide (this is possible if and only if $L$ stabilises some---equivalently, any---upward edge at $w$, or equivalently $K \geq \widehat{t}^{-1}L\widehat{t}$, i.e.\ $\Lambda \geq A^{-1}\Phi(L) = \Phi(L_1)$). Thus, $T$ is obtained from $T'$ using a clockwise slide, and therefore, by \ref{it:slide-unique}, an increasing induction.  But this implies that $T'$ is obtained from $T$ using a decreasing induction, proving \ref{it:slide=ind}.

Suppose that a $G$-tree $S'$ is obtained from $T$ using an $\scra^{-1}$-move (this is possible if and only if the stabiliser of some---equivalently, any---upward edge at $w$ also stabilises $v$, or equivalently $K \leq \widehat{t}^{-1}L\widehat{t}$, i.e.\ $\Lambda \leq A^{-1}\Phi(L) = \Phi(L_1)$).  Note that $S'$ has a vertex $v'$ stabilised by $H$ and a downward edge $e$ stabilised by $L$ incident to $v'$. Then any $G$-tree $T'$ obtained from $S'$ using an expansion that is the start of an $\scra$-move---in which case $T'$ must have type $\calt(\Phi(L))$---can be constructed as follows: pick an $L$-orbit $U$ of upward edges incident to~$v'$, add a vertex $w'$ and an edge $f = \{v',w'\}$ to $S'$, detach each edge in $\{ e \} \cup U$ from $v'$ and reattach it to $w'$, and repeat this process equivariantly; see Figure~\ref{fig:expand}. Note, however, that if the edges $\epsilon$ and $\widehat{t}\epsilon$ of $T$ are adjacent to $w$, they are in the same connected component in the forest obtained by removing the orbit of $v$ from $T$, and this type of property is not changed under induction moves: that is, if $T'$ can be obtained from $T$ using a decreasing induction, then the corresponding edges $\epsilon'$ and $\widehat{t}\epsilon'$ are in the same connected component in the forest obtained by removing the orbit of $v'$ from $T'$. This implies that $U$ must coincide with the $L$-orbit of the edge $\widehat{t}e$ in the construction above, and so the $G$-tree $T'$ can be defined uniquely, proving~\ref{it:ind-exp}.

\begin{figure}[ht]
    \centering
\begin{tikzpicture}
    \begin{scope}
    \foreach \x/\i/\c in {0/-1/black,1/-1/black,2/-1/black,3/1/red,4/1/red,5/1/red} {
        \draw [thick] ({\x-2.5},2) -- ({0.9*\x-2.25+0.2*\i},1.5) ({\x-2.75},2.5) -- ({\x-2.5},2) -- ({\x-2.65},2.5) ({\x-2.55},2.5) -- ({\x-2.5},2) -- ({\x-2.45},2.5) ({\x-2.35},2.5) -- ({\x-2.5},2) -- ({\x-2.25},2.5);
        \draw [thick,\c] (0,0) -- ({\x-2.5},2);
        \draw [fill=yellow] ({\x-2.5},2) circle (2pt);
    }
    \foreach \i/\c in {-1/black,1/red} {
        \draw [thick] (\i,-2) -- ({\i*0.85},-1.5) ({\i*0.95},-1.5) -- (\i,-2) -- ({\i*1.05},-1.5) ({\i*1.15},-1.5) -- (\i,-2) -- ({\i*1.25},-1.5) ({\i*0.85},-2.5) -- (\i,-2) -- ({\i*1.15},-2.5);
        \draw [thick,\c] (0,0) -- (\i,-2);
        \draw [fill=yellow] (\i,-2) circle (2pt);
    }
    \draw [fill=yellow] (0,0) circle (2pt) node [right] {$v'$};
    \node [right,red] at (0.4,-0.8) {$e$};
    \end{scope}
    \draw [->] (2.5,0) -- (4,0);
    \begin{scope}[xshift=7cm]
    \foreach \x/\i in {0/-1,1/-1,2/-1,3/1,4/1,5/1} {
        \draw [thick] ({\i*1.5},0) -- ({\x-2.5-\i*0.25},2) -- ({\x-2.5-\i*0.25},2.5) ({(\x-2.5-\i*1.25)*0.75+\i*1.5},1.5) -- ({\x-2.5+\i*0.25},2) -- ({\x-2.5+\i*0.25},2.5) ({\x-2.85},2.5) -- ({\x-2.75},2) -- ({\x-2.65},2.5) ({\x-2.35},2.5) -- ({\x-2.25},2) -- ({\x-2.15},2.5) ({\x-2.75},2) -- ({\x-2.25},2);
        \draw [fill=blue] ({\x-2.25},2) circle (2pt) ({\x-2.75},2) circle (2pt);
        \draw [fill=yellow] ({\x-2.5},2) circle (2pt);
    }
    \foreach \i in {-1,1} {
        \draw [thick] (0,0) -- ({\i*1.5},0) -- ({\i*0.5},-2) -- ({\i*0.5},-2.5) ({\i*0.6},-1.5) -- ({\i*0.5},-2) -- ({\i*0.45},-1.5) ({\i*2.15},-1.5) -- ({\i*2},-2) -- ({\i*2},-1.5) ({\i*2},-2.5) -- ({\i*2},-2) -- ({\i*1.85},-1.5) ({\i*2},-2) -- ({\i*0.5},-2);
        \draw [fill=blue] (\i*1.5,0) circle (2pt) (\i*0.5,-2) circle (2pt) (\i*2,-2) circle (2pt);
        \draw [fill=yellow] (\i*1.25,-2) circle (2pt);
    }
    \draw [fill=yellow] (0,0) circle (2pt) node [below] {$v'$};
    \node [above] at (0.75,0) {$f$};
    \node [right] at (1.5,0) {$w'$};
    \end{scope}
\end{tikzpicture}
    \cprotect\caption{Construction of the $G$-tree $T'$ (right) from $S'$ (left) in the proof of Lemma~\ref{lem:unique}\ref{it:ind-exp}; in this case, $[H:L] = 2$ and $\left| \det(A) \right| = 3$.  The edges $\{e\} \cup U$ of $S'$ are coloured in red.}
    \label{fig:expand}
\end{figure}
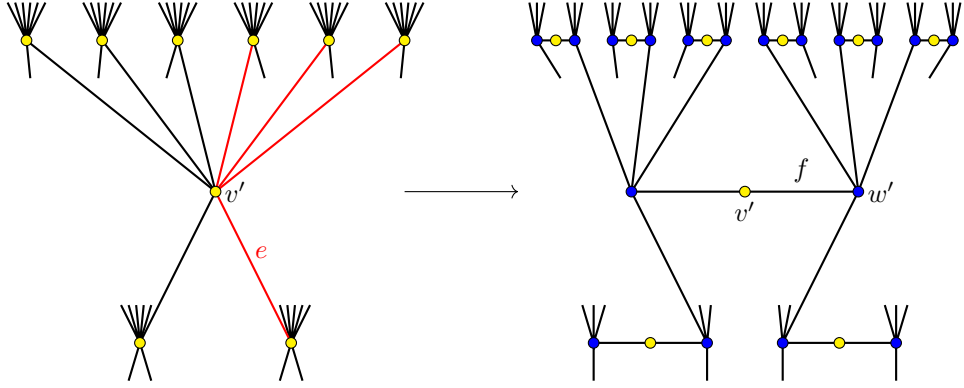

Finally, let $\Lambda_1,\Lambda_2$ be lattices satisfying $\Lambda < \Lambda_1 < \Lambda_2 \leq A^{-1}\Lambda$. Since the map $\Phi|_{K}\colon K \to \Lambda $ is bijective, so is $\Phi|_{\widehat{t}^{-1}K\widehat{t}}\colon \widehat{t}^{-1}K\widehat{t} \to A^{-1}\Lambda$, hence there exist unique subgroups $K_1,K_2 \leq \widehat{t}^{-1}K\widehat{t}$ such that $\Phi(K_1) = \Lambda_1$ and $\Phi(K_2) = \Lambda_2$. Then $T_1 \coloneqq \ind{T}{\Lambda_1}$ is obtained by identifying the edges $[gw,g\widehat{t}^{-1}w]$ and $[ghw,g\widehat{t}^{-1}w]$ for $h \in K_1$ in $T$, 
and $\ind{T_1}{\Lambda_2}$ is obtained by identifying the edges $[gw,g\widehat{t}^{-1}w]$ and $[ghw,g\widehat{t}^{-1}w]$ for $h \in K_2$ in $T_1$. This is clearly equivalent to identifying the edges $[gw,g\widehat{t}^{-1}w]$ and $[ghw,g\widehat{t}^{-1}w]$ for $h \in K_2$ in $T$, showing that $\ind{T}{\Lambda_2} = \ind{T_1}{\Lambda_2}$ and thus proving~\ref{it:ind-ind}.
\end{proof}

Clearly not every tree in $\cald$ has one of the above types. But our next goal is to show that every tree in $W$ does; as a first observation notice that since $H$ is a maximal elliptic subgroup it cannot stabilise a surviving edge, and so every tree in $W$ has a 
unique vertex stabilised by $H$.

First we characterise the reduced trees in $\cald$ (all of which are in $W$).

\begin{lemma} \label{lem:reduced}
The reduced $G$-trees in $\cald$ are exactly the trees of type $\cals$ or $\calt(\Lambda)$, where $\Lambda < \Phi(L_\infty)$ is a lattice satisfying $\Phi(L) < \Lambda < A^{-1}\Lambda$. 
\end{lemma}

\begin{proof}
We first note that any $G$-tree of the form $\cals$ or $\calt(\Lambda)$, where $\Lambda < \Phi(L_\infty)$ is a lattice with $\Phi(L) < \Lambda < A^{-1}\Lambda$, is reduced: edges either have both endpoints in the same orbit or both inclusions proper. Note that by definition $\cald$ contains the $G$-tree  $S$ (of \Cref{fig:gog_for_LMAL}), which has type $\cals$ and where $t \in G$ acts as the stable letter.

We now show that for any $\Lambda$ as in the hypotheses, $\cald$ contains a tree of type $\calt(\Lambda)$. Since $\Phi$ is injective on $L_\infty$, there is a unique subgroup $K < L_\infty$ satisfying $\Phi(K) = \Lambda$, and since $\Lambda < A^{-1}\Lambda$, it follows that $K < t^{-1}Kt$. Since $K$ is finitely generated, there is a minimal $L_j$ which contains $K$ (that is, $K \nleq L_{j-1}$). 
The proof is by induction on $j$. For $1 \leq i \leq j$, let $K_i = K \cap L_i$ (so $K_j = K$) and note that for every $i$, $\Phi(K_i) = \Phi(K \cap L_i) = \Lambda \cap \Phi(L_i)$, and $\Phi(t^{-1}K_it)= A^{-1}\Phi(K_i)$. We have: \begin{enumerate}[(i)]
    \item (Base case) Since $L < K_1 \leq L_1$, there is a $\scra$-move which transforms $S$ to a tree with fundamental domain a segment whose vertices are stabilised by $H$, $K_1$ and $t^{-1}K_1t$. Since this can be done so that $t$ is still a stable letter in the new $G$-tree, by Lemma~\ref{lem:same-type-is-iso} the element $\Phi(t)$ acts as $A$ on $\Phi(t^{-1}K_1t)$, and so this is a tree of type $\calt(\Lambda \cap \Phi(L_1))$.
    \item (Inductive step) Suppose we have constructed a tree $T_i$ of type $\calt(\Lambda \cap \Phi(L_i))$, where the stabilisers in the fundamental domain are $H$, $K_i$ and $t^{-1}K_it$. We have \[ K_i = K \cap L_i \leq K \cap L_{i+1} \leq t^{-1}Kt \cap L_{i+1} = t^{-1}(K \cap L_i)t = t^{-1}K_i t, \] 
    and note that the middle term is $K \cap L_{i+1} = K_{i+1}$. So there is an increasing induction transforming $T_i$ into $T_{i+1}$, which has a fundamental domain consisting of a segment whose vertex stabilisers are $H$, $K_{i+1}$ and $t^{-1}K_{i+1}t$, and this can be done so that $t$ is a stable letter of $T_{i+1}$. Since, by Lemma~\ref{lem:action-stable-letter}, $\Phi(t)$ acts as $A$ on $\Phi(t^{-1}K_{i+1}t)$, this is a tree of type $\calt(\Lambda \cap \Phi(L_{i+1}))$.
    \end{enumerate}

This shows that any tree of type $\cals$ or $\calt(\Lambda)$ is contained in $\cald$. 
By \cite[Theorem~1.1]{ClayForester2009}, it is now enough to show that the set of reduced trees of types $\cals$ and $\calt(\Lambda)$ is invariant under slides, inductions and $\scra^{\pm1}$-moves, described in Definition~\ref{def:slide_induction_A_moves}. 

Consider first a $G$-tree of type $\cals$. Then the only possible moves are $\scra$-moves, resulting in a $G$-tree $T$ which has a fundamental domain a segment of length two whose vertices are stabilised by $H$, $K$ and $\widehat{t}^{-1}K\widehat{t}$ for some $\widehat{t} \in G$.  Moreover, it follows from the construction that, after replacing $\widehat{t}$ by $\widehat{t}^{-1}$ if necessary, the edge joining vertices stabilised by $K$ and $\widehat{t}^{-1}K\widehat{t}$ is itself stabilised by $K$, meaning that $\widehat{t}$ conjugates $K$ to a proper subgroup of itself, and therefore is a stable letter of $T$; hence $T$ has type $\calt(\Lambda)$, where $\Lambda = \Phi(K)$. It then follows from the definition of an $\scra$-move that $L < K$ and so $\Phi(L) < \Lambda$, whereas Lemma~\ref{lem:action-stable-letter} implies that $\Lambda = \Phi(K) < \Phi(\widehat{t}^{-1}K\widehat{t}) = A^{-1}\Lambda$.

Now consider a $G$-tree of type $\calt(\Lambda)$ for some lattice $\Lambda < \Phi(L_\infty)$ with $\Phi(L) < \Lambda < A^{-1}\Lambda$; let $H$ and $K$ be the vertex groups of the corresponding graph of groups (as in Figure~\ref{sfig:gog-standard}). 
Then induction, slide and $\scra^{-1}$-moves may be possible. Any slide move is equivalent to an induction move (as noted after Definition~\ref{def:slide_induction_A_moves}), so it is enough to consider induction and $\scra^{-1}$-moves.

Note first that any induction move preserves the property of the loop-edge of the corresponding graph of groups being \emph{strictly ascending}, i.e.\ such that exactly one of the two edge group monomorphisms is surjective.  Moreover, if after a decreasing induction move a collapse move is possible (resulting in an $\scra^{-1}$-move), then the collapsed edge was stabilised by $L$ by Lemma~\ref{lem:find_H_and_L}, so in the resulting one-edge graph of groups the edge group monomorphisms have images $L$ (which is a finite-index subgroup of $H$) and a proper subgroup of $L$.  Therefore, in either case the requirement that the stable letter $\widehat{t}$ conjugates a finite-index subgroup subgroup of $H$ to a proper subgroup of itself (for type $\cals$ trees) or that $K' \leq \widehat{t}^{-1}K'\widehat{t}$ (for type $\calt(\Phi(K'))$ trees) is satisfied by the new $G$-tree.

Given a subgroup $K'$ of $G$ so that $K < K' \leq \widehat{t}^{-1}K\widehat{t}$, we have an increasing induction resulting in a $G$-tree with a fundamental domain which is a segment of length two whose stabilisers are $H$, $K'$ and $\widehat{t}^{-1}K'\widehat{t}$, and every increasing induction move is of this form; here $\widehat{t}$ is a stable letter of the original $G$-tree.  Setting $\Lambda' \coloneqq \Phi(K')$ we then have $\Lambda < \Lambda' \leq A^{-1}\Lambda$ since $\Phi(\widehat{t}^{-1}K\widehat{t}) = A^{-1}\Lambda$ by Lemma~\ref{lem:action-stable-letter}, so in particular $\Phi(L) < \Lambda' < A^{-1}\Lambda'$, and hence the new $G$-tree has type $\calt(\Lambda')$.

Given a subgroup $K' < K$ with $A\Phi(K) \leq \Phi(K')$ and $\Phi(L) < \Phi(K')$, we similarly have (by Lemma~\ref{lem:action-stable-letter}) a decreasing induction resulting in a $G$-tree with two-edge fundamental domain whose vertex stabilisers are $H$, $K'$ and $\widehat{t}^{-1}K'\widehat{t}$ for some $\widehat{t} \in G$, and every decreasing induction is of this form. 
Setting $\Lambda' \coloneqq \Phi(K')$ it then follows from the construction that $\Phi(L) < \Lambda'$ and $\Lambda' < \Lambda \leq A^{-1}\Lambda'$, and thus he new $G$-tree has type $\calt(\Lambda')$, as required.

Finally, if 
an $\scra^{-1}$-move is possible then this results in a $G$-tree with fundamental domain a single edge and where $H$ is a vertex stabiliser; such a $G$-tree must then have type $\cals$. 
\end{proof}

We now move on to the remaining points of $W$. Recall that the $G$-trees in $W$ are those where every edge is surviving: for each edge there is a collapse map to some reduced tree where this edge survives.

\begin{lemma} \label{lem:W}
Any $G$-tree $T \in W$ is either of type $\cals$ or of type $\calt(\Lambda_1,\ldots,\Lambda_k)$, where $\Phi(L) \leq \Lambda_1 <  \cdots < \Lambda_k \leq A^{-1}\Lambda_1$ are lattices in $\Phi(L_\infty)$.
\end{lemma}

\begin{proof}
Note first that since the deformation space $\cald$ contains a tree of type $\cals$, it follows that $G$ has a unique conjugacy class of maximal elliptic subgroups. This can be used to show, following the proof of \cite[Lemma~3.2]{Clay09} that $G \backslash T$ is either a single loop edge or consists of a single non-loop edge attached to a shelter of type (S3). (The statement is given for Baumslag--Solitar groups, but the description of the quotient graphs is valid under the assumptions that $\beta_1(G \backslash T)=1$, there is a unique conjugacy class of maximal elliptic subgroups, and there are no vertices of valence $2$ in $T$ whose stabiliser agrees with the incident edge stabilisers.) 

If $G \backslash T$ is a single loop edge then $T$ is reduced, implying by Lemma~\ref{lem:reduced} that $T$ is of type $\cals$. 

Thus, suppose $G \backslash T$ consists of a single non-loop edge attached to a shelter of type (S3). Suppose $v$ is some representative of the leaf vertex of $G \backslash T$: then $G_v$ must be a maximal elliptic subgroup (otherwise the adjacent edge is not surviving) and in particular a conjugate of $H$. This implies there is a fundamental domain for the $G$-tree $T$ consisting of a segment stabilised by (in order) $H, K_1, \dots, K_k, \widehat{t}^{-1}K_1\widehat{t}$, where $K_1 < \dots < K_k \leq \widehat{t}^{-1}K_1\widehat{t}$. (The choice is not unique, but provided we chose the increasing direction along an axis at distance 1 from $H$ projecting to the cycle, it does not matter.)

By Lemma~\ref{lem:find_H_and_L}, the edge adjacent to $v$ must be stabilised by $L$, so in fact we have a chain of inclusions $L \leq K_1 < \dots < K_k \leq \widehat{t}^{-1}K_1\widehat{t}$. Setting $\Lambda_i \coloneqq \Phi(K_i)$, in order to show that $T$ has type $\calt(\Lambda_1,\dots,\Lambda_k)$ it is then enough to show that $\Lambda_k \leq A^{-1}\Lambda_1 < \Phi(L_\infty)$. However this follows from the facts that $\Phi(\widehat{t}^{-1}K_1\widehat{t}) = A^{-1}\Lambda_1$ and $\Lambda_k < \Phi(L_\infty)$: indeed, collapsing the edges stabilised by $K_1,\ldots,K_{k-1}$ (and their orbits), we get a reduced $G$-tree, which must have type $\calt(\Lambda_k)$ by Lemma~\ref{lem:reduced}. Thus $\Lambda_k < \Phi(L_\infty)$, and such a reduced tree has $\widehat{t}$ as its stable letter, implying by Lemma~\ref{lem:action-stable-letter} that conjugation by $\Phi(\widehat{t}^{-1})$ acts as $A^{-1}$ on $\Lambda_k$, and so on $\Lambda_1$.
\end{proof}

\begin{remark}
    By starting with a tree of type $\calt(\Lambda_1)$ and applying $k-1$ expansion moves, one may construct a $G$-tree of type $\calt(\Lambda_1,\ldots,\Lambda_k)$ for any lattices $\Phi(L) \leq \Lambda_1 < \cdots < \Lambda_k \leq A^{-1}\Lambda$ in $\Phi(L_\infty)$.  In such a tree all edges are surviving except in the case of trees of type $\calt(\Phi(L))$, in which the non-loop edge is not surviving.  Therefore, the trees in $W$ are precisely those of types $\cals$ and $\calt(\Lambda_1,\ldots,\Lambda_k)$, where $\Phi(L) \leq \Lambda_1 <  \cdots < \Lambda_k \leq A^{-1}\Lambda_1$ are lattices in $\Phi(L_\infty)$ satisfying $\Phi(L) < \Lambda_k$. 
\end{remark}

\subsection{Unsubdividing}

We now give an alternative description of the geometric realisation of the simplicial complex $W$. In particular, given two reduced $G$-trees $T_1$ and $T_2$, we write $T_1 \sqsubsetdot T_2$ if $T_2$ can be obtained from $T_1$ by using an increasing induction or an $\scra$-move.  Note that this is not a transitive relation as hieroglyphed by the dot in the $\sqsubsetdot$ square.  We define $\widehat{W}$ to be a simplicial complex with vertices the reduced $G$-trees in $\cald$, so that $\{ T_0,\ldots,T_k \}$ forms a $k$-simplex in $\widehat{W}$ whenever $T_i \sqsubsetdot T_j$ for $0 \leq i < j \leq k$. 
We will make use of the following auxiliary result.

\begin{lemma} \label{lem:poset}
Let $\sqsubset$ be the transitive closure of the relation $\sqsubsetdot$ on $\widehat{W}$. Then $(\widehat{W},\sqsubset)$ is a poset, and for any $T \in \widehat{W}$ the set $\widehat{W}_T = \{ T' \in \widehat{W} \mid T' \sqsubset T \}$ is finite.
\end{lemma}

\begin{proof}
Since $\sqsubset$ is transitive by definition, in order to prove that $(\widehat{W},\sqsubset)$ is a poset it is enough to show that $\sqsubset$ is irreflexive, i.e.\ the relation $T \sqsubset T$ cannot hold for any $T \in \widehat{W}$. Together with transitivity, irreflexivity will then imply asymmetry as well.

Note that a $G$-tree of type $\cals$ has one orbit of vertices, and a $G$-tree of type $\calt(\Lambda)$ has two orbits, with degrees $[H:L]$ and $[\Lambda:\Phi(L)]+\left|\det A\right|+1$, respectively. Define the \emph{signature} of a reduced $G$-tree $T$ to be $s(T) = 0$ if $T$ is of type $\cals$, and $s(T) = [\Lambda:\Phi(L)]$ if $T$ is of type $\calt(\Lambda)$; by the previous sentence, the signature of a $G$-tree is determined uniquely, i.e.\ isomorphic $G$-trees have the same signature. However, if $T '\sqsubsetdot T$ for some $T',T \in \widehat{W}$, then we must have $s(T') < s(T)$, hence the same is true whenever $T' \sqsubset T$. It follows that one cannot have $T \sqsubset T$ for any $T \in \widehat{W}$, as required.

Now let $T \in \widehat{W}$, and suppose $T'$ is an element of $\widehat{W}_T$. This means that there is a sequence $(T' = T_0, T_1, \dots, T_r=T)$ so that $T_i \sqsubsetdot T_{i+1}$ (it is not necessarily unique). Since $s(T_i)<s(T_{i+1})$ for the trees in any such sequence, the depth is bounded: $r \leq s(T)$. Since all $G$-trees in $\widehat{W}$ are locally finite, there are finitely many decreasing inductions or $\scra^{-1}$-moves that can be used on a given $G$-tree $T' \in \widehat{W}$. It follows that given the sequence $(T_1, \dots, T_r=T)$ of length $r-1$ there are only finitely many trees $T_0 \sqsubsetdot T_1$ that extend it to a sequence of length $r$. So there are only finitely many trees $T'$ at each given depth, and since the depth is bounded there are only finitely many in total. \qedhere

\end{proof}

\begin{lemma} \label{lem:Psi-surjective}
    Let $T_1,\ldots,T_k$ be reduced $G$-trees such that $T_i \sqsubsetdot T_j$ for $1 \leq i < j \leq k$, for some $k \geq 1$.  Then there exists a unique tree $T$ in $W$ such that $T_1,\ldots,T_k$ are precisely the reduced collapses of $T$.  Moreover, 
    \begin{enumerate}
        \item if $k=1$ then $T=T_1$;
        \item if $k>1$ and $T_1$ has type $\cals$, then $T$ has type $\calt(\Lambda_1,\dots \Lambda_k)$ where $\Lambda_1=\Phi(L)$ and (for $i>1$) $T_i$ has type $\calt(\Lambda_i)$;
        \item if $k>1$ and $T_1$ does not have type $\cals$, then $T$ has type $\calt(\Lambda_1, \dots \Lambda_k)$ where $T_i$ has type $\calt(\Lambda_i)$.
    \end{enumerate}
\end{lemma}

\begin{proof}
    Note that, by the definition of the increasing induction and $\scra$-moves, the trees $T_1,\ldots,T_k$ are pairwise compatible. Therefore, by Lemma~\ref{lem:coll-same-red}, there exists a $G$-tree $T$ contained in $W$ that collapses to each $T_i$ and is minimal among the trees satisfying this property.
    
    We first claim that $T_1,\ldots,T_k$ are precisely the reduced collapses of $T$.  If $k = 1$ then $T = T_1$, and this is the only case when $T$ is reduced. Otherwise, by Lemma~\ref{lem:W}, $T$ has type $\calt(\Lambda_1,\ldots,\Lambda_\ell)$ for some $\ell \geq 2$, so that it has a fundamental domain that is a segment with vertices stabilised (in order) by $H,K_1,\dots,K_\ell,\widehat{t}^{-1}K_1\widehat{t}$, where $K_1 < \dots < K_\ell \leq \widehat{t}^{-1}K_1\widehat{t}$ and $\Phi(K_i) = \Lambda_i$; let $e,e_1,\ldots,e_\ell$ be the edges of this segment (in this order).  If an edge $e_j$ did not survive in any of the $T_i$ under the collapse map, then collapsing the orbit of $e_j$ would result in a tree $T' \prec T$ that has all its edges surviving and still collapses to each of the $T_i$, contradicting minimality of~$T$.  Thus each of the $e_j$ survives in some $T_i$; since exactly one of the edges $e_j$ survives in any reduced collapse of $T$, this implies that $\ell \leq k$.  But any reduced collapse of $T$ is obtained by collapsing the orbits of $e_1,\dots,e_{j-1},e_{j+1},\dots,e_\ell$ for some $j$ (and additionally collapsing the orbit of $e$ if $j=1$ and $\Lambda_1 = \Phi(L)$), so there are at most $\ell$ distinct reduced collapses of $T$.  Since the $G$-trees $T_1,\dots,T_k$ are pairwise non-isomorphic (for instance, since the relation $\sqsubsetdot$ is irreflexive, see Lemma~\ref{lem:poset}), it follows that there are no other reduced collapses of $T$, as claimed.

    It now follows from the paragraph above that $\ell = k$ and the reduced collapses of $T$ are trees $T_1',\dots,T_k'$ of types $\calt(\Lambda_1),\dots,\calt(\Lambda_k)$, respectively (with $\cals$ instead of $\calt(\Lambda_1)$ if $\Lambda_1 = \Phi(L)$).  These must therefore be the trees $T_1,\dots,T_k$, in some order.  But for any $1 \leq i < j \leq k$, collapsing all edge orbits apart from the orbits of $e_i$ and $e_j$ results in a tree of type $\calt(\Lambda_i,\Lambda_j)$ that collapses to both $T_i'$ and $T_j'$, showing that $T_i' \sqsubsetdot T_j'$.  Hence the order must be the same and we have $T_i = T_i'$, i.e.\ $T_i$ has type $\calt(\Lambda_i)$ (unless $i = 1$ and $\Phi(L) = \Lambda_1$, in which case $T_1$ has type $\cals$).  This proves the statement about the type of $T$.

    Finally, uniqueness of $T$ follows from its minimality: if $T' \in W$ is a tree whose reduced collapses are precisely $T_1,\dots,T_k$, then $T' \preceq T$ by Lemma~\ref{lem:coll-same-red} and so $T' = T$ by the minimality of $T$.
\end{proof}

Recall that an action of a group $H$ on a simplicial complex $X$ is said to be \emph{admissible} if any element of $H$ fixing a simplex in $X$ fixes that simplex pointwise. (If the simplicial complex is a tree, this is the standard condition of acting without inversions.)

\begin{prop} \label{prop:What}
$W$ is $\Out(G)$-equivariantly isomorphic to the barycentric subdivision of $\widehat{W}$. Moreover, the $\Out(G)$-action on $\widehat{W}$ is admissible.
\end{prop}

\begin{proof}
It is enough to find a bijective map $\Psi\colon \{ \text{vertices of } W \} \to \{ \text{simplices of } \widehat{W} \}$ which is $\Out(G)$-equivariant, such that $\{ T_0,\ldots,T_k \}$ forms a simplex in $W$ if and only if the collection $\{ \Psi(T_0),\ldots,\Psi(T_k) \}$ is totally ordered by inclusion. The simplest description of $\Psi(T)$ is that it is the subcomplex spanned by the set of reduced trees that $T$ collapses to. This map is injective, since by Lemma~\ref{lem:coll-same-red} a vertex in $W$ is uniquely defined by the reduced trees it collapses to. By the same lemma, $T_1 \prec T_2$ if and only if $\Psi(T_1) \subset \Psi(T_2)$, and the assignment is $\Out(G)$-equivariant since the $\Out(G)$-action on $\cald$ preserves the property of being reduced and gives that if $T_1 \prec T_2$ then $T_1 \cdot \phi \prec T_2 \cdot \phi$. 

It remains to check that $\Psi$ sends vertices of $W$ to simplices of $\widehat{W}$, that $\Psi$ is surjective to the set of simplices, and that the $\Out(G)$-action is admissible. The first one of these requires a more careful analysis of $\Psi(T)$, for which we use Lemma~\ref{lem:W} to describe the trees in $W$. Reduced trees are sent to $0$-simplices (themselves), so the interesting case is when $T$ is a tree of type $\calt(\Lambda_1, \dots, \Lambda_k)$. For each $i$, there is an edge in the chosen fundamental domain stabilised by $K_i$ satisfying $\Phi(K_i)=\Lambda_i$. Collapsing all the other orbits projecting to the cycle in $G \backslash T$ gives a tree of type $\calt(\Lambda_i)$, which is reduced unless $i=1$ and $\Lambda_1 = \Phi(L)$. In the latter case, collapsing the edge orbit not projecting to the cycle gives a (reduced!) $G$-tree of type $\cals$, and we take $T_1$ to be this tree. Then $\Psi(T)= \{T_1, \dots, T_k\}$: no other collapses to reduced trees exist, for instance because at least one edge projecting to the cycle survives in any such reduced tree.

For $1 \leq i < j \leq k$, the $G$-trees $T_i$ and $T_j$ are obtained by collapse moves from the same $G$-tree $T_{ij}$ of type $\calt(\Lambda_i,\Lambda_j)$ (since this tree---where the vertex stabilisers in a fundamental domain are $H, K_i,K_j,\widehat{t}^{-1}K_i\widehat{t}$---is, by Lemma~\ref{lem:Psi-surjective}, a tree in $W$ whose reduced collapses are exactly $T_i$ and $T_j$). Inspecting the sequence $T_i \prec T_{ij} \succ T_j$, one finds that the expansion(s) and collapse exactly match those involved in the increasing induction associated to $K_j$ (or an $\scra$-move associated to $K_j$, if $T_i$ is of type $\cals$), and so we have $T_i \sqsubsetdot T_j$. This shows that the image of $\Psi$ lies in the set of simplices of $\widehat{W}$.

Now surjectivity of $\Psi$ follows from Lemma~\ref{lem:Psi-surjective}.
Finally, to show that the $\Out(G)$-action is admissible, note that if $T_1 \sqsubsetdot T_2$ for some $T_1,T_2 \in \widehat{W}$, then by Lemma~\ref{lem:poset} $T_2 \nsqsubsetdot T_1$. \qedhere
\end{proof}

We are now almost ready to define a further subcomplex of $\widehat{W}$ and construct a deformation retract to it. First we prove a lemma giving various forms of uniqueness when we apply the moves of Definition~\ref{def:slide_induction_A_moves} to the reduced trees in $W$ (i.e. the vertex set of $\widehat{W}$).

\subsection{The subcomplex \texorpdfstring{$\widehat{X}$}{X} and a deformation retract}

We now define $\widehat{X}$ to be the subcomplex of $\widehat{W}$ spanned by the $G$-trees of types $\cals$ and $\calt(\Phi(L_i))$ for $i \geq 1$. We first show that $\widehat{X}$ is one-dimensional and invariant under the $\Out(G)$-action, and then that $\widehat{W}$ $\Out(G)$-equivariantly deformation retracts to $\widehat{X}$. Then contractibility of $\widehat{W}$ will imply that $\widehat{X}$ is a tree, and we can use Bass--Serre theory to study $\Out(G)$ through its action on this tree.

\begin{lemma} \label{lem:Xhat}
$\widehat{X}$ is a one-dimensional subcomplex invariant under the $\Out(G)$-action.
\end{lemma}

\begin{proof}
Note that clearly $\dim \widehat{X} \geq 1$ (for instance, $\widehat{X}$ contains an edge $\{ S,T \}$, where $S$ is a $G$-tree of type $\cals$ and $T$ is a $G$-tree of type $\calt(\Phi(L_1))$ obtained from $S$ using an $\scra$-move). Conversely, note that if $T \sqsubsetdot T'$ for some $T,T' \in \widehat{X}$, then $T$ and $T'$ have either types $\cals$ and $\calt(\Lambda)$, respectively (in which case $T'$ is obtained from $T$ using an $\scra$-move, and so $\Phi(L_0) < \Lambda \leq A^{-1}\Phi(L_0) = \Phi(L_1)$, hence $\Lambda = L_1$ since $T' \in \widehat{X}$), or types $\calt(\Phi(L_i))$ and $\calt(\Lambda)$, respectively (in which case $T'$ is obtained from $T$ using an increasing induction, and so $\Phi(L_i) < \Lambda \leq A^{-1}\Phi(L_i) = \Phi(L_{i+1})$, hence $\Lambda = \Phi(L_{i+1})$ since $T' \in \widehat{X}$). This implies that if $T \sqsubsetdot T'$ and $T' \sqsubsetdot T''$ for some $T,T',T'' \in \widehat{X}$, then $T$ and $T''$ have either types $\cals$ and $\calt(\Phi(L_2))$ or types $\calt(\Phi(L_i))$ and $\calt(\Phi(L_{i+2}))$ for some $i$, implying that $T \nsqsubsetdot T''$ and so $\{ T,T',T'' \}$ is not a simplex in $\widehat{X}$. Thus $\dim \widehat{X} = 1$, as required.

To show that $\widehat{X}$ is invariant under the $\Out(G)$-action, note first that the $G$-trees in $\widehat{X}$ have the following alternative description: a reduced $G$-tree is in $\widehat{X}$ if and only if it either has type $\cals$ or can be transformed to a $G$-tree of type $\cals$ using a sequence of slide moves followed by a collapse move. Indeed, if $T \in \widehat{X}$ has type $\calt(\Phi(L_i))$, then by Lemma~\ref{lem:unique}\ref{it:slide=ind} a sequence of $i$ anti-clockwise slides transforms $T$ to a $G$-tree of type $\calt(\Phi(L_0))$, after which a collapse transforms this to a $G$-tree of type $\cals$, showing the ``only if'' part. Conversely, if $T \in \widehat{W} \smallsetminus \widehat{X}$ has type $\calt(\Lambda)$, then by Lemma~\ref{lem:unique}\ref{it:slide-unique} and Lemma~\ref{lem:unique}\ref{it:slide=ind} any sequence of slide moves transforms $T$ to a $G$-tree $T'$ of type $\calt(A^j\Lambda)$ for some $j \in \ZZ$, and we have $A^j\Lambda \neq \Phi(L_0)$ (as otherwise $\Lambda = \Phi(L_j)$, contradicting the fact that $T \notin \widehat{X}$), implying that no collapse move is possible to convert $T'$ to a $G$-tree of type $\cals$; this shows the ``if'' part.

Now note that the set of $G$-trees of type $\cals$ is characterised (among all reduced $G$-trees in $\cald$) by the property that they have only one orbit of vertices, and therefore this set is invariant under the $\Out(G)$-action. Furthermore, if $T,T'$ are $G$-trees and $\phi \in \Out(G)$, then $T'$ can be obtained from $T$ using a slide (respectively collapse) move if and only if $T' \cdot \phi$ can be obtained from $T \cdot \phi$ using a slide (respectively collapse) move. Using the alternative description of $\widehat{X}$ given in the previous paragraph, it follows that $\widehat{X}$ is invariant under the $\Out(G)$-action, as required.
\end{proof}

\begin{prop} \label{prop:dr}
There exists an $\Out(G)$-equivariant deformation retraction of $\widehat{W}$ to~$\widehat{X}$.
\end{prop}

\begin{proof}
Our proof proceeds in the following steps:
\begin{enumerate}[label=(\roman*)]
\item \label{it:dr-enum} enumerating the vertices of $\widehat{W} \smallsetminus \widehat{X}$ as $T_0,T_1,\ldots$, so that if $T_i \sqsubsetdot T_j$ then $i < j$;
\item \label{it:dr-elcoll} finding a sequence of elementary collapses transforming $X_i$ to $X_{i+1}$, where $X_j$ is the subcomplex of $\widehat{W}$ induced by $\widehat{X} \cup \{ T_j,T_{j+1},\ldots \}$;
\item \label{it:dr-dr} using these sequences of elementary collapses to define a deformation retraction of $\widehat{W}$ to~$\widehat{X}$;
\item \label{it:dr-equiv} explaining how this construction can be modified so that the resulting deformation retraction is $\Out(G)$-equivariant.
\end{enumerate}

We now deal with each of the steps in order.

\medskip\pushQED{$\blackdiamond$}

\noindent\emph{Proof of \ref{it:dr-enum}.} Let $\sqsubset$ be the transitive closure of the relation $\sqsubsetdot$ on $\widehat{W}$, so that, by Lemma~\ref{lem:poset}, $(\widehat{W},\sqsubset)$ is a poset and $\widehat{W}_T := \{ T' \in \widehat{W} \mid T' \sqsubset T \}$ is finite for any $T \in \widehat{W}$. Since the $G$-trees in $\widehat{W}$ are locally finite, each of them admits only finitely many induction or $\scra^{\pm 1}$-moves, implying that the simplicial complex $\widehat{W}$ is locally finite; as $\widehat{W}$ is also connected, it follows that it has countably many vertices, so we can enumerate the vertices of $\widehat{W} \smallsetminus \widehat{X}$ by $S_0,S_1,\ldots$, say.

For any $i \geq 0$, let 
$Y_i\coloneqq \{T \in \widehat{W} \smallsetminus \widehat{X} \mid T \sqsubseteq S_i, T \nsqsubseteq S_j \text{ for any } j < i\}$.  Then the $Y_i$ are finite disjoint subsets of $\widehat{W}$, and $\bigsqcup_{i=0}^\infty Y_i = \widehat{W} \smallsetminus \widehat{X}$. As any finite non-empty poset has a minimal element, it follows that we can write any non-empty $Y_i = \{ T_{i,0},\ldots,T_{i,k_i} \}$ so that if $T_{i,j} \sqsubset T_{i,j'}$ then $j < j'$. In particular, if $Y_i$ is non-empty we have $T_{i,k_i}=S_i$.

Now consider the ordering $T_{0,0},\ldots,T_{0,k_0},T_{1,0},\ldots,T_{1,k_1},\ldots$ of vertices of $\widehat{W} \smallsetminus \widehat{X}$. We claim that if $T \sqsubsetdot T'$ for some $T,T' \in \widehat{W} \smallsetminus \widehat{X}$ then $T$ appears before $T'$ in this ordering. Indeed, suppose $T_{i,j} \sqsubsetdot T_{m,n}$. Since $T_{m,n} \sqsubseteq S_{m}$ we also have $T_{i,j} \sqsubset S_{m}$ and therefore $T_{i,j} \in \widehat{W}_{S_{m}}$; since $T_{i,j} \notin \widehat{X}$ this gives that $T_{i,j} \in \bigcup_{l=0}^{m} Y_l$, implying that $i \leq m$. This implies that either $i < m$, or $i = m$ and (by the choice of ordering of $Y_i$) $j < n$, thus proving the claim. Relabelling the $G$-trees in the given order as $T_0,T_1,\ldots$, we have $i < j$ whenever $T_i \sqsubsetdot T_j$, as required. \hfill\popQED

\medskip\pushQED{$\blackdiamond$}

\noindent\emph{Proof of \ref{it:dr-elcoll}.} Recall that given a simplicial complex $X$ and simplices $\sigma \subsetneq \tau \in X$, we say that $\sigma$ is a \emph{free face} of $\tau$ if $\sigma$ is not contained in any simplex other than itself and $\tau$; an \emph{elementary collapse} is then the removal of such a pair of faces $\{ \sigma,\tau \}$ from $X$. We aim to find a sequence of elementary collapses that transforms $X_i$ to $X_{i+1}$. We do this in two steps: we first transform $X_i$ to $X_i'$, where $X_i'$ consists of precisely those simplices of $X_i$ that do not contain edges of the form $\{S,T\}$ with $S \sqsubsetdot T_i$,  and then transform $X_i'$ to $X_{i+1}$. Let $T = T_i$, and let $\Lambda$ be the lattice such that $T$ has type $\calt(\Lambda)$.

The first of these transformations is trivial if $X_i' = X_i$, so suppose that $X_i' \neq X_i$, so that $S \sqsubsetdot T$ for some $S \in X_i$. By our choice of the ordering of the trees $T_j$, it follows that $S \in \widehat{X}$, and so $S$ is a $G$-tree of type $\calt(\Phi(L_k))$ for some $k \geq 1$ or of type $\cals$ (in the latter case, we set $k := 0$). In particular, we have $\Phi(L_k) < \Lambda < A^{-1}\Phi(L_k) = \Phi(L_{k+1})$, and so by Lemma~\ref{lem:unique}\ref{it:ind-unique} there exists a unique $G$-tree $U = \ind{T}{\Phi(L_{k+1})}$ of type $\calt(\Phi(L_{k+1}))$ such that $T \sqsubsetdot U$. It then follows from Lemma~\ref{lem:unique}\ref{it:ind-ind} (in conjunction with Lemma~\ref{lem:unique}\ref{it:ind-exp} when $k = 0$) that for any $S \in \widehat{X}$ with $S \sqsubsetdot T$ we have $S \sqsubsetdot U$; similarly, it follows that if $\widehat{T} \in X_i$ and $S \in \widehat{X}$ are such that $S \sqsubsetdot T \sqsubsetdot \widehat{T}$ and $S \sqsubsetdot \widehat{T}$, then either $\widehat{T} \sqsubsetdot U$ or $\widehat{T} = U$ (depending on whether $\widehat{T}$ is contained in $\widehat{X}$). We may thus deduce that if $\sigma$ is a simplex in $X_i$ but not in $X_i'$, then so is $\sigma \cup \{U\}$, and in particular if $\sigma \sqcup \{U\}$ is a maximal simplex then $\sigma$ is its free face (indeed, if $\sigma \subseteq \sigma'$ for a simplex $\sigma'$ then $\sigma' \cup \{U\} = \sigma \sqcup \{U\}$ by maximality, and therefore $\sigma' = \sigma$ or $\sigma' = \sigma \sqcup \{U\}$). This allows us to obtain $X_i'$ from $X_i$ using elementary collapses by repeatedly removing pairs of the form $\left\{ \sigma, \sigma \sqcup \{U\} \right\}$, where $\sigma \sqcup \{U\}$ is a simplex that is maximal (in an intermediate simplicial complex) among simplices containing edges of the form $\{S,T\}$ with $S \sqsubsetdot T$; see Figure~\ref{sfig:dr-xi'}.

In order to transform $X_i'$ to $X_{i+1}$, let $r \geq 0$ be the largest such that $\Phi(L_r) < \Lambda$, 
and note that $\Lambda < \Lambda+\Phi(L_{r+1}) = \Lambda + A^{-1}\Phi(L_r) \leq A^{-1}\Lambda$. Let $\widehat{T} \coloneqq \ind{T}{\Lambda+L_{r+1}}$ and $U \coloneqq \ind{T}{A^{-1}\Lambda}$, as in Lemma~\ref{lem:unique}\ref{it:ind-unique}, and note that $X_i'$ contains the face $\{ T, \widehat{T}, U \}$. Any maximal simplex of $X_i'$ that is not a simplex of $X_{i+1}$ has the form $\sigma = \{ T=S_0,S_1,\ldots,S_k \}$, where $S_j \sqsubsetdot S_{j'}$ for $j < j'$, so that $S_j$ has type $\calt(\Lambda_j)$ and $\Lambda = \Lambda_0 < \Lambda_1 < \cdots < \Lambda_k \leq A^{-1}\Lambda$; it then follows from Lemma~\ref{lem:unique}\ref{it:ind-ind} that adding the vertex $U$ would result in a simplex contained in $X_i'$, and so $U \in \sigma$ by maximality. Similarly as in the previous paragraph, we may use elementary collapses by repeatedly removing pairs of the form $\left\{ \sigma, \sigma \sqcup \{U\} \right\}$, where $\sigma \sqcup \{U\}$ is a simplex that is maximal (in an intermediate simplicial complex) among simplices not contained in $X_{i+1}$ or $\{ T, \widehat{T}, U \}$, to obtain the simplicial complex $X_i'' = X_{i+1} \cup \{ \{T\},\{T,\widehat{T}\},\{T,U\},\{T,\widehat{T},U\} \}$. Two further elementary collapses---removing $\{ \{T,U\},\{T,\widehat{T},U\} \}$ and then removing $\{ \{T\},\{T,\widehat{T}\} \}$---transform $X_i''$ into $X_{i+1}$, as required; see Figure~\ref{sfig:dr:xi+1}.  
\hfill\popQED

\smallskip

The order of the last two collapses is not unique.  Indeed, first removing $\{ \{T,\widehat{T}\}, \{T,\widehat{T},U\} \}$ and then $\{ \{T\}, \{T,U\} \}$ would also work.  However the choice we have made is crucial in \ref{it:dr-dr}.  In particular, we will show that the tree $\widehat{T}$ has lower \emph{complexity} than $T$, as defined below, whereas $U$ does not.

\begin{figure}[ht]
    \centering
\begin{subfigure}[t]{0.3\textwidth}
\centering
\begin{tikzpicture}
    \draw [dashed] (0,4) -- (0,-1) node [below] {$c=0$};
    \draw [dashed] (-2,3) -- (-2,0) node [below] {$c=1$};
    \fill [blue,fill opacity=0.2] (0,0) -- (-2,2) -- (0,3);
    \fill [blue,fill opacity=0.2] (-2,1) -- (-2,2) -- (0,3);
    \fill [red!50!blue,fill opacity=0.1] (0,0) -- (-2,1) -- (-1.5,1.5);
    \draw [ultra thick,green!70!black] (0,0) -- (0,3);
    \draw [very thick,red!50!blue] (0,0) -- (-2,1);
    \draw [very thick,blue] (0,0) -- (-2,2) (0,3) -- (-2,1) (-2,1) -- (-2,2) -- (0,3);
    \draw [fill=green!70!black] (0,0) circle (3pt) node [right] {$S$} (0,3) circle (3pt) node [right] {$U$};
    \draw [fill=blue] (-2,1) circle (2pt) node [left] {$T$} (-2,2) circle (2pt) node [left] {$\widehat{T}$};
\end{tikzpicture}
    \caption{Constructing $X_i'$ (blue) by collapsing the simplices of $X_i$ containing $\{S,T\}$ with $S \sqsubsetdot T$ (purple).}
\label{sfig:dr-xi'}
\end{subfigure}\hfill
\begin{subfigure}[t]{0.65\textwidth}
\centering
\begin{tikzpicture}[x=1.4cm]
    \foreach \i in {1,...,20}
        \fill [blue!\i] (0,{5.3-0.015*\i}) -- ({-4.3+0.015*\i},{3.15-0.0075*\i}) -- (-4,3) -- (0,5);
    \foreach \i in {1,...,15}
        \fill [black!\i] ({-4.3+0.02*\i},{3.15-0.01*\i}) -- ({-5.3+0.02*\i},{2.65-0.01*\i}) -- (-5,2.5) -- (-4,3);
    \draw [dashed] (0,5.5) -- (0,-0.5) node [below] {$c=0$};
    \draw [dashed] (-1,5) -- (-1,0) node [below] {$c=1$};
    \draw [dashed] (-2,4.5) -- (-2,0.5) node [below] {$c=2$};
    \draw [dashed] (-3,4) -- (-3,1) node [below] {$c=3$};
    \draw [dashed] (-4,3.5) -- (-4,1.5) node [below] {$c=4$};
    \draw [dashed] (-5,3) -- (-5,2) node [below] {$c=5$};
    \fill [blue!20] (0,1) -- (-4,3) -- (0,5);
    \fill [red!10] (0,0) -- (0,2) -- (-2,3) -- (-4,2);
    \fill [black!15] (-4,2) -- (-5,2.5) -- (-4,3) -- (-3,2.5);
    \fill [red!50!blue!20] (-2,2) -- (-3,2.5) -- (-2,3) -- (-1,2.5);
    \draw [very thick,red,densely dotted] (0,0) -- (-4,2) (0,1) -- (-2,2) (-2,2) -- (-3,1.5) (0,2) -- (-2,1) (0,1) -- (-1,0.5) (-1,0.5) -- (-1,2.5) (-2,1) -- (-2,2) (-3,1.5) -- (-3,2.5);
    \draw [very thick,blue] (-3,2.5) -- (-4.2,3.1) (0,2) -- (-3.2,3.6) (0,3) -- (-2.2,4.1) (0,4) -- (-1.2,4.6) (0,5) -- (-0.2,5.1) (0,5) -- (-4,3) (0,4) -- (-3,2.5) (0,3) -- (-1,2.5) (-1,2.5) -- (-1,4.7) (-2,3) -- (-2,4.2) (-3,2.5) -- (-3,3.7) (-4,3) -- (-4,3.2);
    \draw [very thick,red!50!blue] (-2,2) -- (-3,2.5) (-2,2) -- (-1,2.5) (-2,2) -- (-2,3);
    \draw [very thick,black!60] (-4,2) -- (-5.2,2.6) (-4,2) -- (-3,2.5) (-4,2) -- (-4,3) (-5,2.5) -- (-5,2.7) (-4,3) -- (-5,2.5);
    \draw [ultra thick,green!70!black] (0,0) -- (0,5.2);
    \draw [fill=green!70!black] (0,0) circle (3pt) node [right] {$\cals$} (0,1) circle (3pt) node [right] {$\calt(\Phi(L_1))$} (0,2) circle (3pt) node [right] {$\calt(\Phi(L_2))$} (0,3) circle (3pt) node [right] {$\calt(\Phi(L_3))$} (0,4) circle (3pt) node [right] {$\calt(\Phi(L_4))$} (0,5) circle (3pt) node [right] {$\calt(\Phi(L_5))$};
    \draw [fill=red] (-1,0.5) circle (2pt) (-1,1.5) circle (2pt) (-2,1) circle (2pt) (-3,1.5) circle (2pt);
    \draw [fill=black!60] (-4,2) circle (2pt) (-5,2.5) circle (2pt);
    \draw [fill=red!50!blue] (-2,2) circle (2pt) node [below,xshift=5pt] {$T$};
    \draw [fill=blue] (-2,3) circle (2pt) node [above,xshift=5pt] {$U$} (-1,2.5) circle (2pt) node [above,xshift=5pt] {$\widehat{T}$} (-1,3.5) circle (2pt) (-1,4.5) circle (2pt) (-2,4) circle (2pt) (-3,2.5) circle (2pt) (-3,3.5) circle (2pt) (-4,3) circle (2pt);
    \draw [red!50!blue,->] (-2.5,2.25) to [bend left] node [pos=0.3,above] {\tiny A} (-2.2,2.5);
    \draw [red!50!blue,->] (-2,2.5) to [bend left] node [pos=0.6,above] {\tiny B} (-1.7,2.35);
    \draw [red!50!blue,->,yshift=-0.1cm] (-1.7,2.15) to node [pos=0.6,below] {\tiny C} (-1.3,2.35);
\end{tikzpicture}
    \caption{Constructing $X_{i+1}$ (blue) by collapsing the simplices of $X_i'$ containing $T$ (purple). The red simplices had already been collapsed in the previous steps; the grey ones might have and might have not. The labelled arrows represent (A) constructing $X_i''$, (B) collapsing $\{T,U\}$, and finally (C) collapsing $T$.}
\label{sfig:dr:xi+1}
\end{subfigure}
    \caption{An example of the construction of $X_{i+1}$ from $X_i$, as in step~\ref{it:dr-elcoll} in the proof of \Cref{prop:dr}.  The pictures are drawn so that increasing induction and $\scra$-moves are always represented as going up (rather than down) along an edge. The complex $\widehat{X}$ is shown in green, and the values of $c$ refer to the complexity of trees, as defined in step~\ref{it:dr-dr}.}
    \label{fig:dr}
\end{figure}
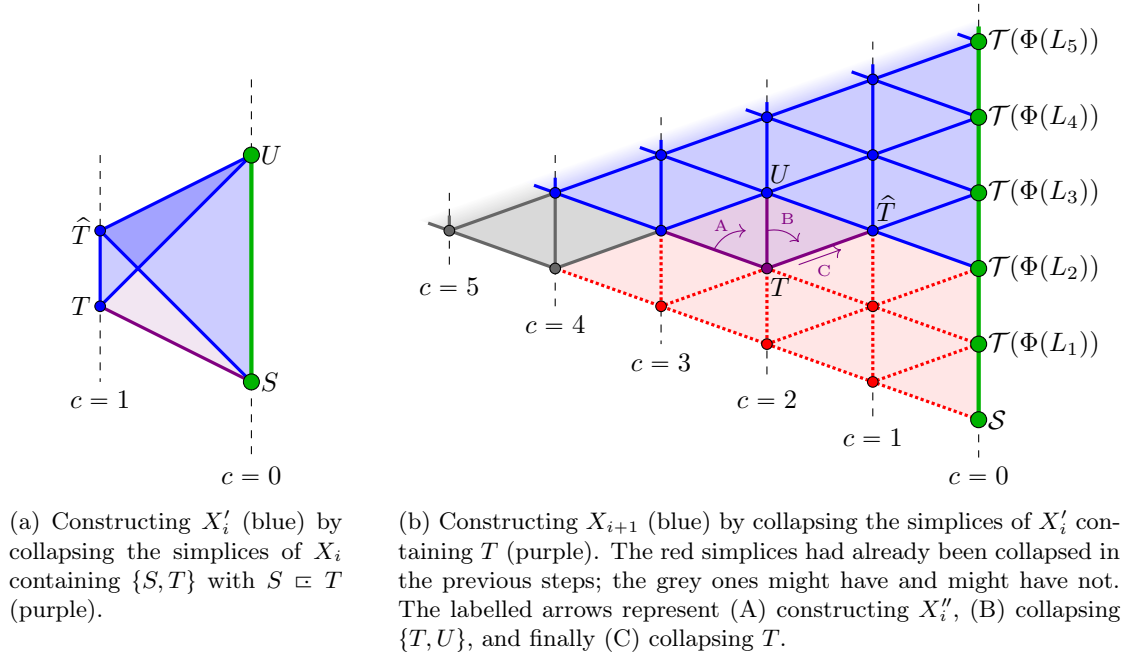

\medskip\pushQED{$\blackdiamond$}

\noindent\emph{Proof of \ref{it:dr-dr}.} Note that removing a pair of simplices $\{ \sigma,\tau \}$ of a simplicial complex $X$ via an elementary collapse induces a deformation retraction of $X$ onto $X \smallsetminus \{\sigma,\tau\}$. Therefore, since there is a sequence of (finitely many) elementary collapses transforming $X_i$ to $X_{i+1}$ as above, we have induced deformation retractions $\calh_i\colon X_i \times [0,1] \to X_i$ of $X_i$ onto $X_{i+1}$. We construct a deformation retraction of $\widehat{W}$ to $\widehat{X}$ as a composition of the deformation retractions $\calh_i$ done at the interval $[1-2^{-i},1-2^{-i-1}]$; in particular, we define a function $\calh\colon \widehat{W} \times [0,1) \to \widehat{W}$ so that for $t \in [1-2^{-i},1-2^{-i-1})$ we have $\calh(x,t) = \calh_i((r_{i-1} \circ \cdots \circ r_0)(x), 2-2^{i+1}(1-t))$ for all $x \in \widehat{W}$, where $r_i\colon X_i \to X_{i+1}$ is defined by $r_i(x) = \calh_i(x,1)$.

It is clear that the function $\calh$ is continuous. We claim that it extends to a deformation retraction $\calh\colon \widehat{W} \times [0,1] \to \widehat{W}$ of $\widehat{W}$ to $\widehat{X}$. In order to show that, since $\widehat{W}$ is locally finite and $\calh$ is induced by a sequence of elementary collapses, it is enough to show that for any $T \in \widehat{W}$, there exists $\varepsilon > 0$ such that the function $t \mapsto \calh(T,t)$ is constant on $(1-\varepsilon,1)$ and equal to some $\widehat{T} \in \widehat{X}$, in which case we set $\calh(T,1) = \widehat{T}$.

We define the \emph{complexity} $c(T)$ of a $G$-tree $T \in \widehat{W}$ as follows. If $T$ is of type $\cals$, we set $c(T) = 0$. If $T$ is of type $\calt(\Lambda)$, then we have $\Phi(L_0) \leq \Lambda$, and so we may choose the largest $r \geq 0$ such that $\Phi(L_r) \leq \Lambda$. Since $\Lambda$ is finitely generated and $\Lambda \leq \bigcup_{n=0}^\infty \Phi(L_n)$, we may choose the smallest $n \geq r$ such that $\Lambda \leq \Phi(L_n)$. We then set $c(T) = n-r$. It is clear that $T \in \widehat{X}$ if and only if $c(T) = 0$.

Now let $T \in \widehat{W}$, and let $T = S_0,S_1,\ldots$ be all the distinct $G$-trees appearing (in this order) in the sequence $\left( (r_{i-1} \circ \cdots \circ r_0)(T) \right)_{i=0}^\infty$. Note that $S_{k+1} = r_i(S_k)$, where $i \geq 0$ is such that $S_k = T_i$, as long as $S_k \notin \widehat{X}$. Hence, by the construction of $\calh_i$, it follows that if $S_k$ has type $\calt(\Lambda)$ then $S_{k+1}$ has type $\calt(\Lambda')$, where $\Lambda' = \Lambda+\Phi(L_{r+1})$ for $r \geq 0$ largest such that $\Phi(L_r) < \Lambda$. 
We now compare $c(S_{k+1})$ to $c(S_k)$. Note that if $\Lambda \leq \Phi(L_n)$ but $\Lambda \nless \Phi(L_{n+1})$ then $\Lambda < \Lambda + \Phi(L_{r+1}) \leq \Phi(L_n)$, so $\Lambda'$ is contained in $\Phi(L_n)$ but not in $\Phi(L_{n-1})$. On the other hand, we now have $\Phi(L_{r+1}) \leq \Lambda'$, and we claim this is optimal: if we had $\Phi(L_{r+2}) \leq \Lambda' \leq A^{-1}\Lambda$, then by applying $A$ we see that $\Phi(L_{r+1})=A\Phi(L_{r+2}) \leq AA^{-1}\Lambda = \Lambda$. So $c(S_{k+1}) = n-(r+1) = c(S_k)-1$. 
This shows that the sequence $T = S_0,S_1,\ldots$ terminates in $c(T)$ steps, and we have $\widehat{T} \coloneqq S_{c(T)} \in \widehat{X}$, 
a tree of type $\calt(\Phi(L_n))$ obtained from $T$ by a sequence of increasing inductions. If $j \geq 0$ is such that $S_{c(T)-1} = T_j$, it then follows that the function $t \mapsto \calh(T,t)$ is constant and equal to $S_{c(T)} \in \widehat{X}$ on $(1-2^{-j-1},1)$, as required. \hfill\popQED

\medskip\pushQED{$\blackdiamond$}

\noindent\emph{Proof of \ref{it:dr-equiv}.} We now describe the modifications of the construction above to replace the deformation retraction defined in step~\ref{it:dr-dr} with an $\Out(G)$-equivariant one.

In step~\ref{it:dr-enum}, after enumerating the vertices of $\widehat{W} \smallsetminus \widehat{X}$ as $T_0,T_1,\ldots$, we also enumerate their orbits under the $\Out(G)$-action (which is possible since, by Lemma~\ref{lem:Xhat}, $\widehat{X}$ and so $\widehat{W} \smallsetminus \widehat{X}$ is $\Out(G)$-invariant), as follows. Given a $G$-orbit $\calo \subseteq \widehat{W} \smallsetminus \widehat{X}$, we write $\calo = \{ T_{i_0},T_{i_1},\ldots \}$ so that $i_j < i_{j'}$ for $j < j'$, and let $i(\calo) := i_0$. We then list the orbits as $\calo_0,\calo_1,\ldots$ in increasing order of the values of $i(\calo_j)$, so that if $j < j'$ then $i(\calo_j) < i(\calo_{j'})$. We claim that if $T \sqsubsetdot T'$ for some $T \in \calo$ and $T' \in \calo'$, where $\calo$ and $\calo'$ are $\Out(G)$-orbits in $\widehat{W} \smallsetminus \widehat{X}$, then $i(\calo) < i(\calo')$. Indeed, in this case we have $T = T_i \cdot \phi$ and $T' = T_{i'} \cdot \phi'$ for some $\phi,\phi' \in \Out(G)$, where $i = i(\calo)$ and $i' = i(\calo')$, implying that $T_i \cdot \phi (\phi')^{-1} \sqsubsetdot T_{i'}$. If $j \geq 0$ is such that $T_i \cdot \phi (\phi')^{-1} = T_j$, we then have $i(\calo) \leq j < i' = i(\calo')$, as claimed.

In step~\ref{it:dr-elcoll}, instead of making elementary collapses one-by-one, we make them $\Out(G)$-equivariantly: when we remove a pair $\{ \sigma,\tau \}$, where $\sigma$ is a free face of a maximal simplex $\tau$, we remove at the same time the pair $\{ \phi\cdot\sigma,\phi\cdot\tau \}$ for all $\phi \in \Out(G)$. Note that any two such distinct pairs $\{ \sigma,\tau \}$ and $\{ \phi\cdot\sigma,\phi\cdot\tau \}$ are pairwise disjoint: indeed, if $\sigma = \phi \cdot \sigma$ then $\tau$ and $\phi \cdot \tau$ both coincide with the unique maximal simplex containing $\sigma$, and so $\tau = \phi \cdot \tau$; conversely, if $\tau = \phi \cdot \tau$ then also $\sigma = \phi \cdot \sigma$ since the action of $\Out(G)$ on $\widehat{W}$ is admissible by Proposition~\ref{prop:What}. This implies that the simultaneous removal of the pairs $\{ \phi\cdot\sigma,\phi\cdot\tau \}$ for all $\phi \in \Out(G)$ induces an $\Out(G)$-equivariant deformation retraction. Hence the maps $\calh_i$ and therefore $\calh$ in step~\ref{it:dr-dr} become $\Out(G)$-equivariant as well, as required. \hfill\popQED

\medskip

As we have verified all four steps, the proof is complete.
\end{proof}

Summarising our results so far, we prove \Cref{thmx:deformation_retract}.

\setcounter{thmx}{3}
\begin{thmx}\label{thmx:deformation_retract}
    If $G = \LM(A,L)$ with either $L < AL < \ZZ^m$ or $AL < L < \ZZ^m$, then the deformation space of reduced $G$-trees $\Out(G)$-equivariantly deformation retracts to a simplicial tree.
\end{thmx}

\begin{proof}
    We assemble the results in this section. The deformation space $\cald$ equivariantly deformation retracts to the subcomplex $W$ of its spine by Clay's Theorem~\ref{thm:dr-to-W}. By Proposition~\ref{prop:What}, the $W$  is $\Out(G)$-equivariantly isomorphic to the barycentric subdivision of $\widehat{W}$, where the action remains admissible. By Proposition~\ref{prop:dr} there is an $\Out(G)$-equivariant deformation retract from $\widehat{W}$ to $\widehat{X}$, which is a tree since it is one-dimensional by Lemma~\ref{lem:Xhat}, and contractible since this was true of the original deformation space.
\end{proof}

\subsection{The structure of \texorpdfstring{$\Out(G)$}{Out(G)}}

We now use Theorem~\ref{thmx:deformation_retract}, together with our analysis of stabilisers in Section~\ref{sec:non_metab_intro}, to prove the remainder of Theorem~\ref{thmx:not_metabelian}. Recall the specific $G$-trees $S$ (of type $\cals$, shown in Figure~\ref{fig:gog_for_LMAL}) and $T_n$ (of type $\calt(\Phi(L_n))$, shown in Figure~\ref{fig:BAD_trees}); these trees, unlike an arbitrary reduced tree, all have stable letter $t$.

\begin{prop}
    The quotient graph $\widehat{X}/\Out(G)$ is an infinite ray. Representatives of the vertices are given by the $G$-trees $S$ and $T_n$, and furthermore these trees span a connected fundamental domain for the action.
\end{prop}

The quotient is shown in Figure~\ref{fig:quotient}.
\begin{figure}[ht]
    \centering
\begin{tikzpicture}
\draw [very thick] (0,0) -- (7,0);
\draw [very thick,dotted] (7,0) -- (8,0);
\draw [thick,fill=white] (0,0) circle (0.1) node [below=0.1] {$S$};
\draw [thick,fill=white] (2,0) circle (0.1) node [below=0.1] {$T_1$};
\draw [thick,fill=white] (4,0) circle (0.1) node [below=0.1] {$T_2$};
\draw [thick,fill=white] (6,0) circle (0.1) node [below=0.1] {$T_3$};
\end{tikzpicture}
    \caption{The quotient $\widehat{X}/\Out(G)$.}
    \label{fig:quotient}
\end{figure}
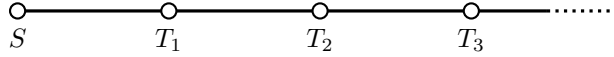

\begin{proof}
The vertices of $\widehat{X}$ are $G$-trees of types $\cals$ and $\calt(\Phi(L_n))$ for $n \geq 1$. Any two $G$-trees of the same type are in the same orbit under the $\Out(G)$-action by Lemma~\ref{lem:same-type-is-iso}. Conversely, if two $G$-trees $T,T' \in \widehat{X}$ are in the same orbit then they have the same type, because of the following two reasons.
\begin{itemize}
\item The $G$-trees of type $\cals$ (respectively $\calt(\Phi(L_n))$) have one (respectively two) $G$-orbits of vertices, showing that either both $T$ and $T'$ have type $\cals$, or they have types $\calt(\Phi(L_n))$ and $\calt(\Phi(L_r))$ for some $n,r \geq 1$.
\item If $T$ and $T'$ have types $\calt(L_n)$ and $\calt(L_r)$, respectively, then the degrees of the vertices in $T$ (respectively $T'$) in the two $G$-orbits are $[H:L_0]$ and $1+\left|\det(A)\right|+\left|\det(A)\right|^n$ (respectively $[H:L_0]$ and $1+\left|\det(A)\right|+\left|\det(A)\right|^r$), implying that $n=r$.
\end{itemize}
This shows that two $G$-trees in $\widehat{X}$ are in the same $\Out(G)$-orbit if and only if they have the same type. Since $S$ is of type $\cals$ and $T_n$ is of type $\calt(\Phi(L_n))$, these trees serve as orbit representatives. Note that there is a $\scra$-move transforming $S$ into $T_1$, and there are increasing inductions (in fact slides) transforming $T_n$ into $T_{n+1}$, and so there are corresponding edges in $\widehat{X}$ (see Lemma~\ref{lem:Xhat}). So the vertices $\{S,T_1,T_2,\dots\}$ span an infinite ray in $\widehat{X}$. There are no other edge orbits, and so this is a strict fundamental domain for the action, and the quotient graph is an infinite ray.
\end{proof}

We now use this graph of groups to describe $\Out(G)$. Recall that $H \cong \ZZ^m$ is the vertex group of $S$; that $L_n=t^{-n}Lt^n$ is the non-leaf vertex group of $T_n$ and $L_\infty=\bigcup_{n=1}^\infty L_n$; and that $\calc^+ = \{ M \in \GL_m(\ZZ) \mid MA=AM, M(L) = L \}$.

\begin{thm} \label{thm:OutG-semidir-nonfg}
We have a semidirect product splitting
\[
\Out(G) \cong \left(H/(I-A)L  \aster_{L/(I-A)L} L_\infty/(I-A)L\right) \rtimes \calc^+
\]
where the action is induced from $M \cdot v = M(v)$ for $M \in \calc^+$ and $v \in H$ or $v \in L_\infty$. In particular, $\Out(G)$ is not finitely generated.
\end{thm}

\begin{proof}
    We use the results and notation of Section~\ref{sec:non_metab_intro}. Specifically: \begin{itemize}
        \item $\Stab_{\Out(G)}(S) \cong \overline \calb \rtimes \calc^+$, with the isomorphism induced by $\delta \colon H \rtimes \calc^+ \to \Aut(G)$ (Proposition~\ref{prop:stab_S}), and $\overline{\calb} \cong H/(I-A)L$ (Proposition~\ref{prop:H_to_aut_nonmet});
        \item $\Stab_{\Out(G)}(T_n) \cong \overline \calb_n \rtimes \calc^+$, with the isomorphism induced by $\delta_n \colon L_n \rtimes \calc^+ \to \Aut(G)$ (Proposition~\ref{prop:stab_Tn}), and $\overline{\calb}_n \cong L_n/(I-A)L$ (Proposition~\ref{prop:Ln_to_aut_nonmet});
        \item $\Stab_{\Out(G)}(S) \cap \Stab_{\Out(G)}(T_n) \cong  L/(I-A)L \rtimes \calc^+$ (Remark~\ref{rem:inters}(2));
        \item There are inclusions $\Stab_{\Out(G)}(T_n) \rightarrowtail \Stab_{\Out(G)}(T_{n+1})$ (Remark~\ref{rem:inters}(1)).
    \end{itemize}

We view $\widehat{X}/\Out(G)$ as a graph of groups and calculate its Bass--Serre theoretic fundamental group. Since $\Stab_{\Out(G)}(T_n) < \Stab_{\Out(G)}(T_{n+1})$, we have that \[\Out(G) \cong (\overline{\calb} \rtimes \calc^+) \ast_{L/(I-A)L \rtimes \calc^+} \bigcup_{n=1}^\infty (\overline{\calb}_n \rtimes \calc^+ ).\]
Note that the descriptions of $\overline\calb_n$ give that \[\bigcup_{n=1}^\infty \overline\calb_n = \bigcup_{n=1}^\infty L_n/(I-A)L = \left(\bigcup_{n=1}^\infty L_n\middle)\right/(I-A)L = L_\infty/(I-A)L.\] Finally, it follows from the explicit descriptions that the actions of $\calc^+$ in each factor agree on the intersections, allowing us to move it `outside' the amalgam.

The fact that $\Out(G)$ is not finitely generated can be seen geometrically from the that the action of $\Out(G)$ on $\widehat{X}$ is minimal (e.g. because a product of $\varphi \in \Stab(S)$ and $\psi \in \Stab(T_n)$ is a hyperbolic element whose axis includes the segment joining $S$ and $T_n$ in $\widehat{X}$) but is not cocompact. Algebraically, it is witnessed by the fact that $\Out(G)$ is the strictly increasing union of the groups $(\overline{\calb} \rtimes \calc^+) \ast_{L/(I-A)L \rtimes \calc^+} \bigcup_{n=1}^k (\overline{\calb}_n \rtimes \calc^+ ).$
\end{proof}

This completes the proof of Theorem~\ref{thmx:not_metabelian}.

\setcounter{thmx}{0}
\begin{thmx}
    \label{thmx:not_metabelian}
    Let $G=\LM(A,L)$ and suppose $G$ is not metabelian.  $\Out(G)$ is finitely generated if and only if $L$ and $AL$ do not properly contain each other.
\end{thmx}
\begin{proof}
    This follows from combining \Cref{thm:nonMet_VF_Foo_fin} with \Cref{thm:OutG-semidir-nonfg}.
\end{proof}

\section{The metabelian case}\label{sec:metabelian}

In this section we treat the case of metabelian Leary--Minasyan groups.  To this end we prove \Cref{thm:metabelian}, from which \Cref{thmx:metabelian} of the introduction easily follows.  Almost the entire section is dedicated to proving \Cref{thm:metabelian}, but before that to aid the reader we recall the definition of a matrix being essentially Jordan gap free.

\begin{defn}
Let $A \in \GL_n(\QQ)$ and let $\lambda$ be an eigenvalue of $A$. We say an integer $d \geq 1$ is a \emph{$\lambda$-Jordan gap} for $A$ if the Jordan normal form for $A$ has no $\lambda$-Jordan blocks of size $d$, but has at least one $\lambda$-Jordan block of size $d+1$. We say $A$ is \emph{essentially Jordan gap free} if all $\lambda$-Jordan gaps for $A$ occur when $\lambda$ is an algebraic unit.
\end{defn}

The main goal of this section is to prove the following theorem:

\begin{thm}\label{thm:metabelian}
    Let $A\in \GL_n(\QQ)$ have integer entries and let $G=\LM(A,L)$ with $L=\ZZ^n$.  Then, $\Out(G)$ is finitely generated if and only if $A$ is essentially Jordan gap free.
\end{thm}

 Before we prove \Cref{thm:metabelian}, we prove \Cref{thmx:metabelian} from the introduction.

\setcounter{thmx}{1}
\begin{thmx}\label{thmx:metabelian}
     Let $G=\LM(A,L)$ and suppose $G$ is metabelian.  $\Out(G)$ is finitely generated if and only if $A$ is essentially Jordan gap free.
\end{thmx}
\begin{proof}
    By \Cref{thm:metabelian}, it suffices to show that any metabelian $\LM(A,L)$ is isomorphic to one with $A\in \GL_n(\QQ)\cap M_n(\ZZ)$ and with $L=\ZZ^n$.  The condition on $L$ is trivial since for $\LM(A,L)$ to be metabelian, it must be an ascending HNN extension so at least one of $L$ and $AL$ are equal (as sets) to $\ZZ^n$.  Thus, it remains to verify the condition on $A$.  To this end, we first note that we can always conjugate the matrix $A$ or its inverse to an integral one.  Indeed, since $AL\subseteq\ZZ^n$ we see that $A$ defines an embedding $AL\to L$.   Thus, up to replacing $L$ by $AL$ and $A$ by a conjugate and potentially its inverse we are in the situation of \Cref{thm:metabelian} as required.
\end{proof}

\subsection{Proof strategy and example}
   The proof of \Cref{thm:metabelian} is technical and of a distinct flavour to the rest of the article.  Before we begin, for the aid of the reader, we will describe the key steps and then demonstrate them with an example.  We now describe the strategy:
\begin{enumerate}[1.]
    \setcounter{enumi}{-1}
    \item (Notation) Let $K=\bigcup_{s\in \ZZ} A^s(L)$; $f=\prod_{i=1}^lp_i^{\alpha_i}(X)$ be the characteristic polynomial of $A$ with each $p_i$ irreducible over $\QQ$ and each $\alpha_i\geq 1$; $\QQ^n=\prod_{i=1}^l V_i$ with each $V_i$ an $A$-invariant subspace.
    \item (\Cref{metab:step1}) We show $\Out(G)$ is finitely generated if and only if $C=C_{\GL(K)}(A)$ is finitely generated.
    \item (\Cref{metab:step2}) We recast the problem into one of automorphisms of a module over a principal ideal domain.\\  More precisely: Fix some $p=p_i$, $\alpha=\alpha_i$ and let $F$ be the splitting field for $p$ over $\QQ$ (we work one $A$-indecomposable subspace at a time).  Let $\lambda$ be a root of $p$.  We choose a Jordan basis of the generalised $\lambda$-eigenspace in $M\leqslant F^n$ for $A$ and construct a preferred basis $\calb$ for $\QQ^n$. Using the basis we show that $M$ is an $F(X)$-module, where $X$ acts by $X(v)=Av-\lambda v$.  Using this we show that $C_{\GL_n(\QQ)}(A)\simeq \Aut_{F(X)}(M)$.
    \item (\Cref{metab:step3}) Here we show that changing the lattice $L<\QQ^n$ within its commensurability class does not greatly affect the set of elements of $\GL_n(\QQ)$ which preserve $\bigcup_{s\in\ZZ} A^s L$. 
    \item (\Cref{metab:step4}) In this section we conclude the proof.  There are two cases:
    \begin{enumerate}[(a)]
        \item \label{it:A-has-gap} If $A$ has a $\lambda$-Jordan gap, then there exists an epimorphism $C\twoheadrightarrow (\ZZ[\lambda^{-1}],+)$.  This essentially follows from analysing the module structure.
        \item If $A$ is essentially Jordan gap free, then $C$ is finitely generated.  Working one eigenvalue at a time there are two subcases: either $\lambda^{-1}$ is an algebraic integer, in which case $K$ is finitely generated and so $C = C_{\GL(K)}(A)$ is arithmetic; 
        or $A$ has no $\lambda$-Jordan gaps.   This latter case is more complicated, here we argue directly that the centraliser is generated by a finite list of elements and finitely many groups $\Gamma_d=\GL_{m_d}(Z)$ where $Z$ is a ring of $S$-integers.  Thus, finite generation follows from the classical fact that each $\Gamma_d$ is finitely generated \cite[Proposition~6.10]{BorelSerre1976}.
    \end{enumerate}
\end{enumerate}

\begin{example}
We now provide an example with
\[
A = \begin{pmatrix}
    2&0&0 \\ 0&2&0 \\ 0&1&2
\end{pmatrix}.
\]
We would like to (Proposition~\ref{prop:metab-Out}) compute $C_{\GL(K)}(A)$, where $K = \bigcup_{s \in \ZZ} A^s(\ZZ^3) = \ZZ[1/2]^3$. We have one eigenvalue ($2$) and two Jordan blocks (of sizes $1$ and $2$).  We will study the $Z[X]$-module $M = Z[X]/(X) \oplus Z[X]/(X^2)$, where $Z = \ZZ[1/2]$ and $X$ acts as $A-2I$. Then $C_{\GL(K)}(A)$ is commensurable with $\Aut_{Z[X]}(M)$, and an element $\xi$ of the latter can be written as a ``matrix'' (the entries come from different rings) of the form
\[
B(\xi) = \begin{pmatrix}
    \beta_{1,1} & \beta_{1,2} \\ \beta_{2,1} & \beta_{2,2}
\end{pmatrix},
\]
where the $\beta_{i,j} \colon Z[X]/(X^j) \to Z[X]/(X^i)$ are $Z[X]$-module homomorphisms. In particular, (after identifying a homomorphism of cyclic $Z[X]$-modules with an element of $Z[X]$ multiplication by which sends a generator to a generator) we may write 
\[
B(\xi) = \begin{pmatrix}
    z_1 & z_2 \\ z_3X & z_4+z_5X
\end{pmatrix}
\]
for some $z_1,\ldots,z_5 \in Z$, where the polynomials in the $i$-th row should be seen as elements of $Z[X]/(X^i)$. Matrix multiplication then works as usual to describe the group operation in $\Aut_{F[X]}(M)$.

Conversely, an element $\xi \in \End_{Z[X]}(M)$ described by such a ``matrix'' is an automorphism of $M$ if and only if the matrix consisting of constant coefficients, i.e. $c(\xi) = \begin{pmatrix}
    z_1 & z_2 \\ 0 & z_4
\end{pmatrix}$, is invertible over $Z$, i.e.\ an element of $\GL_2(Z)$.

[A reason for the previous paragraph (note that this is essentially the argument in Lemma~\ref{lem:inverse-condition}):  If $\xi$ is an automorphism of $M$ then it induces an automorphism $\overline\xi$ of $M/XM \cong Z^2$, and $c(\xi)$ is precisely the matrix representing $\overline\xi$, so it must be an element of $\GL_2(Z)$. Conversely, if $c(\xi)$ is invertible over $Z$, then we claim that $\xi$ is surjective (this will imply that $\xi$ is injective as well, since $\xi$ is a linear map acting on a subspace $K < \QQ^3$ of full rank, so if it is surjective, then it sends a $\QQ$-basis to a $\QQ$-basis). Since $c(\xi)$ represents a surjective linear map on $Z^2$, it follows given any $w_1,w_2 \in Z$ there exist $m \in M$ and $w' \in Z$ such that \[\xi(m) = \begin{pmatrix}
    w_1 \\ w_2+w'X
\end{pmatrix}.\] In particular, there exists $m' \in M$ such that \[\xi(m') = \begin{pmatrix}
    0\\1+yX
\end{pmatrix}\] for some $y \in Z$, and then given $w_3 \in Z$ we have \[\xi(m+(w_3-w')Xm') = \begin{pmatrix}
    w_1 \\ w_2+w_3X
\end{pmatrix},\] implying that $\xi$ is surjective, as claimed.]

We thus aim to show that the group of all $B(\xi)$ as above with $c(\xi)$ invertible is finitely generated. Note that $c(\xi) = \begin{pmatrix}
    z_1 & z_2 \\ 0 & z_4
\end{pmatrix}$ is invertible over $Z$ if and only if we have $z_1,z_4 \in Z^\times$. Note also that $Z^\times$ is generated by the finite set $\calu = \{ -1,2 \}$, and consider the automorphisms of $M$ (written in matrix form)
\[
\alpha_u = \begin{pmatrix} u & 0 \\ 0 & 1 \end{pmatrix}, \quad
\psi = \begin{pmatrix} 1 & 1 \\ 0 & 1 \end{pmatrix}, \quad
\varphi = \begin{pmatrix} 1 & 0 \\ X & 1 \end{pmatrix} \quad\text{and}\quad
\gamma_u = \begin{pmatrix} 1 & 0 \\ 0 & u \end{pmatrix}
\]
where $u \in \calu$ [the notation here, apart from the $\alpha_u$, is as in the proof of Proposition~\ref{prop:no-Jordan-gaps}, just without redundant indices]. It is easy to see that these six elements are indeed automorphisms of the $Z[X]$-module $M$; we denote by $C$ the group generated by them.

We will show, in this order, that the automorphisms of $M$ of the forms
\begin{gather*}
\begin{pmatrix} 1 & * \\ 0 & 1 \end{pmatrix}, \quad
\begin{pmatrix} 1 & 0 \\ *X & 1 \end{pmatrix}, \quad
\begin{pmatrix} 1 & 0 \\ 0 & 1+*X \end{pmatrix}, \quad
\begin{pmatrix} \dag & 0 \\ 0 & \dag+*X \end{pmatrix}
\quad\text{and}\quad
\begin{pmatrix} \dag & * \\ *X & \dag+*X \end{pmatrix}
\end{gather*}
are elements of $C$, where $*$ and $\dag$ stand for arbitrary elements of $Z$ and $Z^\times$, respectively. The last one of these will then imply that $C = \Aut_{Z[X]}(M)$, and in particular $\Aut_{Z[X]}(M)$ is finitely generated.

Computations give
\[
\psi_{2^sk} \coloneqq \begin{pmatrix} 1 & 2^sk \\ 0 & 1 \end{pmatrix} = \begin{pmatrix} 1 & 0 \\ 0 & 2^{-s} \end{pmatrix} \begin{pmatrix} 1 & k \\ 0 & 1 \end{pmatrix} \begin{pmatrix} 1 & 0 \\ 0 & 2^s \end{pmatrix} = \gamma_2^{-s} \psi^k \gamma_2^s \in C
\]
and
\[
\varphi_{2^sk} \coloneqq \begin{pmatrix} 1 & 0 \\ 2^skX & 1 \end{pmatrix} = \begin{pmatrix} 1 & 0 \\ 0 & 2^s \end{pmatrix} \begin{pmatrix} 1 & 0 \\ kX & 1 \end{pmatrix} \begin{pmatrix} 1 & 0 \\ 0 & 2^{-s} \end{pmatrix} = \gamma_2^s \varphi^k \gamma_2^{-s} \in C
\]
for any $s,k \in \ZZ$, implying that matrices of the forms $\begin{pmatrix} 1 & * \\ 0 & 1 \end{pmatrix}$ and $\begin{pmatrix} 1 & 0 \\ *X & 1 \end{pmatrix}$ are in $C$. Consequently, given any $z \in Z$ we have
\begin{equation*} 
\gamma'_z \coloneqq \begin{pmatrix} 1 & 0 \\ 0 & 1+zX \end{pmatrix} = \begin{pmatrix} 1 & -z \\ 0 & 1 \end{pmatrix} \begin{pmatrix} 1 & 0 \\ X & 1 \end{pmatrix} \begin{pmatrix} 1 & z \\ 0 & 1 \end{pmatrix} \begin{pmatrix} 1 & 0 \\ -X & 1 \end{pmatrix} = \psi_{-z} \varphi \psi_z \varphi^{-1} \in C, 
\end{equation*}
implying that matrices of the form $\begin{pmatrix} 1 & 0 \\ 0 & 1+*X \end{pmatrix}$ are in $C$. Now given arbitrary elements $u = (-1)^\varepsilon 2^\zeta$ and $u' = (-1)^{\varepsilon'} 2^{\zeta'}$ of $Z^\times$ and an arbitrary $z \in Z$, we have
\[
\chi_{u,u',z} \coloneqq \begin{pmatrix} u' & 0 \\ 0 & u+zX \end{pmatrix} = \begin{pmatrix} u' & 0 \\ 0 & 1 \end{pmatrix} \begin{pmatrix} 1 & 0 \\ 0 & u \end{pmatrix} \begin{pmatrix} 1 & 0 \\ 0 & 1+u^{-1}zX \end{pmatrix} = \alpha_{-1}^{\varepsilon'} \alpha_2^{\zeta'} \gamma_{-1}^\varepsilon \gamma_2^\zeta \gamma_{u^{-1}z}' \in C
\]
implying that $C$ contains all matrices of the form $\begin{pmatrix} \dag & 0 \\ 0 & \dag+*X \end{pmatrix}$. Finally, note that given any $u,u' \in Z^\times$ and $z,z',z'' \in Z$ we have
\[
\begin{pmatrix} u' & z'' \\ z'X & u+zX \end{pmatrix} = \begin{pmatrix} 1 & u^{-1}z'' \\ 0 & 1 \end{pmatrix} \begin{pmatrix} 1 & 0 \\ (u')^{-1}z'X & 1 \end{pmatrix} \begin{pmatrix} u' & 0 \\ 0 & u+zX \end{pmatrix} = \psi_{u^{-1}z''} \varphi_{(u')^{-1}z'} \chi_{u,u',z} \in C,
\]
implying that indeed $C = \Aut_{Z[X]}(M)$.\hfill$\blackdiamond$
\end{example}

\begin{table}[ht]
    \centering
    \begin{tabular}{|c|c|}
       \hline
        $G=\LM(A,L)$ & $A\in \MM_n(\ZZ)\cap \GL_n(\QQ)$ and $L=\ZZ^n$ \\
        $K=\bigcup_{s\in \ZZ}A^s(L)$  & this is a subgroup of $\LM(A,L)\cong K\rtimes \ZZ$ and of $\QQ^n$ \\
        $C=C_{\GL(K)}(A)$ & where $\GL(K)\leqslant \GL_n(\QQ)$ is the elements which preserve $K$.\\
        $f=\prod_{i=1}^lp_i^{\alpha_i}(X)$ & characteristic polynomial of $A$ with $p_i$ irreducible over $\QQ$\\
        $V_i$ & an $A$-invariant subspace of $\QQ^n$ \\
        $p(X)$ & an irreducible polynomial in $\QQ[X]$ \\
        $F$ & splitting field for $p$ over $\QQ$\\
        $\lambda$ & a root of $p(X)$ in $F$ \\
        $Z=\ZZ[\lambda^{-1}]$ & for $\lambda$ an algebraic integer \\
        $\widehat Z=Z[X]/(X^d)$ & for some $d\geq 1$\\
        \hline
    \end{tabular}
    \caption{Table of notation for \Cref{sec:metabelian}.}
    \label{tab:notation_metabelian}
\end{table}

\subsection{The structure of the automorphisms}\label{metab:step1}

Let $G$ be a commensurating HNN-extension of $H \cong \ZZ^n$ with associated subgroups $H_0 = H$ and $H_{-1} < H$ and stable letter $t$, and let $A \in \GL_n(\QQ)$ be a matrix as above, so that $A$ has integer entries and $\left|\det(A)\right| > 1$. Then we can identify $G$ with $K \rtimes_A \langle t \rangle$, where $K = \bigcup_{s \in \ZZ} A^s(\ZZ^n) < \QQ^n$. Explicitly, we can define the multiplication in $G$ by $(v,t^q)(w,t^r) = (v+A^qw,t^{q+r})$ for $v,w \in K$ and $q,r \in \ZZ$.

\begin{defn}
Given $k \in K$, define the \emph{$k$-twist automorphism} $\tau_k \in \Aut(G)$ by setting $\tau_k(v,1) = (v,1)$ for $v \in K$ and $\tau_k(0,t) = (k,t)$. Let $C = C_{\GL(K)}(A)$, and given $B \in C$ define the \emph{$B$-vertex automorphism} $\chi_B \in \Aut(G)$ by setting $\chi_B(v,1) = (Bv,1)$ for $v \in K$ and $\chi_B(0,t) = (0,t)$.
\end{defn}

We'll be using the following auxiliary Lemma a few times.

\begin{lemma} \label{lem:sounds-trivial-but-is-it?}
Let $A \in \GL_n(\QQ)$, let $L' \leq L < \QQ^n$ be lattices such that $AL \subseteq L$ and $AL' \subseteq L'$, and let $K = \bigcup_{s \in \ZZ} A^sL$ and $K' = \bigcup_{s \in \ZZ} A^sL'$. Then $K/K'$ is finite.
\end{lemma}

\begin{proof}
Since $L$ and $L'$ are lattices in $\QQ^n$, there exists $p \in \ZZ_{\geq 1}$ such that $pL \subseteq L'$. It follows that $pK \subseteq K'$, and so it is enough to show that $K/pK$ is finite. In particular, if $p = p_1 \cdots p_r$ for some prime numbers $p_i$, it is then enough to show that $K'/p_iK'$ is finite for each $i$, where $K' = p_1 \cdots p_{i-1}K < \QQ^n$; therefore, after replacing $L$ and $K$ with $p_1 \cdots p_{i-1}L$ and $K'$, if necessary, we may assume that $p$ is prime.

Now $K / pK$ is a $p$-periodic abelian group and so a vector space over the finite field $\ZZ/p\ZZ$, so it is enough to show that it has dimension $\leq n$. Suppose, on the contrary, that there exist $v_0,\ldots,v_n \in K$ such that $v_0+pK,\ldots,v_n+pK$ are linearly independent in $K/pK$. But since $K < \QQ^n$, there exist $\lambda_0,\ldots,\lambda_n \in \QQ$, not all zero, such that $\sum_{i=0}^n \lambda_i v_i = 0$. After multiplying all the $\lambda_i$ by a non-zero element of $\ZZ[1/p]$, we may assume that the $\lambda_i$ are integers and not all of them are divisible by $p$. But then 
$$\sum_{i=0}^n (\lambda_i+p\ZZ)(v_i+pK) = 0$$
is a linear dependence relation in the $(\ZZ/p\ZZ)$-vector space $K/pK$, contradicting the fact that $v_0+pK,\ldots,v_n+pK$ are linearly independent. Thus $K/pK$ is a finite-dimensional vector space over $\ZZ/p\ZZ$ and so finite, as required.
\end{proof}

\begin{prop} \label{prop:metab-Out}
We have $\Aut(G) \cong C \ltimes K$, and $\Inn(G)$ corresponds to $\langle A \rangle \ltimes (I-A)K$ under this isomorphism. In particular, we have
\[
\Out(G) \cong (C / \langle A \rangle) \ltimes D,
\]
where $D$ is a finitely generated abelian group equal to $K/(I-A)K$, and so $\Out(G)$ is finitely generated if and only if $C$ is.
\end{prop}

\begin{proof}
Given $k \in K$ and $B \in C$, it is easy to check that $\chi_B \circ \tau_k \circ \chi_B^{-1} = \tau_{Bk}$ and that $\chi_B \neq \tau_k$ unless $B = I$ and $k = 0$. Thus $C \ltimes K$ can be identified with a subgroup of $\Aut(G)$.

Let $\gamma \in \Aut(G)$ be arbitrary. Since $H_{-1} \neq H$, the group $G$ is not polycyclic, and so it follows from \cite[Proposition~2.7]{Valiunas2022} that $\gamma(K) = K$. Moreover, given any $g \in G$ and any lattice $L' < K$, it is easy to check that we have \[\frac{[L' : L' \cap gL'g^{-1}]}{[gL'g^{-1} : L' \cap gL'g^{-1}]} = \left|\det(A)\right|^{\psi(g)},\] where $\psi\colon G \to \ZZ$ is the homomorphism sending $K$ to $0$ and $(0,t)$ to $1$. In particular, it follows that $\psi(\gamma(0,t)) = \psi(0,t) = 1$, and so $\gamma(t) = (k,t)$ for some $k \in K$. Let $\gamma' = \tau_k^{-1} \circ \gamma$, so that $\gamma'(0,t) = (0,t)$ and $\gamma'(v,1) = (Bv,1)$ for all $v \in K$, where $B \in \GL(K)$. Moreover, we have
\begin{align*}
(BAv,1) &= \gamma'(Av,1) = \gamma'((0,t)(v,1)(0,t)^{-1}) = (0,t)\gamma'(v,1)(0,t)^{-1} \\ &= (0,t)(Bv,1)(0,t)^{-1} = (ABv,1)
\end{align*}
for all $v \in K$, implying that $BA = AB$ and so $\gamma' = \chi_B$ is a $B$-vertex automorphism for some $B \in C$. It follows that $\gamma = \tau_k \circ \chi_B \in C \ltimes K$, and so $\Aut(G) \cong C \ltimes K$, as required.

Now given any $g = (k,t^q) \in G$, it is easy to check that left conjugation by $g$ is precisely the automorphism $\tau_{(I-A)k} \circ \chi_{A^q}$, and so $\Inn(G)$ is identified with $\langle A \rangle \ltimes (I-A)K$ under the isomorphism $\Aut(G) \cong C \ltimes K$, as required.

Finally, it remains to show that the abelian group $D \coloneqq K/(I-A)K$ is finitely generated; the statement about finite generation of $\Out(G)$ will then follow from the existence of surjections $\Out(G) \to C/\langle A \rangle$ and $C \to C/\langle A \rangle$ with finitely generated kernels. Let $W < \QQ^n$ be the generalised $1$-eigenspace for $A$; it then follows that $\QQ^n = W \oplus W'$ for some $A$-invariant subspace $W'$ such that $A|_{W'}$ has no eigenvalues equal to $1$ (see \cite[Theorem~21.5.2]{BhattacharyaJainNagpaul}). In particular, $Z' \coloneqq L \oplus L'$, where $L = W \cap \ZZ^n$ and $L' = W' \cap \ZZ^n$, is a finite-index subgroup of $\ZZ^n$, and we have $AL \subseteq L$ and $AL' \subseteq L'$. It follows from Lemma~\ref{lem:sounds-trivial-but-is-it?} that $K' \coloneqq \bigcup_{s \in \ZZ} A^sZ'$ is a finite-index subgroup of $K$; applying Lemma~\ref{lem:sounds-trivial-but-is-it?} within the subspace $(I-A)(\QQ^n)$ of $\QQ^n$, it also follows that $(I-A)K'$ is a finite-index subgroup of $(I-A)K$. It is thus enough to show that $K'/(I-A)K'$ is finitely generated.

But indeed, since $W$ is the generalised $1$-eigenspace for $A$ we have $AL = L$ and consequently $K' = L \oplus K''$, where $K'' = \bigcup_{s \in \ZZ} A^sL'$. Moreover, since $A|_{W'}$ has no eigenvalues equal to $1$, it follows that $(I-A)|_{W'}$ is invertible and so, by Lemma~\ref{lem:sounds-trivial-but-is-it?} applied to $L'$ and $(I-A)L'$, the subgroup $(I-A)K''$ has finite index in $K''$. Thus $K'/(I-A)K'$ is the direct sum of the finitely generated group $L/(I-A)L$ and the finite group $K''/(I-A)K''$, so is finitely generated, as required.
\end{proof}

\subsection{Recasting the problem in terms of modules}\label{metab:step2}

Let $A \in \GL_n(\QQ)$, and let $f(X) \in \QQ[X]$ be the characteristic polynomial of $A$. Suppose $f = p_1^{\alpha_1} \cdots p_l^{\alpha_l}$, where $p_1(X),\ldots,p_l(X) \in \QQ[X]$ are distinct monic polynomials that are irreducible over $\QQ$ and where $\alpha_1,\ldots,\alpha_l \in \ZZ_{\geq 1}$. Then we have a decomposition
\[
\QQ^n = V_1 \oplus \cdots \oplus V_l
\]
into $A$-invariant subspaces, where $V_i = \ker(p_i^{\alpha_i}(A))$ (see \cite[Theorem~21.5.2]{BhattacharyaJainNagpaul}). This motivates working with each of the subspaces $V_i$ separately.

\begin{lemma} \label{lem:rational-Jordan}
Let $A \in \GL_n(\QQ)$, and let $f(X) \in \QQ[X]$ be the characteristic polynomial of~$A$. Suppose that $f = p^\alpha$, where $p(X) = X^m + a_{m-1}X^{m-1} + \cdots + a_0 \in \QQ[X]$ is an irreducible polynomial (over $\QQ$) and $\alpha \geq 1$, and let $F$ be the splitting field of $p$ over $\QQ$. Let $\lambda \in F$ be a root of $p(X)$, and let $(v_{i,j} \mid 0 \leq i < \beta, 0 \leq j < d_i)$ be a (lower triangular) Jordan basis
for the generalised $\lambda$-eigenspace $M \leq F^n$ for $A$, so that
\[
Av_{i,j} = \begin{cases}
    \lambda v_{i,j} & \text{if } j = d_i-1, \\
    \lambda v_{i,j} + v_{i,j+1} & \text{otherwise}.
\end{cases}
\]
Then $\calb = (w_{i,j,k} \mid 0 \leq i < \beta, 0 \leq j < d_i, 0 \leq k < m)$ is a basis for $\QQ^n$, where $w_{i,j,k} = \sum_{\sigma \in \Gal(F/\QQ)} \sigma(\lambda^k v_{i,j})$, and we have
\[
A w_{i,j,k} = \begin{cases}
w_{i,j+1,k} - \sum_{\ell=0}^{m-1} a_\ell w_{i,j,\ell} & \text{if } k = m-1, \\
w_{i,j+1,k} + w_{i,j,k+1} & \text{otherwise},
\end{cases}
\]
where by convention $w_{i,d_i,k} \coloneqq 0$.
\end{lemma}

\begin{proof}
The vectors $w_{i,j,k}$ are fixed by the $\Gal(F/\QQ)$-action, and so are indeed elements of $\QQ^n$. Since $[F:\QQ] = \deg(p) = m$, there are exactly $n$ of the elements $w_{i,j,k}$, and so it is enough to show that they are linearly independent.

We first claim that $\calb' = (\sigma(v_{i,j}) \mid 0 \leq i < \beta, 0 \leq j < d_i, \sigma \in \Gal(F/\QQ))$ is a basis of $F^n$. Indeed, for any $\sigma \in \Gal(F/\QQ)$, the set $\calb'_\sigma = (\sigma(v_{i,j}) \mid 0 \leq i < \beta, 0 \leq j < d_i)$ is linearly independent, since otherwise the $\sigma^{-1}$-image of a linear dependence relation for $\calb'_\sigma$ would be a linear dependence relation for $\calb'_{\id}$, contradicting the fact that $\calb'_{\id}$ is a basis for the generalised $\lambda$-eigenspace in $F^n$. Since all the generalised eigenspaces in $F^n$ have dimension $\alpha =  \sum_{i=0}^{\beta-1} d_i$, it follows that each $\calb'_\sigma$ is a basis for the generalised $\sigma(\lambda)$-eigenspace. Finally, since $F^n$ splits as a direct sum of its generalised eigenspaces, any linear dependence relation for $\calb'$ would imply a linear dependence relation for some $\calb'_\sigma$. Thus $\calb'$ is linearly independent, and so a basis of $F^n$, as claimed.

Now note that the matrix $P$ transforming $\calb'$ to $\calb$ is block diagonal, where each diagonal block is the Vandermonde matrix associated to $(\sigma(\lambda) \mid \sigma \in \Gal(F/\QQ))$. In particular, $P$ is non-singular, so since $\calb'$ is a basis of $F^n$ so is $\calb$, as required. The explicit expression for $A w_{i,j,k}$ can be verified directly, using the fact that $p(\lambda) = 0$ in the case $k = m-1$.
\end{proof}

\begin{lemma} \label{lem:inverse-condition}
In the notation of Lemma~\ref{lem:rational-Jordan}, the map $X\colon M \to M$ defined by $X(v) = Av-\lambda v$ makes $M$ into an $F[X]$-module of the form $M = \bigoplus_{i=0}^{\beta-1} M_i$, where $M_i = F[X]/(X^{d_i})$; we write $m_i$ for a generator of $M_i$ as an $F[X]$-module.

An $F$-linear map $\xi \colon M \to M$ commutes with the action of $A$ if and only if it is an $F[X]$-module endomorphism. Given such a map $\xi$, we define the matrix $c(\xi) = (c_{i,j}(\xi))_{i,j = 0}^{\beta-1} \in GL_\beta(F)$, where $c_{i,j}(\xi) \in F$ is the unique element such that 
$$\xi(m_j) \in \sum_{i=0}^{\beta-1} (c_{i,j}(\xi)m_i + X \cdot M_i).$$ 
Then $\xi$ is invertible if and only if $c(\xi)$ is.
\end{lemma}

\begin{proof}
The $F[X]$-module structure for $M$ can be easily verified by setting $m_i = v_{i,0}$ for $0 \leq i < \beta$, so that $X^j \cdot m_i = v_{i,j}$ for $0 \leq j < d_i$. Moreover, since multiplication by $\lambda$ commutes with any $F$-linear map $\xi\colon M \to M$, it follows that $\xi$ commutes with the action of $A$ if and only if it commutes with the action of $X$, which happens if and only if $\xi$ is an $F[X]$-module endomorphism.

Now let $\xi\colon M \to M$ be an $F[X]$-module endomorphism. Then $\xi(X \cdot M) = X \cdot \xi(M) \subseteq X \cdot M$; in particular, $\xi$ induces an $F[X]$-module endomorphism $\overline\xi$ of $M / (X \cdot M) \cong F^\beta$. Now if $\xi$ is invertible then so is $\overline\xi$; but it follows from the construction that $\overline\xi$ is represented by the matrix $c(\xi)$, hence if $\xi$ is invertible then so is $c(\xi)$.

Conversely, suppose that $c(\xi)$ is invertible. Note that in order to show that $\xi$ is invertible, it is enough to show that it is surjective. Indeed, if $\xi$ is surjective then it is invertible as an $F$-linear map, and the inverse $\xi^{-1}\colon M \to M$ commutes with the action of $A$, implying that $\xi^{-1}$ is an $F[X]$-module endomorphism.

We claim that for each $0 \leq i < \beta$ and each $j \geq 0$ we have $m_i \in \xi(M) + X^j \cdot M$; taking $j = \max \{ d_i \mid 0 \leq i < \beta \}$ will then imply that $m_i \in \xi(M)$ and so that $\xi$ is surjective, as required. We prove the claim by induction on $j$; the base case, $j = 0$, is trivial. Let $j \geq 1$, and suppose that $m_i \in \xi(M) + X^{j-1} \cdot M$; therefore there exist $v \in M$ and $u_0,\ldots,u_{\beta-1} \in F$ such that 
$$m_i \in \xi(v) + \sum_{k=0}^{\beta-1} u_kX^{j-1} \cdot m_k + X^j \cdot M.$$ 
Since $c(\xi) \in \GL_\beta(F)$ is invertible, and therefore surjective, there exists $w = (w_k)_{k=0}^{\beta-1}$ such that $c(\xi) \cdot w = (u_k)_{k=0}^{\beta-1}$. It then follows from the definition of $c(\xi)$ that we have 
$$\xi\left(\sum_{k=0}^{\beta-1} w_km_k\right) \in \sum_{k=0}^{\beta-1} u_km_k + X \cdot M,$$ 
implying that $$m_i \in \xi\left(v+\sum_{k=0}^{\beta-1} w_kX^{j-1} \cdot m_k\right) + X^j \cdot M \subseteq \xi(M) + X^j \cdot M,$$ as claimed.
\end{proof}

Given an $F$-linear subspace $M \leq F^n$, an $F$-linear map $\xi\colon M \to M$ and an element $\sigma \in \Gal(F/\QQ)$, we can define an $F$-linear map $\sigma \cdot \xi \colon \sigma(M) \to \sigma(M)$ by setting $(\sigma \cdot \xi)(v) = \sigma(\xi(\sigma^{-1}(v)))$ for $v \in \sigma(M)$. If $F^n = \bigoplus_{\sigma \in \Gal(F/\QQ)} \sigma(M)$, then we also define an $F$-linear map 
$$\widetilde\xi \colon F^n \to F^n \quad \text{by}  \quad\widetilde\xi\left(\sum_{\sigma \in \Gal(F/\QQ)} v_\sigma\right) = \sum_{\sigma \in \Gal(F/\QQ)} (\sigma \cdot \xi)(v_\sigma),$$ 
where $v_\sigma \in \sigma(M)$.

\begin{lemma} \label{lem:centr-lambda}
In the notation of Lemmas \ref{lem:rational-Jordan} and \ref{lem:inverse-condition}, we have
\[
C_{\GL_n(\QQ)}(A) = \{ \widetilde\xi|_{\QQ^n} \mid \xi \in \End_{F[X]}(M), \det c(\xi) \neq 0 \} \cong \Aut_{F[X]}(M).
\]
\end{lemma}

\begin{proof}
By Lemma~\ref{lem:inverse-condition}, we have 
$$C_{\GL(M)}(A) = \{ \xi \in \End_{F[X]}(M) \mid \det c(\xi) \neq 0 \},$$ 
and a similar argument shows that 
$$C_{\GL(\sigma(M))}(A) = \{ \xi \in \End_{F[X]}(\sigma(M)) \mid \det c(\xi) \neq 0 \} = \{ {\sigma \cdot \xi} \mid \xi \in C_{\GL(M)}(A) \}$$ 
for any $\sigma \in \Gal(F/\QQ)$. 
Moreover, since maps commuting with $A$ preserve its generalised eigenspaces, we have 
$$C_{\GL_n(F)}(A) = \bigoplus_{\sigma \in \Gal(F/\QQ)} C_{\GL(\sigma(M))}(A).$$ 
In particular, since we have $\QQ^n = \{ \sum_{\sigma \in \Gal(F/\QQ)} \sigma(v) \mid v \in M \}$ by Lemma~\ref{lem:rational-Jordan}, it follows that given any $\xi \in \End_{F[X]}(M)$ we have $\widetilde\xi(\QQ^n) \subseteq \QQ^n$.  So if $\det c(\xi) \neq 0$ then we have $\widetilde\xi|_{\QQ^n} \in C_{\GL_n(\QQ)}(A)$.

Conversely, let $\zeta \in C_{\GL_n(\QQ)}(A)$, and for $\sigma \in \Gal(F/\QQ)$ define a  $\QQ$-linear map  $$ \zeta_\sigma \colon \sigma(M) \to \sigma(M) \quad \text{by} \quad  \zeta\left(\sum_{\sigma \in \Gal(F/\QQ)} \sigma(v)\right) = \sum_{\sigma \in \Gal(F/\QQ)} \zeta_\sigma(\sigma(v)) \text{ for all } v \in M.$$ Since $\zeta$ extends to an $F$-linear automorphism of $F^n$, it follows that the maps $\zeta_\sigma$ are $F$-linear as well. The facts that $\zeta(\QQ^n) = \QQ^n$ and that $F^n = \bigoplus_{\sigma \in \Gal(F/\QQ)} \sigma(M)$ then imply that for each $v \in M$ there exists $w_v \in M$ such that $\zeta_\sigma(\sigma(v)) = \sigma(w_v)$ for all $\sigma \in \Gal(F/\QQ)$, and therefore
\[
\zeta_\sigma(\sigma(v)) = \sigma(w_v) = \sigma(\zeta_{\id}(v)) = (\sigma \circ \zeta_{\id} \circ \sigma^{-1})(\sigma(v)) = (\sigma \cdot \zeta_{\id})(\sigma(v)),
\]
implying that $\zeta = \widetilde{\zeta_{\id}}|_{\QQ^n}$. Moreover, since $\zeta$ commutes with the action of $A$ so does $\zeta_{\id}$, and so we have $\zeta \in \{ \widetilde\xi|_{\QQ^n} \mid \xi \in C_{\GL(M)}(A) \}$, as required.
\end{proof}

\subsection{Varying the lattice in a commensurability class}\label{metab:step3}

The following result shows that for results preserved by commensurability (such as finite generation), it doesn't matter which lattice in $\QQ^n$ we start with. Given a subgroup $K \leq \QQ^n$ of rank $n$, we slightly abuse the notation and write $\GL(K)$ for the subgroup of $\GL_n(\QQ)$ consisting of elements that preserve $K$.

\begin{lemma} \label{lem:dont-care-about-L}
Let $A \in \GL_n(\QQ)$, let $L,L' < \QQ^n$ be lattices such that $AL \subseteq L$ and $AL' \subseteq L'$, and let $K = \bigcup_{s \in \ZZ} A^sL$ and $K' = \bigcup_{s \in \ZZ} A^sL'$. Then $\GL(K)$ and $\GL(K')$ are commensurable in $\GL_n(\QQ)$.
\end{lemma}

\begin{proof}
Suppose, without loss of generality, that $L' \subseteq L$. Since $L$ and $L'$ are lattices, we have $pL \subseteq L'$ for some $p \in \ZZ_{\geq 1}$; it follows that $pK \subseteq K' \subseteq K$.

Now since the subgroup $\GL(K) < \GL_n(\QQ)$ preserves $K$, it also preserves its characteristic subgroup $pK$, and so $\GL(K)$ acts on the set of subgroups $K'' \leq K$ satisfying $pK \subseteq K''$. By Lemma~\ref{lem:sounds-trivial-but-is-it?}, the quotient $K/pK$ is finite, and so there are finitely many such subgroups $K''$. It follows that the stabiliser $\GL(K) \cap \GL(K')$ of $K'$ under this action has finite index in $\GL(K)$. Similarly, since $pK' \subseteq pK \subseteq K'$, the intersection $\GL(pK) \cap \GL(K')$ has finite index in $\GL(K')$. But $\GL(pK) = \GL(K)$, and so $\GL(K)$ and $\GL(K')$ are commensurable in $\GL_n(\QQ)$, as required.
\end{proof}

\begin{example}
Consider $A = PDP^{-1} \in \GL_2(\QQ)$, where
\[
A = \begin{pmatrix} 2&1\\1&2 \end{pmatrix}, \quad D = \begin{pmatrix} 1&0\\0&3 \end{pmatrix} \quad \text{and} \quad P = \begin{pmatrix} 1&1\\-1&1 \end{pmatrix}.
\]
We then have
\[
K \coloneqq \bigcup_{s \in \ZZ} A^s(\ZZ^2) = \left\{ \frac{1}{2}\begin{pmatrix} x+3^sy \\ -x+3^sy \end{pmatrix} \,\middle|\, x,y,s \in \ZZ, x+y \in 2\ZZ \right\}.
\]
On the other hand,
\[
V_1 \cap K = \begin{pmatrix} 1 \\ -1 \end{pmatrix} \ZZ \quad \text{and} \quad V_3 \cap K = \begin{pmatrix} 1 \\ 1 \end{pmatrix} \ZZ[1/3],
\]
where $V_\lambda < \QQ^2$ is the $\lambda$-eigenspace of $A$, implying that $(V_1 \cap K) \oplus (V_3 \cap K) \neq K$. Thus Lemma~\ref{lem:dont-care-about-L} is indeed necessary in our arguments.
\end{example}

\begin{lemma} \label{lem:lattice-A-image}
In the notation of Lemma~\ref{lem:rational-Jordan}, suppose $A$ has integer entries, and let $R < M$ be the free abelian group with basis $\calb' = (v_{i,j} \mid 0 \leq i < \beta, 0 \leq j < d_i)$. Then we have $$\bigcup_{s \in \ZZ} A^sR = R[\lambda^{-1}] \coloneqq \left\{ \sum_{v \in \calb'} c_v v \mid c_v \in \ZZ[\lambda^{-1}] \right\}.$$ In particular, if $L < \QQ^n$ is the lattice with basis $\calb$, then $$K \coloneqq \bigcup_{s \in \ZZ} A^sL = \left\{ \sum_{\sigma \in \Gal(F/\QQ)} \sigma(v) \mid v \in R[\lambda^{-1}] \right\}.$$
\end{lemma}

\begin{remark}
It might seem that this last conclusion is the same as inverting $a_0$ (in notation of \Cref{lem:rational-Jordan}) but it is not.  For example\ if $A = \begin{pmatrix} 0&3\\1&1 \end{pmatrix}$, then the second coordinate of $A^k\!\begin{pmatrix} 0\\1 \end{pmatrix}$ is $1 \pmod 3$ for all $k > 0$, implying that $\begin{pmatrix} 0\\1/3 \end{pmatrix} \notin \bigcup_{s \in \ZZ} A^s(\ZZ^2)$.
\end{remark}

\begin{proof}[Proof of \Cref{lem:lattice-A-image}]
Note first that we have $\ZZ[\lambda^{-1}] = \ZZ[\lambda,\lambda^{-1}]$. Indeed, since $A$ has integer entries the coefficients $a_0,\ldots,a_{m-1}$ of $p(X)$ are integers, implying that
\[
\lambda = \lambda-\lambda^{1-m}p(\lambda) = -\sum_{\ell=1}^m a_{m-\ell} \lambda^{1-\ell} \in \ZZ[\lambda^{-1}].
\]

Now note that $R[\lambda^{-1}]$ is clearly preserved by $A^{\pm 1}$, so since $R \leq R[\lambda^{-1}]$ it follows that $R[\lambda^{-1}]$ contains $\widehat{R} \coloneqq \bigcup_{s \in \ZZ} A^sR$. Conversely, in order to show that $R[\lambda^{-1}] \subseteq \widehat{R}$, we claim that $R_{i,j} \subseteq \widehat{R}$ for $0 \leq i < \beta$ and $0 \leq j \leq d_i$, where 
$$R_{i,j} = \bigoplus_{j'=j}^{d_i-1} \ZZ[\lambda^{-1}] \cdot v_{i,j'}.$$ 
We will prove this by induction on $d_i-j$; the base case, $j = d_i$, is trivial. Now let $j < d_i$, and note that $R_{i,j+1}$ is preserved under the action of $A^{\pm 1}$ and multiplication by $\lambda^{\pm 1}$. The fact that $Av_{i,j} \in \lambda v_{i,j} + R_{i,j+1}$ then implies that $$A^kv_{i,j} \in \lambda^k v_{i,j} + R_{i,j+1} \quad \text{for all} k \geq 0.$$
In particular, if $R_{i,j+1} \subseteq \widehat{R}$ then we have 
$$\lambda^{-k} v_{i,j} \in A^{-k} v_{i,j} + R_{i,j+1} \subseteq \widehat{R} \quad \text{for all} k \geq 0,$$ 
and therefore 
$R_{i,j} \subseteq \widehat{R}$, as claimed. It follows that $R_{i,0} \subseteq \widehat{R}$ for $0 \leq i < \beta$ and therefore $R[\lambda^{-1}] = \widehat{R}$, as required.

Finally, the description of $K$ follows from the definitions of $\calb$ and $\calb'$, the fact that $\ZZ[\lambda^{-1}] = \ZZ[\lambda,\lambda^{-1}]$, and the fact $A$ commutes with the action of $\Gal(F/\QQ)$.
\end{proof}

\subsection{Describing the centraliser}\label{metab:step4}

In order to describe $C = C_{\GL(K)}(A)$ in terms of matrices, we will use the following well-known result.

\begin{lemma} \label{lem:units}
Let $Z$ be a commutative ring with $1$ and let $d \geq 1$. Then an element $u = u_0 + u_1X + \cdots + u_{d-1}X^{d-1}$, where $u_0,\ldots,u_{d-1} \in Z$, is a unit in $Z[X]/(X^d)$ if and only if $u_0 \in Z^\times$.
\end{lemma}

\begin{proof}
Since, given $u,v \in Z[X]/(X^d)$, the constant coefficient of $uv$ is the product of the constant coefficients of $u$ and of $v$, it is clear that if $u$ is a unit then $u_0 \in Z^\times$.

Conversely, suppose that $u_0 \in Z^\times$. We will show that $u = u_0 + \cdots + u_{d-1}X^{d-1} \in Z[X]/(X^d)$ is a unit by induction on $d$. In the base case, $d = 1$, we have $Z[X]/(X^d) = Z$ and so the result is clear. Suppose now that $d > 1$ and assume, by the inductive hypothesis, that the image of $u = u_0 + \cdots + u_{d-1}X^{d-1} \in Z[X]$ in $Z[X]/(X^{d-1})$ is a unit. Then there exists $v \in Z[X]$ such that $uv \in 1+(X^{d-1})$, and therefore $uv \in 1 + zX^{d-1} + (X^d)$ for some $z \in Z$. But it then follows that $u(v-zu_0^{-1}X^{d-1}) \in 1 + (X^d)$ and therefore the image of $u$ in $Z[X]/(X^d)$ is a unit, as required.
\end{proof}

\begin{lemma} \label{lem:matrices}
Consider an $F[X]$-module $M = \bigoplus_{d=1}^\delta \left( F[X]/(X^d) \right)^{m_d}$. 
Then any $F[X]$-module endomorphism $\xi\colon M \to M$ can be represented by a block ``matrix'' (note that the entries may come from different rings)
\[
B(\xi) = \begin{pmatrix}
    B_{1,1}(\xi) & \cdots & B_{1,\delta}(\xi) \\
    \vdots & \ddots & \vdots \\
    B_{\delta,1}(\xi) & \cdots & B_{\delta,\delta}(\xi)
\end{pmatrix},
\]
where $B_{d,d'}(\xi)$ is an $m_d \times m_{d'}$ matrix whose entries consist of $F[X]$-module homomorphisms $F[X]/(X^{d'}) \to F[X]/(X^d)$, i.e.\ elements of $F[X]/(X^d)$ if $d \leq d'$, and of $X^{d-d'}F[X]/(X^d)$ if $d > d'$; conversely, any such ``matrix'' represents an $F[X]$-module endomorphism of $M$. Under this representation, the matrix $c(\xi)$ is block upper triangular, with the $(d,d')$-th block consisting of the constant coefficients of the entries of $B_{d,d'}(\xi)$.
\end{lemma}

\begin{proof}
Follows immediately by construction.
\end{proof}

\begin{corollary} \label{cor:matrices-int}
In the notation of Lemma~\ref{lem:matrices}, the set of $Z[X]$-module endomorphisms $\xi$ of $R[\lambda^{-1}]$ (where $Z = \ZZ[\lambda^{-1}]$) can be identified with the matrices $B(\xi)$, so that the entries of $B_{d,d'}(\xi)$ are $Z[X]$-module homomorphisms $Z[X]/(X^{d'}) \to Z[X]/(X^d)$, i.e.\ elements of $Z[X]/(X^d)$ if $d \leq d'$, and of $X^{d-d'}Z[X]/(X^d)$ if $d > d'$. Such an endomorphism $\xi\colon R[\lambda^{-1}] \to R[\lambda^{-1}]$ is an automorphism if and only if the matrix consisting of constant coefficients of $B_{d,d}(\xi)$ is an element of $\GL_{m_d}(Z)$, i.e.\ has determinant in $Z^\times$, for all $1 \leq d \leq \delta$.
\end{corollary}

\begin{proof}
This follows from combining Lemmas \ref{lem:centr-lambda}, \ref{lem:lattice-A-image} and \ref{lem:matrices}.
\end{proof}

\begin{lemma}
In the notation of Lemma~\ref{lem:lattice-A-image}, let $C = C_{\GL(K)}(A)$. If $A$ has a $\lambda$-Jordan gap for some eigenvalue $\lambda$, then there exists a surjective group homomorphism $C \to (\ZZ[\lambda^{-1}],+)$. In particular, if $A$ is not essentially Jordan gap free, then $C$ is not finitely generated.
\end{lemma}

\begin{proof}
Let $d \geq 1$ be a $\lambda$-Jordan gap for $A$. Note that the submodules $X^{d \pm 1}M$ of the $F[X]$-module $M$ are preserved by $F[X]$-module automorphisms, implying that we have a homomorphism $\Phi\colon \Aut_{F[X]}(M) \to \Aut_{F[X]}(M')$, where $M' = X^{d-1}M/X^{d+1}M$. Moreover, since $d$ is a $\lambda$-Jordan gap, it follows that we have $M' \cong (F[X]/(X^2))^m$ for some $m \geq 1$, and so $\Aut_{F[X]}(M') \cong \GL_m(F[X]/(X^2))$. In particular, the determinant $\det\colon \Aut_{F[X]}(M') \to (F[X]/(X^2))^\times$ is a well-defined homomorphism. In terms of Lemma~\ref{lem:matrices}, the map $\Phi$ has the following description: given $\xi \in \Aut_{F[X]}(M)$, we have
\[
B(\Phi(\xi)) = \begin{pmatrix}
    \overline{B}_{d+1,d+1}(\xi) & \cdots & \overline{B}_{d+1,\delta}(\xi) \\
    \vdots & \ddots & \vdots \\
    \overline{B}_{\delta,d+1}(\xi) & \cdots & \overline{B}_{\delta,\delta}(\xi)
\end{pmatrix},
\]
where the entries of $\overline{B}_{d',d''}(\xi)$ are the images of the entries of $B_{d',d''}(\xi)$ under the ring homomorphism $F[X]/(X^{d'}) \to F[X]/(X^2)$.

Now by Corollary~\ref{cor:matrices-int}, given an automorphism $\xi$ of $R[\lambda^{-1}]$ the matrix $B(\Phi(\xi^{\pm 1}))$ has entries in $Z[X]/(X^2)$, and so we have $\Phi(\xi) \in \GL_m(Z[X]/(X^2))$. Consider the group homomorphisms
\[
\alpha\colon C \xrightarrow{\Phi} \GL_m(Z[X]/(X^2)) \xrightarrow{\det} \left( Z[X]/(X^2) \right)^\times \text{ and } \beta\colon Z^\times \times (Z,+) \to \left( Z[X]/(X^2) \right)^\times,
\]
where $\beta(z,w) = z(1+wX)$. Note that $\beta$ is well-defined and surjective by Lemma~\ref{lem:units}; in particular, $\beta$ is an isomorphism. Note also that $\alpha$ is surjective: indeed, given any $u \in \left( Z[X]/(X^2) \right)^\times$ we have $u = z+wX$ for some $z \in Z^\times$ and $w \in Z$, and then $u' = z+wX$ is also invertible in $Z[X]/(X^{d+1})$ by Lemma~\ref{lem:units}; given any generator $v$ of the $Z[X]$-submodule 
$$\left( Z[X]/(X^{d+1}) \right)^{m_{d+1}} \leq \bigoplus_{d'=1}^\delta \left( Z[X]/(X^{d'}) \right)^{m_{d'}} = R[\lambda^{-1}],$$ 
we have an automorphism $\xi$ of $R[\lambda^{-1}]$ sending $v$ to $u'v$ and all the other generators of $R[\lambda^{-1}]$ to themselves, and therefore $\alpha(\xi) = u$. It follows that $C$ surjects onto $\left( Z[X]/(X^2) \right)^\times \cong Z^\times \times (Z,+)$, and so onto $(Z,+)$, as required.

Finally, if $A$ is not essentially Jordan gap free, then there exists an eigenvalue $\lambda$ of $A$ such that $\lambda^{-1}$ is not an algebraic integer and $C$ surjects onto $(\ZZ[\lambda^{-1}],+)$. The fact that $\lambda^{-1}$ is not an algebraic integer is equivalent to the fact that the abelian group $(\ZZ[\lambda^{-1}],+)$ is not finitely generated, and so in this case $C$ is not finitely generated either.
\end{proof}

\begin{lemma} \label{lem:gaussian-elimination}
Let $\lambda$ be an algebraic integer, let $Z = \ZZ[\lambda^{-1}]$, let $d,m \geq 1$ be integers, and let $\widehat{Z} = Z[X]/(X^d)$. Then $\GL_m(\widehat{Z})$ is generated by the subgroup $\GL_m(Z)$ together with the maps
\[
\varphi_{i,j,k}\colon w_i \mapsto w_i+X^kw_j \quad\text{and}\quad \gamma_{i,u}\colon w_i \mapsto uw_i,
\]
where $(w_1,\ldots,w_m)$ is the standard generating set for $\widehat{Z}^m$, and where $1 \leq i,j \leq m$ with $i \neq j$, $0 \leq k < d$ and $u \in \widehat{Z}^\times$. 
\end{lemma}

\begin{proof}
Let $G$ be the group generated by $\GL_m(Z)$ together with the maps $\varphi_{i,j,k}$ and $\gamma_{i,u}$. For all $n \in \NN$, the map $\gamma_{i,\lambda^n} \circ \varphi_{i,j,k} \circ \gamma_{i,\lambda^{-n}}$ sends 
$$w_i \mapsto w_i + \lambda^{-n}X^kw_j\quad \text{and} \quad w_{i'} \mapsto w_{i'}\quad  \text{for } i' \neq i.$$ 
Given $z \in \widehat{Z}$, let $\varphi_{i,j}^{(z)}$ be the automorphism sending $w_i \mapsto w_i + zw_j$ and $w_{i'} \mapsto w_{i'}$ for $i' \neq i$; since $\varphi_{i,j}^{(z)} \circ \varphi_{i,j}^{(z')} = \varphi_{i,j}^{(z+z')}$ for all $z,z' \in \widehat{Z}$ and since $\{ \lambda^{-n}X^k \mid n \in \NN, 0 \leq k < d \}$ generates $(\widehat{Z},+)$, it follows that $\varphi_{i,j}^{(z)} \in G$ for all $1 \leq i,j \leq m$ with $i \neq j$ and all $z \in \widehat{Z}$.

Now let $B \in \GL_m(\widehat{Z})$; we claim that $B \in G$. Indeed, looking at $B$ as a matrix, we can write $B = B_0 + B'X$, where $B_0$ has entries in $Z$. It is clear that $\det B \in \det B_0 + X\widehat{Z}$; therefore, since $\det B \in \widehat{Z}^\times$, it follows from Lemma~\ref{lem:units} that we have $\det B_0 \in Z^\times$, and therefore $B_0 \in \GL_m(Z)$. We then have $B_0^{-1}B = I + B_0^{-1}B'X \in G$, and so it is enough to show that $I+DX \in G$ for any $m \times m$ matrix $D$ with entries in $\widehat{Z}$.

We now use the elements $\varphi_{i,j}^{(z)}$ and $\gamma_{i,u}$ to convert the columns of $I+DX$ to the standard basis vectors, one by one. Indeed, suppose the $i'$-th column of $I+DX$ is the $i'$-th basis vector for $1 \leq i' < i$. Now the $(i,i)$-th entry of $I+DX$ is of the form $u = 1+Xz$ for some $z \in \widehat{Z}$ and so is a unit in $\widehat{Z}$ by Lemma~\ref{lem:units}; after composing $I+DX$ with $\gamma_{i,u^{-1}}$, i.e.\ multiplying the $i$-th row of $I+DX$ by $u^{-1}$, we can then assume that $u = 1$. Now for each $1 \leq j \leq m$ with $j \neq i$, we can apply $\varphi_{i,j}^{(-D_{j,i}X)}$ to subtract $D_{j,i}X$ times the $i$-th row from the $j$-th row, and therefore setting the $(j,i)$-th entry of $I+DX$ to zero. Thus, continuing this procedure we obtain the identity matrix, and so $B \in G$, as claimed.
\end{proof}

\begin{prop} \label{prop:no-Jordan-gaps}
Suppose that a matrix $A \in \GL_n(\QQ)$ with integer coefficients is essentially Jordan gap free, and let $C = C_{\GL(K)}(A)$, where $K = \bigcup_{s \in \ZZ} A^s(\ZZ^n)$. Then $C$ is finitely generated.
\end{prop}

\begin{proof}
Suppose first that $A$ has characteristic polynomial $f = p_1^{\alpha_1} \cdots p_l^{\alpha_l}$, where the $p_i$ are distinct monic polynomials irreducible over $\QQ$ and $\alpha_i \geq 1$. Then $\QQ^n = \bigoplus_{i=1}^l V_i$, where $V_i = \ker(p_i^{\alpha_i}(A))$, and so $\ZZ^n$ has a finite-index subgroup $L = \bigoplus_{i=1}^l (\ZZ^n \cap V_i)$ such that $AL \subseteq L$. By Lemma~\ref{lem:dont-care-about-L}, we may, without loss of generality, replace $C$ with $C_{\GL(K')}(A)$, where $K' = \bigoplus_{i=1}^l K_i'$ and $K_i' = \bigcup_{s \in \ZZ^n} A^s(\ZZ^n \cap V_i)$. But since any matrix commuting with $A$ preserves the generalised eigenspaces of $A$, it follows that we then have $C = \prod_{i=1}^l C_{\GL(K_i')}(A)$ and so it is enough to consider the case when $l = 1$. We thus adopt the notation of Lemma~\ref{lem:rational-Jordan}. Since $A$ is essentially Jordan gap free, it follows that either $\lambda^{-1}$ is an algebraic integer, or $A$ has no $\lambda$-Jordan gaps.

Suppose first that $\lambda^{-1}$ is an algebraic integer. Then the additive group of $\ZZ[\lambda^{-1}]$ is finitely generated, implying by Lemma~\ref{lem:lattice-A-image} that $K$ is a finitely generated abelian group and therefore isomorphic to $\ZZ^n$. Thus $C$ is the centraliser of a $\GL_n(\QQ)$-conjugate of $A$ in the arithmetic group $\GL_n(\ZZ)$, and so is an arithmetic group and hence finitely generated \cite[Theorem~6.5]{BorelHarishChandra1962} (see also \cite[Algorithm B]{GrunewaldSegal1980} and \cite{EickHofmannOBrien2019}).

Suppose now that $\lambda^{-1}$ is not an algebraic integer, and therefore $A$ has no $\lambda$-Jordan gaps---that is, $A$ has $\lambda$-Jordan blocks of sizes $1,\ldots,\delta$ for some $\delta \geq 1$. We will prove, by induction on $\delta$, that $C = C_\delta$ is finitely generated. We adopt the matrix representation of elements of $C_\delta$ as in Corollary~\ref{cor:matrices-int}. In particular, $C_{\delta-1}$ can be realised as a subgroup of $C_\delta$, consisting of matrices whose $(d,\delta)$-th and $(\delta,d)$-th blocks are zero for $d < \delta$ and whose $(\delta,\delta)$-th block is the identity. We can also realise $\GL_{m_\delta}(Z)$ as a subgroup of $C_\delta$, consisting of matrices which are the identity except for a matrix in $\GL_{m_\delta}(Z)$ in the $(\delta,\delta)$-th block.

Let $\calw = (w_{d,i} \mid 1 \leq d \leq \delta, 1 \leq i \leq m_d)$ be a generating set for the $Z[X]$-module $R[\lambda^{-1}]$, so that $(w_{d,1},\ldots,w_{d,m_d})$ is the canonical generating set for the submodule $\left( Z[X]/(X^d) \right)^{m_d}$. By Dirichlet's $S$-unit theorem, the group $Z^\times$ is finitely generated; let $\calu$ be a finite generating set for $Z^\times$.

\begin{claim}\label{claim:generating_set} $C_\delta$ is generated by $\langle C_{\delta-1}, \GL_{m_\delta}(Z) \rangle$ together and the finite set of automorphisms
\begin{gather*}
\psi_{d,i,j,k}\colon w_{\delta,j} \mapsto w_{\delta,j} + X^k w_{d,i}, \qquad \varphi_{d,i,j,k'}\colon w_{d,i} \mapsto w_{d,i} + X^{k'} w_{\delta,j}, \\ \text{and}\qquad
\gamma_{j,u}\colon w_{\delta,j} \mapsto uw_{\delta,j},
\end{gather*}
where the automorphisms send all the unnamed elements of $\calw$ to themselves, and where $1 \leq d \leq \delta$, $1 \leq i \leq m_d$, $1 \leq j \leq m_\delta$, $0 \leq k < d$, $\delta-d \leq k' < \delta$, $u \in \calu$ and $(d,i) \neq (\delta,j)$.
\end{claim}

In terms of matrices, these automorphisms differ from the identity matrix as follows: $\psi_{d,i,j,k}$ has entry $X^k$ in the $(i,j)$-th position of the $(d,\delta)$-th block; $\varphi_{d,i,j,k'}$ has entry $X^{k'}$ in the $(j,i)$-th position of the $(\delta,d)$-th block; and $\gamma_{j,u}$ has entry $u$ in the $(j,j)$-th position of the $(\delta,\delta)$-th block. It is easy to check that these are indeed automorphisms of $R[\lambda^{-1}]$; moreover, note that we have $\varphi_{\delta,i,j,k} = \psi_{\delta,j,i,k}$.

\begin{claimproof}[Proof of \Cref{claim:generating_set}]
Note first that if $\delta = 1$ then we have $C_1 \cong \GL_{m_1}(Z)$, and so the claim follows. Thus, suppose that $\delta \geq 2$. We aim to show, in this order, that the following automorphisms are elements of the group $C_\delta'$ generated by $\langle C_{\delta-1}, \GL_{m_\delta}(Z) \rangle$ together with the elements $\psi_{d,i,j,k}$, $\varphi_{d,i,j,k'}$ and $\gamma_{j,u}$:
\begin{enumerate}[label=(\roman*)]
\item \label{it:autos-elem-offdiag} $\psi_{d,i,j}^{(z)}\colon w_{\delta,j} \mapsto w_{\delta,j} + z w_{d,i}$ and $\varphi_{d,i,j}^{(z')}\colon w_{d,i} \mapsto w_{d,i} + z' w_{\delta,j}$, with $d$, $i$ and $j$ as above, where $z \in Z[X]/(X^d)$ and $z' \in X^{\delta-d}Z[X]/(X^\delta)$;
\item \label{it:autos-elem-diag} $\gamma_{j,u'}\colon w_{\delta,j} \mapsto u' w_{\delta,j}$, where $1 \leq j \leq m_\delta$ and $u' = 1 + zX \in Z[X]/(X^\delta)$ for some $z \in Z[X]/(X^{\delta-1})$;
\item \label{it:autos-block-diag} all $Z[X]$-module automorphisms of $R[\lambda^{-1}]$ that preserve the submodules \[S \coloneqq \bigoplus_{d=1}^{\delta-1} \left( Z[X]/(X^d) \right)^{m_d}\quad\text{ and }\quad S' \coloneqq \left( Z[X]/(X^\delta) \right)^{m_\delta};\]
\item \label{it:autos-block-triang} all $Z[X]$-module automorphisms of $R[\lambda^{-1}]$ that preserve $S' = \left( Z[X]/(X^\delta) \right)^{m_\delta}$;
\item \label{it:autos-all} all $Z[X]$-module automorphisms of $R[\lambda^{-1}]$.
\end{enumerate}
The claim will then follow from \ref{it:autos-all}.

\medskip\pushQED{$\blackdiamond$}

\noindent\emph{Proof of \ref{it:autos-elem-offdiag}.} Note that the isomorphisms defined in \ref{it:autos-elem-offdiag} satisfy $$\psi_{d,i,j}^{(z_1)} \circ \psi_{d,i,j}^{(z_2)} = \psi_{d,i,j}^{(z_1+z_2)} \quad \text{and} \quad \varphi_{d,i,j}^{(z_1')} \circ \varphi_{d,i,j}^{(z_2')} = \varphi_{d,i,j}^{(z_1'+z_2')}.$$
Since the elements $\{ \lambda^{-n} X^k \mid n \in \NN, 0 \leq k < d \}$ (respectively $\{ \lambda^{-n} X^{k'} \mid n \in \NN, \delta-d \leq k' < \delta \}$) generate the additive subgroup of $Z[X]/(X^d)$ (respectively $X^{\delta-d}Z[X]/(X^\delta)$), in order to show \ref{it:autos-elem-offdiag} it is enough to show that we have $\psi_{d,i,j}^{(\lambda^{-n}X^k)} \in C_\delta'$ and $\varphi_{d,i,j}^{(\lambda^{-n}X^{k'})} \in C_\delta'$ for all $n \in \NN$, $0 \leq k < d$ and $\delta-d \leq k' < \delta$. But indeed, since $\lambda^{\pm n} \in Z^\times$ we have 
$$\psi_{d,i,j}^{(\lambda^{-n}X^k)} = \gamma_{j,\lambda^n} \circ \psi_{d,i,j,k} \circ \gamma_{j,\lambda^{-n}} \in C_\delta'$$
and 
$$\varphi_{d,i,j}^{(\lambda^{-n}X^{k'})} = \gamma_{j,\lambda^{-n}} \circ \varphi_{d,i,j,k'} \circ \gamma_{j,\lambda^n} \in C_\delta',$$
which proves \ref{it:autos-elem-offdiag}. \hfill\popQED

\medskip\pushQED{$\blackdiamond$}

\noindent\emph{Proof of \ref{it:autos-elem-diag}.} Let $z \in Z[X]/(X^{\delta-1})$ and consider the automorphism
\begin{align*}
\chi \coloneqq \varphi_{\delta-1,1,j,1} \circ \psi_{\delta-1,1,j}^{(z)} \colon w_{\delta-1,1} &\mapsto w_{\delta-1,1} + Xw_{\delta,j}, \\
w_{\delta,j} &\mapsto (1+zX)w_{\delta,j} + zw_{\delta-1,1},
\end{align*}
which is well-defined since the fact that $A$ has no $\lambda$-Jordan gaps implies that $m_{\delta-1} \geq 1$. By Lemma~\ref{lem:units}, the element $u = 1+zX$ is a unit in $Z[X]/(X^\delta)$ and thus has an inverse~$u^{-1}$. We then have an automorphism
\begin{align*}
\chi' \coloneqq \psi_{\delta-1,1,j}^{(-u^{-1}z)} \circ \chi \circ \varphi_{\delta-1,1,j}^{(-u^{-1}X)} \colon w_{\delta-1,1} &\mapsto u^{-1}w_{\delta-1,1}, \\
w_{\delta,j} &\mapsto uw_{\delta,j}.
\end{align*}
Since $C_{\delta-1}$ contains the automorphism $\gamma'\colon w_{\delta-1,1} \mapsto uw_{\delta-1,1}$, it follows from \ref{it:autos-elem-offdiag} that we have $\gamma_{j,u} = \gamma' \circ \chi' \in C_\delta'$, proving \ref{it:autos-elem-diag}. \hfill\popQED

\medskip\pushQED{$\blackdiamond$}

\noindent\emph{Proof of \ref{it:autos-block-diag}.} It follows from \ref{it:autos-elem-diag} that given any $u = u_0+u'X \in \left( Z[X]/(X^\delta) \right)^\times$, where $u_0 \in Z^\times$ and $u' \in Z[X]/(X^{\delta-1})$, the map 
$$\gamma_{j,u} \coloneqq \gamma_{j,u_0} \circ \gamma_{j,u_0^{-1}u} \quad \text{by} \quad w_{\delta,j} \mapsto uw_{\delta,j}$$ 
is an element of $C_\delta'$. Since $C_\delta'$ also contains the elements $\varphi_{\delta,i,j,k}$ for $1 \leq i,j \leq m_\delta$ with $i \neq j$ and $0 \leq k < \delta$ and the subgroup $\GL_{m_\delta}(Z)$, it follows from Lemma~\ref{lem:gaussian-elimination} that $C_\delta'$ contains the group $\GL_{m_\delta}(Z[X]/(X^\delta))$ of all automorphisms of $R[\lambda^{-1}]$ that preserve $S'$ and act as the identity on $S$. Since $C_\delta'$ also contains the group $C_{\delta-1}$ of automorphisms that act as identity on $S'$ and preserve $S$, this proves \ref{it:autos-block-diag}. \hfill\popQED

\medskip\pushQED{$\blackdiamond$}

\noindent\emph{Proof of \ref{it:autos-block-triang}.} Now given any automorphism $\xi$ as in \ref{it:autos-block-triang}, we can write $\xi(w) = \chi(w) + \varphi'(w)$ for $w \in S$ and $\xi(w') = \chi'(w')$ for $w' \in S'$, for some endomorphisms $\chi\colon S \to S$, $\chi'\colon S' \to S'$ and $\varphi'\colon S \to S'$. Since $\xi$ is an automorphism, it follows from Corollary~\ref{cor:matrices-int} that so are $\chi$ and $\chi'$, and in particular the endomorphism $\xi'$ of $R[\lambda^{-1}]$, defined by $\xi'(w) = \chi(w)$ for $w \in S$ and $\xi'(w') = \chi'(w')$ for $w' \in S'$, is an automorphism, and hence $\xi' \in C_\delta'$ by \ref{it:autos-block-diag}. But note that we have $\xi = \varphi \circ \xi'$, where $\varphi(w) = w + (\varphi' \circ \chi^{-1})(w)$ for $w \in S$ and $\varphi(w') = w'$ for $w' \in S'$. It is then clear from the matrix descriptions that $\varphi$ is a composite of automorphisms of $R[\lambda^{-1}]$ of the form $\varphi_{d,i,j}^{(z)}$, where $1 \leq d < \delta$, $1 \leq i \leq m_d$, $1 \leq j \leq m_\delta$ and $z \in X^{\delta-d}Z[X]/(X^\delta)$, and therefore $\varphi \in C_\delta'$ by \ref{it:autos-elem-offdiag}. Thus we have $\xi \in C_\delta'$, proving \ref{it:autos-block-triang}. \hfill\popQED

\medskip\pushQED{$\blackdiamond$}

\noindent\emph{Proof of \ref{it:autos-all}.} Consider an arbitrary automorphism $\xi$ of $R[\lambda^{-1}]$, so that we can write $\xi(w) = \chi(w) + \varphi'(w)$ for $w \in S$ and $\xi(w') = \chi'(w') + \psi'(w')$ for $w' \in S'$, for some endomorphisms $\chi\colon S \to S$, $\chi'\colon S' \to S'$, $\varphi'\colon S \to S'$ and $\psi'\colon S' \to S$. Now it follows from Corollary~\ref{cor:matrices-int} that $\chi'$ is an automorphism of $S'$, and we can verify that $\xi = \psi \circ \xi'$, where
\begin{align*}
\psi\colon w &\mapsto w, & \xi'\colon w &\mapsto \chi(w) - (\psi' \circ (\chi')^{-1} \circ \varphi')(w) + \varphi'(w), \\
w' &\mapsto w' + (\psi' \circ (\chi')^{-1})(w'), & w' &\mapsto \chi'(w'),
\end{align*}
for $w \in S$ and $w' \in S'$; in particular, since $\xi$ is an automorphism, so are $\psi$ and $\xi'$. But we have $\xi' \in C_\delta'$ by \ref{it:autos-block-triang}, whereas by construction $\psi$ is a composite of automorphisms of $R[\lambda^{-1}]$ of the form $\psi_{d,i,j}^{(z)}$, where $1 \leq d < \delta$, $1 \leq i \leq m_d$, $1 \leq j \leq m_\delta$ and $z \in Z[X]/(X^d)$, and therefore $\psi \in C_\delta'$ by \ref{it:autos-elem-offdiag}. This shows that $\xi \in C_\delta'$, and so \ref{it:autos-all} holds. \hfill\popQED

\medskip

As we have verified \ref{it:autos-all}, the claim is proven.
\end{claimproof}

We have thus shown that $C_\delta$ is generated by $\langle C_{\delta-1}, \GL_{m_\delta}(Z) \rangle$ together with a finite collection of elements and therefore, by induction on $\delta$, we have that $C = C_\delta$ is generated by a finite collection of elements together with copies of $\GL_{m_d}(Z)$ for $1 \leq d \leq \delta$. However, the groups $\GL_{m_d}(Z)$ are $S$-arithmetic subgroups of $\GL_n(\overline \QQ)$, and therefore finitely generated \cite[Proposition~6.10]{BorelSerre1976}. It follows that $C$ is finitely generated, as required.
\end{proof}

\setcounter{thmx}{2}
\begin{examplex}\label{ex:intro}
    Let $G=\langle a,b,t \mid [a,b],\ a^t=a^2b,\ b^t=b^2\rangle.$  Note that this is the Leary--Minasyan group $\LM(A,L)$ where $A = \begin{pmatrix}
    2 & 1 \\
    0 & 2
\end{pmatrix}$ and $L=\ZZ^2$.  The group $G$ is metabelian but not polycyclic and is of type $\mathsf{F}$ since it is an ascending HNN extension of $\ZZ^2$.  Moreover, $\Out(G)\cong \ZZ[1/2]\times\ZZ/2\ZZ$, which we now verify.

By \Cref{prop:metab-Out}, we only need to compute $C/\langle A\rangle$ where $C=C_{\GL(K)}(A)$ and  $K=\ZZ[1/2]^2$.  But this is straightforward in such a low dimension.  We have 
\[
\begin{pmatrix} 2 & 1\\ 0& 2\end{pmatrix}
=\begin{pmatrix} a & b \\ c & d \end{pmatrix}
\begin{pmatrix} 2 & 1\\ 0& 2\end{pmatrix}
\begin{pmatrix} a &b\\ c & d \end{pmatrix}^{-1}
=\frac{1}{ad-bc}\begin{pmatrix}2ad -2bc -ac & a^2\\ -c^2 & ac+2ad-2bc\end{pmatrix}
\]
Solving the four equations gives $a=d$ and $c=0$.  Thus, $C$ is given by the set of matrices
\[
\left\{ \begin{pmatrix}
    x & y \\ 0 & x
\end{pmatrix} \:\middle|\: x\in\ZZ[1/2]^\times, y\in \ZZ[1/2]\right\}
\]
which is isomorphic to $\ZZ \times \ZZ/2\ZZ \times \ZZ[1/2]$.  Now, $C/\langle A\rangle \cong \ZZ/2\ZZ \times \ZZ[1/2]$. Observe, the matrix $I-A$ is invertible over $\ZZ$.  Thus, we see that the twist group $D=K/(I-A)K$ (in the notation of \Cref{prop:metab-Out}) is trivial.  Hence, by \Cref{prop:metab-Out} we conclude that $\Out(G)\cong \ZZ[1/2] \times \ZZ/2\ZZ$ as required. \hfill$\blackdiamond$
\end{examplex}

Finally, we record \Cref{thmx:not_metabelian} and \Cref{thmx:metabelian} specialised to the case of HNN extensions of $\ZZ$.   This recovers the well-known result below which was previously proven over the papers \cite{Collins1978,CollinsLevin1983,GilbertHowieMetaftsisRaptis2000,Clay09}

\begin{corollary}\label{cor.BS}
    Let $1\leq r\leq |s|$ be non-zero integers.  The following conclusions hold:
    \begin{enumerate}
        \item \label{it:BS23} if $r\not\mid s$, then $\Out(\BS(r,s)) \cong D_{\mid r-s \mid}$;
        \item \label{it:BS2-2} if $r=-s\neq 1$, then $\Out(\BS(r,-r)) \cong D_{2r} \times \ZZ/2\ZZ$;
        \item \label{it:BS22} if $r=s \neq 1$, then $\Out(\BS(r,r)) \cong D_\infty \times \ZZ/2\ZZ$;
        \item \label{it:BS24} if $r \neq 1$, $r\neq |s|$ and $r\mid s$ then $\Out(\BS(r,s))$ is not finitely generated;
        \item \label{it:BS12} if $r=1$ and $|s| \neq 1$ then $\Out(\BS(1,s)) \cong \ZZ/(s-1)\ZZ \rtimes (\ZZ^k \times \ZZ/2\ZZ)/\ZZ$, where $k$ is the number of distinct prime factors of $s$, and the copy of $\ZZ$ is ``skew diagonally'' embedded recording the multiplicity of each factor;
        \item \label{it:BS11} if $r=s=1$, $\Out(\BS(1,1)) \cong \GL_2(\ZZ)$, and if $r=-s=1$ then $\Out(\BS(1,-1)) \cong \ZZ/2\ZZ \times \ZZ/2\ZZ$.
    \end{enumerate}
\end{corollary}

\begin{proof}
Throughout we have $A = s/r$, $L=r\ZZ$ and $AL = s\ZZ$. Parts \ref{it:BS23}--\ref{it:BS24} correspond to non-metabelian groups, and so we apply material from Sections~\ref{sec:non_metab_intro} and~\ref{sec:deform}. In each case $\calc^+ = GL_1(\ZZ) \cong \ZZ/2\ZZ$, acting as inversion on the vertex group generator $a$. We therefore discuss the kernel of the appropriate semidirect product structure, and determine if $\calc$ strictly contains $\calc^+$. \begin{enumerate}
    \item Here $L$ and $AL$ are incomparable. We have $(I-A)L =((r-s)/r)r\ZZ = (r-s)\ZZ$, giving $\overline{\calb} \cong \ZZ/(r-s)\ZZ$. We have $\calc = \calc^+$. This gives $\Out(G) \cong \ZZ/(r-s)\ZZ \rtimes \ZZ/2\ZZ \cong D_{\mid r-s \mid}$ (the last isomorphism since the action is non-trivial).
    \item We have $A=-I$ and $L=AL$, and so $\calc^+ < \calc \cong \ZZ/2\ZZ \times \ZZ/2\ZZ$. For the kernel, note that $(I-A)L =2r\ZZ$. This gives that $\Out(G) \cong \ZZ/2r\ZZ \rtimes (\ZZ/2\ZZ \times \ZZ/2\ZZ)$. Inspecting the action, both generators of the right hand group (one inverts $a$, the other $t$) invert the generator $t \mapsto at$ of the left hand group (working up to inner automorphisms), and so this is isomorphic to $D_{2r} \times \ZZ/2\ZZ$.
    \item Here we have $A=I$ and $L=AL$, so $\calc^+ < \calc \cong \ZZ/2\ZZ \times \ZZ/2\ZZ$. For the kernel, we have $(I-A)L= 0$, giving $\overline{\calb}=\ZZ$. As in the previous item, $\Out(G) \cong \ZZ \rtimes (\ZZ/2\ZZ \times \ZZ/2\ZZ)$, with both generators of the right hand group acting as inversion on the kernel. So $\Out(G)\cong D_\infty \times \ZZ/2\ZZ$.
    \item Here we have $AL < L$, and so we are in the situation analysed in Section~\ref{sec:deform}. The groups $L_n/(I-A)L$ are finite cyclic groups 
    $\ZZ/(s/r)^n(s/r-1)\ZZ$
    with strict containments.
\end{enumerate}

The last two cases are metabelian, with the final one being the two polycyclic examples.

\begin{enumerate}[resume]
    \item First note that $D = K/(I-A)K \cong \ZZ[1/s]/(1-s)\ZZ[1/s] \cong \ZZ/(1-s)\ZZ$. Now we study the subgroup of $\GL_1(\QQ) = \QQ^\times$ preserving $\ZZ[1/s]$. This is generated by $-1$ and $1/p_i$ with $p_i$ ranging over the distinct prime factors of $s$, which gives $C \cong \ZZ^k \times \ZZ/2\ZZ$, where $k$ is the number of distinct prime factors of $s$. We then quotient by $\langle 1/s \rangle$, giving the claim.
    \item This is classical. (For $\BS(1,-1)$ see e.g. \cite[Lemma 2.1]{GoncalvesWong2009}; for $\BS(1,1)$ this is nothing but the isomorphism $\Out(\ZZ^2)\cong \GL_2(\ZZ)$.) 
\end{enumerate}

\end{proof}

\bibliographystyle{halpha}
\bibliography{ref}

\end{document}